\documentclass[10pt,oneside]{amsart}   	

\usepackage{geometry}                		
\usepackage{geometry,graphicx,amssymb,amsmath,amsfonts,bm,tcolorbox,enumitem,amsbsy}
\usepackage{caption,booktabs,array,multirow,cellspace}
\usepackage[all]{xy}
\usepackage[hidelinks]{hyperref}

\usepackage{tikz}
\usetikzlibrary{arrows,shapes}
\usetikzlibrary{arrows, decorations.markings,fit}
\usetikzlibrary{calc,3d}
\usepackage{tikz-3dplot}
\usepackage{tkz-euclide}
\usepackage{epstopdf}
\usepackage{pgfplots}
\usepgfplotslibrary{colormaps,patchplots}
\usetikzlibrary{angles}
\usetikzlibrary{arrows,shapes}
\usetikzlibrary{arrows, decorations.markings,fit}
\usetikzlibrary{calc,3d}
\usepgfplotslibrary{groupplots}
\usetikzlibrary{matrix, positioning}

\tikzset{
    >=stealth',
    punkt/.style={
           rectangle,
           rounded corners,
           draw=black, very thick,
           text width=6.5em,
           minimum height=2em,
           text centered},
    pil/.style={
           ->,
           thick,
           shorten <=2pt,
           shorten >=2pt,}
}

\usepackage{subcaption}
\numberwithin{equation}{section}

\allowdisplaybreaks[3]

\newtheorem{theorem}{Theorem}[section]
\newtheorem{lemma}[theorem]{Lemma}
\newtheorem{assumption}[theorem]{Assumption}
\newtheorem{corollary}[theorem]{Corollary}
\newtheorem{proposition}[theorem]{Proposition}
\theoremstyle{definition}

\newtheorem{remark}[theorem]{Remark}

\definecolor{myred}{RGB}{195,0,0}
\definecolor{myblue}{RGB}{0,90,170}
\definecolor{mygreen}{RGB}{0,140,0}
\definecolor{myreddark}{RGB}{140,0,0}
\definecolor{myredlight}{RGB}{255,100,100}
\definecolor{mybluedark}{RGB}{0,50,120}
\definecolor{mybluelight}{RGB}{100,180,255}
\definecolor{mygreendark}{RGB}{0,90,0}
\definecolor{mygreenlight}{RGB}{100,220,100}
\definecolor{mypurple}{RGB}{120,0,120}
\definecolor{myorange}{RGB}{230,120,0}
\definecolor{mycyan}{RGB}{0,150,150}
\definecolor{myyellow}{RGB}{210,180,0}
\definecolor{mybrown}{RGB}{150,100,50}
\definecolor{mygray}{RGB}{120,120,120}

\newcommand{\R}{\mathbb{R}}

\newcommand{\bbR}{\mathbb{R}}
\newcommand{\bbN}{\mathbb{N}}
\newcommand{\bbS}{\mathbb{S}}

\newcommand{\dive}{{\ensuremath\mathop{\mathrm{div}}}}
\newcommand{\curl}{{\ensuremath\mathop{\mathrm{curl}}}}

\newcommand{\bc}{\bm c}
\newcommand{\bfd}{\bld{d}}

\newcommand{\bfe}{\bld{e}}

\newcommand{\bff}{\bld{f}}
\newcommand{\bgg}{\bld{g}}

\newcommand{\bn}{\bld{n}}
\newcommand{\bp}{\bld{p}}

\newcommand{\bu}{\bld{u}}

\newcommand{\bv}{\bld{v}}

\newcommand{\bx}{\bld{x}}
\newcommand{\by}{\bld{y}}
\newcommand{\bz}{\bld{z}}

\newcommand{\bH}{\bld{H}}

\newcommand{\bfI}{\mathbf{I}}
\newcommand{\bfK}{\mathbf{K}}
\newcommand{\bL}{\bld{L}}

\newcommand{\bfQ}{\mathbf{Q}}
\newcommand{\bfR}{\mathbf{R}}

\newcommand{\bP}{\bld{P}}

\newcommand{\bX}{\bm X}

\newcommand{\balpha}{\bm \alpha}

\newcommand{\bchi}{\bld{\chi}}

\newcommand{\bfeta}{\bm \eta}

\newcommand{\bxi}{\bld{\xi}}

\newcommand{\calA}{\mathcal{A}}
\newcommand{\calB}{\mathcal{B}}

\newcommand{\calD}{\mathcal{D}}

\newcommand{\calE}{\mathcal{E}}

\newcommand{\calF}{\mathcal{F}}

\newcommand{\calI}{\mathcal{I}}

\newcommand{\calN}{\mathcal{N}}

\newcommand{\calP}{\mathcal{P}}

\newcommand{\calQ}{\mathcal{Q}}
\newcommand{\calW}{\mathcal{W}}
\newcommand{\calS}{\mathcal{S}}

\newcommand{\calV}{\mathcal{V}}

\newcommand{\dist}{\operatorname{dist}}
\newcommand{\bld}[1]{\boldsymbol{#1}}

\title[Divergence-free interpolation]{Divergence-free interpolation of tangential vector fields via matrix-valued kernels}\thanks{The first author was supported in part by National Natural Science Foundation of China (No. 12571407), Basic Research Program of Jiangsu (No. BK20252037), and a Jiangsu Shuangchuang Team program (No. JSSCTD202449).}

\author[Z. Sun]{Zhengjie Sun}
\address{School of Mathematics and Statistics, Nanjing University of Science and Technology}
\email{zhengjiesun2020@gmail.com}

\author[L. Dong]{Lishuo Dong}
\address{School of Mathematics and Statistics, Nanjing University of Science and Technology}
\email{dlishuo@163.com}

\author[B. Huang]{Biao Huang}
\address{School of Mathematics and Statistics, Nanjing University of Science and Technology}
\email{biaohuangbest@163.com}


\begin{document}

\maketitle

\begin{abstract}
We develop and analyze a family of divergence-free kernel interpolation methods for tangential vector fields on the unit sphere. Starting from scalar radial kernels in Euclidean space, we construct tangent-valued, surface divergence-free matrix kernels without repeatedly applying surface differential operators. The construction separates the geometric enforcement of the divergence-free constraint from the choice of scalar generator and admits lower-order variants with reduced regularity requirements compared with classical potential-based methods. Using vector spherical harmonics, we derive explicit kernel representations and characterize their Fourier multipliers. We also introduce an inverse Laplace–Beltrami construction that preserves the multipliers of the underlying scalar zonal kernel. For interpolation at scattered nodes, we establish a lower bound for the smallest eigenvalue of the interpolation matrix and derive pointwise and Sobolev error estimates, including superconvergence for targets smoother than the native space. Numerical experiments corroborate the theoretical convergence and stability results, while additional examples on nonspherical surfaces illustrate the applicability of the kernel formula beyond the sphere.

\end{abstract}

\section{Introduction}
\label{sec:intro}
\thispagestyle{empty}

Tangential vector fields on surfaces arise in fluid flow on curved interfaces, geophysical modeling, and transport on manifolds \cite{dubois_1990SINUM_discrete,lederer_2020_divergence,stone_2010JFM_interfaces,swarztrauber_1981SINUM_approximation,zhang_2006ToG_vector}.  In many applications, these fields satisfy intrinsic differential constraints,
most notably vanishing surface divergence or surface curl \cite{freedman_1991AM_divergence,hermes_1991SIREV_nilpotent,wu_2021NM_provably}. Approximation methods that preserve both tangency and the relevant differential
constraint are therefore important for the analysis and simulation of such
fields.

In Euclidean domains, divergence-free and curl-free matrix-valued radial basis functions (RBFs) are well-established tools for approximating constrained vector fields \cite{Fuselier_2008Adv_improved,Fuselier_2008MCoM_sobolev,Lowitzsch_2005JAT_error,Lowitzsch_2005AdvCM_matrix,Narcowich_1994MCoM_generalized}. Applications include incompressible fluid flow and solenoidal magnetic fields \cite{Guzman_2014IMAJNA_conforming,Guzman_2014MCoM_conforming,Neilan_2021SINUM_divergence,schwarzacher_2025NM_stability,Soga_2024NM_mathematical,Zhao_2019SINUM_divergence}. 
A standard potential-based construction applies differential operators to a
sufficiently smooth scalar positive-definite radial kernel \cite{Dodu_2004NM_irrotational,Farrell_2017IMAJNA_multilevel,Lowitzsch_2005AdvCM_matrix,Narcowich_1994MCoM_generalized,Wendland_2009SINUM_divergence}. Specifically, let $\varphi:[0,\infty)\to\mathbb{R}$ be a univariate function. The associated divergence-free and curl-free matrix-valued kernels are given by
\begin{subequations}\label{eq:Euclid_potential}
    \begin{align}
        &\Phi_{\dive}(\bx,\by):=(-\Delta \bfI +\nabla\nabla^{\top})\varphi(\|\bx-\by\|), \label{eq:Euclid_div_op}\\
        &\Phi_{\curl}(\bx,\by)
    :=-\nabla\nabla^{\top}\varphi(\|\bx-\by\|), \label{eq:Euclid_curl_op}
    \end{align}
\end{subequations}
where $\bx,\by\in\bbR^3$ and $\bfI$ is the identity matrix. These kernels generate divergence-free
and curl-free interpolants in $\R^3$, respectively.

Writing $r=\|\bx-\by\|$, the curl-free kernel has the explicit representation
\begin{equation}\label{eq:psicurl_derivative_form}
    \Phi_{\curl}(\bx,\by)
    =
    -\frac{\varphi'(r)}{r}\bfI
    -
    \left(
        \frac{\varphi''(r)}{r^2}
        -
        \frac{\varphi'(r)}{r^3}
    \right)
    (\bx-\by)(\bx-\by)^\top .
\end{equation}
Motivated by this representation, Sun et al.~\cite{Sun_2026_error} considered
the more general isotropic form
\begin{equation}\label{eq:curlkernel_form}
    \Phi_{\curl}(\bx,\by)
    =
    \alpha(r)\bfI
    -
    \beta(r)(\bx-\by)(\bx-\by)^\top.
\end{equation}
The columns of \eqref{eq:curlkernel_form} are curl-free when the scalar
coefficients satisfy $\calD\alpha=\beta$. A useful family is obtained by setting
\[
\alpha(r)=\calI\beta(r), \quad \beta(r)=\calD^m\varphi(r),\quad m\in \bbN_0,
\]
where $\calD$ and $\calI$ are classical differential and integral operators \cite{Schaback_1996JCAM_operators,Wendland_2004book_scattered},
\[
(\calD f)(r)=-\frac{1}{r}f'(r),
\qquad
(\calI f)(r)=\int_r^\infty s f(s)\,\mathrm{d}s.
\]

Extending these constructions to surfaces is not immediate. Restricting an
ambient kernel to a surface $\calS\subset\mathbb{R}^3$ generally preserves
neither tangency nor the corresponding intrinsic differential constraint. In
particular, the restriction of an ambient divergence-free field need not be
surface divergence-free. Existing operator-based constructions overcome this
difficulty by combining an ambient curl-free kernel with surface-curl or
tangent-projection operators
\cite{Narcowich_2007JFAA_divergence,Fuselier_2009MCoM_error}.

Let $\calS$ be a smooth oriented surface with unit normal field $\bn$, and let
$\nabla_*$ and $\bL_*$ denote the surface gradient and surface curl,
respectively. The identity $\bL_*=\bn\times\nabla_*$ shows that the surface curl
maps scalar potentials to tangential vector fields. To express this operation
algebraically, let $\bX_{\bu}$ be the skew-symmetric matrix satisfying
$\bX_{\bu}\bv=\bu\times\bv$ for every $\bv\in\mathbb{R}^3$. For
$\bu=[u_1,u_2,u_3]^\top$,
\begin{equation}\label{eq:Xu_intro}
    \bX_{\bu}
    =
    \begin{bmatrix}
        0 & -u_3 & u_2 \\
        u_3 & 0 & -u_1 \\
        -u_2 & u_1 & 0
    \end{bmatrix}.
\end{equation}
Let $\bP_{\bn}:=\bfI-\bn\bn^\top$ denote the orthogonal projection onto the
tangent plane. The corresponding surface kernels are
\begin{subequations}\label{eq:surf_potential_intro}
\begin{align}
    \bfK_{\dive}(\bx,\by)
    &:=
    \bX_{\bn_{\bx}}
    \Phi_{\curl}(\bx,\by)
    \bX_{\bn_{\by}}^\top, \label{eq:surf_div_op_intro}\\
    \bfK_{\curl}(\bx,\by)
    &:=
    \bP_{\bn_{\bx}}
    \Phi_{\curl}(\bx,\by)
    \bP_{\bn_{\by}}^\top. \label{eq:surf_curl_op_intro}
\end{align}
\end{subequations}
If the ambient kernel has curl-free columns, each translate of
$\bfK_{\dive}$ is tangential and surface divergence-free, whereas each
translate of $\bfK_{\curl}$ is tangential and surface curl-free
\cite[Theorem~1]{Narcowich_2007JFAA_divergence}. If the ambient matrix-valued
kernel is positive definite, these transformations inherit positive
definiteness on tangent data under the corresponding nondegeneracy conditions.

The purpose of this paper is to extend the Euclidean isotropic representation of matrix-valued RBF kernels to tangential vector fields on embedded surfaces. Rather than beginning with repeated derivatives of a scalar potential in \eqref{eq:Euclid_potential}, we start from the isotropic matrix form \eqref{eq:curlkernel_form} and incorporate the surface geometry through the cross-product and tangent-projection operators $\bX_{\bn}$ and $\bP_{\bn}$. This separates geometric constraint enforcement from the choice of scalar coefficient functions $\alpha$ and $\beta$ and yields a unified construction of surface divergence-free and curl-free kernels.

On the unit sphere, we derive vector spherical harmonic expansions of the
divergence-free kernels and characterize their Fourier multipliers. Positive
multipliers yield positive definiteness on tangent data, while appropriate
two-sided decay bounds identify the associated native spaces, up to norm
equivalence, with Sobolev spaces of divergence-free tangential fields. Under
appropriate assumptions on the scalar multipliers, the $m=1$ construction
preserves the scalar Sobolev order, whereas the classical $m=2$
potential-based construction incurs a loss of one order. We also introduce an
inverse Laplace--Beltrami construction for which the vector multipliers agree
with those of the underlying scalar zonal kernel.
For scattered data interpolation, we formulate the linear system in local
orthonormal tangent frames and use Fourier localization to establish a lower
bound for the smallest eigenvalue of the interpolation matrix. We then derive
pointwise and Sobolev error estimates, including a Hilbert-scale
superconvergence result for targets smoother than functions in the native
space. Numerical experiments on $\bbS^2$ examine the theoretical convergence and
stability behavior. Additional experiments on a torus, a red blood cell
surface, and a bumpy sphere serve as illustrations beyond the scope
of the spherical theory.

The remainder of the paper is organized as follows. Section~2 reviews scalar
and vector spherical harmonics, matrix-valued kernels, and native spaces.
Section~3 develops the surface construction, specializes it to $\bbS^2$, and
analyzes the resulting Fourier multipliers. Sections~4 and~5 formulate and
analyze divergence-free interpolation on $\bbS^2$, including stability and
error estimates. Section~6 presents the numerical experiments, and Section~7
concludes the paper.

\section{Preliminaries}
\label{sec:prelim}

\subsection{Spherical harmonics}
Let $\Delta_*$ denote the Laplace--Beltrami operator on the unit sphere $\bbS^2$.
For each $\ell\ge0$ and $1\le k\le 2\ell+1$, let $Y_{\ell,k}$ be a spherical harmonic of degree $\ell$, that is, an eigenfunction of $-\Delta_*$ with eigenvalue
\[
\lambda_\ell=\ell(\ell+1).
\]
The family $\{Y_{\ell,k}\}$ is an orthonormal basis of $L_2(\bbS^2)$; see, for example, \cite{Dai_2013book_approximation,Freeden_2008book_spherical,Mueller_1966book_spherical}.
Accordingly, every $f\in L_2(\bbS^2)$ has the Fourier expansion
\begin{equation*}
    f
    =
    \sum_{\ell=0}^{\infty}\sum_{k=1}^{2\ell+1}
    \widehat f_{\ell,k}Y_{\ell,k},
    \quad
    \widehat f_{\ell,k}:=\langle f,Y_{\ell,k}\rangle.
\end{equation*}
For $\sigma\ge0$, we define the Sobolev space $H^\sigma(\bbS^2)$ by the norm
\begin{equation}\label{eq:sobolev_norm}
    \|f\|_{H^\sigma(\bbS^2)}^2
    :=
    \sum_{\ell=0}^{\infty}\sum_{k=1}^{2\ell+1}
    (1+\lambda_\ell)^\sigma
    |\widehat f_{\ell,k}|^2.
\end{equation}

We next recall the corresponding $\bL_2$-theory for tangential vector fields on $\bbS^2$; see \cite{Freeden_2008book_spherical,Fuselier_2009MCoM_error}.
A vector field $\bff:\bbS^2\to\bbR^3$ is said to be tangential if
\[
\bx^\top \bff(\bx)=0,
\quad \bx\in\bbS^2.
\]
Throughout, we identify each point $\bx\in\bbS^2$ with the outward unit normal at $\bx$.
We denote the space of square-integrable tangential vector fields by $\bL_2(\bbS^2)$ and equip it with the inner product
\begin{equation*}
    \langle \bff,\bgg\rangle
    :=
    \int_{\bbS^2}
    \bff(\bx)^\top \bgg(\bx)\,d\mu(\bx).
\end{equation*}

For every scalar function $g$ on $\bbS^2$, both $\nabla_* g$ and $\bL_* g$ are tangential;
moreover, $\nabla_* g$ is curl-free, whereas $\bL_* g$ is divergence-free.
For $\ell\ge1$ and $1\le k\le2\ell+1$, define the vector spherical harmonics by
\begin{equation}\label{eq:def_ylk}
    \mathbf y_{\ell,k}
    :=
    \frac{\bL_*Y_{\ell,k}}{\sqrt{\ell(\ell+1)}},
    \quad
    \mathbf z_{\ell,k}
    :=
    \frac{\nabla_*Y_{\ell,k}}{\sqrt{\ell(\ell+1)}}.
\end{equation}
With the normalization in \eqref{eq:def_ylk}, the family $\{\mathbf y_{\ell,k}\}$ forms an orthonormal basis of the divergence-free subspace of $\bL_2(\bbS^2)$, whereas $\{\mathbf z_{\ell,k}\}$ is an orthonormal basis of the curl-free subspace.
Because $\bbS^2$ has no nontrivial harmonic tangential vector fields, the union of these two families is a
complete orthonormal basis of $\bL_2(\bbS^2)$.
Thus, every $\bff\in\bL_2(\bbS^2)$ has the orthogonal expansion
\begin{equation*}
    \bff
    =
    \sum_{\ell=1}^{\infty}\sum_{k=1}^{2\ell+1}
    \left(
    \widehat{\bff}_{\ell,k}^{\rm div}\mathbf y_{\ell,k}
    +
    \widehat{\bff}_{\ell,k}^{\rm curl}\mathbf z_{\ell,k}
    \right),
\end{equation*}
where
\[
\widehat{\bff}_{\ell,k}^{\rm div}
:=
\langle \bff,\mathbf y_{\ell,k}\rangle,
\quad
\widehat{\bff}_{\ell,k}^{\rm curl}
:=
\langle \bff,\mathbf z_{\ell,k}\rangle.
\]
For $\sigma\ge0$, define the Sobolev space $\bH^\sigma(\bbS^2)$ of tangential vector fields by the norm
\begin{equation*}
    \|\bff\|_{\bH^\sigma(\bbS^2)}^2
    :=
    \sum_{\ell=1}^{\infty}\sum_{k=1}^{2\ell+1}
    (1+\lambda_\ell)^\sigma
    \left(
    |\widehat{\bff}_{\ell,k}^{\rm div}|^2
    +
    |\widehat{\bff}_{\ell,k}^{\rm curl}|^2
    \right).
\end{equation*}
We write $\bH_{\rm div}^\sigma(\bbS^2)$ for the closed subspace of  divergence-free tangential vector fields in $\bH^\sigma(\bbS^2)$ and $\bH_{\rm curl}^\sigma(\bbS^2)$ for the corresponding curl-free subspace.

The scalar spherical harmonics satisfy the classical addition formula
\cite{Mueller_1966book_spherical},
\begin{equation}\label{eq:addition_formula}
    \sum_{k=1}^{2\ell+1}
    Y_{\ell,k}(\bx)Y_{\ell,k}(\by)
    =
    \frac{2\ell+1}{4\pi}P_\ell(\bx\cdot \by),
    \quad \bx,\by\in\bbS^2,
\end{equation}
where $P_\ell$ denotes the Legendre polynomial of degree $\ell$, normalized by $P_\ell(1)=1$.

The corresponding addition formulas for vector spherical harmonics were established in \cite{Sun_2026_vector}. We record them here because they are central to the Fourier analysis and error estimates
below.

\begin{lemma}\label{lem:vec_addition}
	Let $\{\mathbf{y}_{\ell,k}\}$ and $\{\mathbf{z}_{\ell,k}\}$ be the orthonormal bases of divergence-free and curl-free vector spherical harmonics of degree $\ell \ge 1$, respectively, defined in \eqref{eq:def_ylk}.  For $\bx,\by\in\bbS^2$, define the matrix-valued kernels
	\begin{equation*}
		\mathcal{S}_\ell(\bx,\by) := \sum_{k=1}^{2\ell+1} \mathbf{y}_{\ell,k}(\bx) \mathbf{y}_{\ell,k}(\by)^{\top}, \quad
		\mathcal{T}_\ell(\bx,\by) := \sum_{k=1}^{2\ell+1} \mathbf{z}_{\ell,k}(\bx) \mathbf{z}_{\ell,k}(\by)^{\top}.
	\end{equation*}
	 Then the following addition formulas hold:
	\begin{align*}
		\mathcal{S}_\ell(\bx,\by) &= \frac{2\ell+1}{4\pi\lambda_\ell} \left[ P_\ell''(t) \mathbf{Q}(\bx,\by) + P_\ell'(t) \mathbf{R}(\bx,\by) \right], \\
		\mathcal{T}_\ell(\bx,\by) &= \frac{2\ell+1}{4\pi\lambda_\ell} \left[ P_\ell''(t) \mathbf{V}(\bx,\by) + P_\ell'(t) \mathbf{W}(\bx,\by) \right],
	\end{align*}
	where $t := \bx \cdot \by$, and the matrix-valued functions $\mathbf{Q, R, V, W}$ are defined by
	\begin{align*}
		\mathbf{Q}(\bx,\by) &:= -(\bx \times \by)(\bx \times \by)^{\top}, & 
		\mathbf{R}(\bx,\by) &:= t\mathbf{I} - \by \bx^{\top}, \\
		\mathbf{V}(\bx,\by) &:= (\by - t \bx)(\bx - t \by)^\top, &
		\mathbf{W}(\bx,\by) &:= (\mathbf{I} - \bx \bx^\top)(\mathbf{I} - \by \by^\top).
	\end{align*}
\end{lemma}

\subsection{Kernels and native spaces}

We next introduce the scalar and matrix-valued kernels used throughout the paper. These kernels determine both the approximation spaces and the associated native-space norms.
On the sphere, the natural counterpart of a Euclidean radial kernel is a zonal kernel, whose value depends only on the inner product of its arguments.

A kernel
$\Psi:\bbS^2\times\bbS^2\to\bbR$
is called \emph{zonal} if there exists a function
$\psi:[-1,1]\to\bbR$ such that
\[
\Psi(\bx,\by)=\psi(\bx\cdot\by),
\quad \bx,\by\in\bbS^2.
\]
When no ambiguity can arise, we identify $\Psi$ with its scalar kernel $\psi$.
The Fourier--Legendre expansion of $\psi$ is
\begin{equation}\label{eq:Four-LegendreSeries}
    \psi(\bx\cdot\by)
    =
    \sum_{\ell=0}^{\infty}
    \frac{2\ell+1}{4\pi}
    \widehat\psi(\ell)
    P_\ell(\bx\cdot\by),
    \quad
    \widehat\psi(\ell)
    :=
    2\pi\int_{-1}^{1}\psi(t)P_\ell(t)\,\mathrm{d}t.
\end{equation}
By the addition formula \eqref{eq:addition_formula}, this expansion can be rewritten equivalently as
\begin{equation}\label{eq:zonal_kernel_SH}
    \psi(\bx\cdot\by)
    =
    \sum_{\ell=0}^{\infty}
    \widehat\psi(\ell)
    \sum_{k=1}^{2\ell+1}
    Y_{\ell,k}(\bx)Y_{\ell,k}(\by).
\end{equation}
These identities hold in the appropriate $L_2$-sense, and pointwise whenever the smoothness of $\psi$ permits.
If $\widehat\psi(\ell)>0$ for all $\ell\ge0$, then $\psi$ defines a positive definite kernel on $\bbS^2$.
Its native space is the reproducing kernel Hilbert space
\begin{equation*}
    \mathcal N_\psi
    :=
    \left\{
    f\in L_2(\bbS^2):
    \|f\|_{\mathcal N_\psi}<\infty
    \right\}\quad \text{with}\quad \|f\|_{\mathcal N_\psi}^2
    =
    \sum_{\ell=0}^{\infty}\sum_{k=1}^{2\ell+1}
    \frac{|\widehat f_{\ell,k}|^2}{\widehat\psi(\ell)}.
\end{equation*}

We next consider matrix-valued kernels.
Let $\bfK:\bbS^2\times\bbS^2\to\bbR^{3\times3}$ be symmetric.
We say that $\bfK$ is \emph{positive semidefinite} if, for all finite set
$X=\{\bx_j\}_{j=1}^N\subset\bbS^2$ and all choice of vectors
$\balpha_j\in T_{\bx_j}\bbS^2$,
\[
\sum_{j,k=1}^N
\balpha_j^\top \bfK(\bx_j,\bx_k)\balpha_k
\ge 0.
\]
It is \emph{positive definite} if equality can hold only when $\balpha_j=\mathbf 0$ for $j=1,\dots,N$.

We focus on matrix-valued kernels associated with tangential,
divergence-free fields.  Suppose that a divergence-free kernel $\bfK_{\dive}$ has the Fourier expansion
\begin{equation}\label{eq:Fourier_expan_Matrix}
    \bfK_{\dive}(\bx,\by)
    =
    \sum_{\ell=1}^{\infty}
    \kappa_\ell
    \sum_{k=1}^{2\ell+1}
    \mathbf y_{\ell,k}(\bx)\mathbf y_{\ell,k}(\by)^\top,
\end{equation}
with coefficients $\kappa_\ell>0$.
Then the associated native space can be characterized as
\begin{equation*}
    \mathcal N_{\bfK_{\dive}}
    =
    \left\{
    \bff\in \bH_{\rm div}^{0}(\bbS^2):
    \sum_{\ell=1}^{\infty}\sum_{k=1}^{2\ell+1}
    \frac{|\widehat{\bff}_{\ell,k}^{\rm div}|^2}{\kappa_\ell}
    <\infty
    \right\}.
\end{equation*}
Analogous descriptions hold for the native spaces generated by a curl-free kernel $\bfK_{\curl}$, and by the full kernel
$\bfK=\bfK_{\dive}+\bfK_{\curl}$. 
The corresponding native-space formulas have formal
curl-free and combined-field analogues. The spectral and approximation results proved in this paper, however, concern only the divergence-free space.

Moreover, if there exist positive constants $c_1$ and $c_2$, independent of $\ell$, such that
\begin{equation}\label{eq:decay_condition_kernel}
    c_1(1+\lambda_\ell)^{-\sigma}\le \kappa_\ell\le c_2(1+\lambda_\ell)^{-\sigma},
\quad \ell\ge 1.
\end{equation}
then $\mathcal N_{\bfK_{\dive}}$ is norm-equivalent to $\bH_{\rm div}^\sigma(\bbS^2)$. Here and throughout, the notation $\kappa_\ell\asymp (1+\lambda_\ell)^{-\sigma}$
denotes this two-sided bound.

\section{General surface divergence-free kernels}

The main objective of this section is to develop a general framework for constructing surface divergence-free kernels. We begin with a general curl-free RBF construction in $\bbR^3$, and then apply the surface divergence-free operator to obtain matrix-valued kernels on embedded surfaces. We next specialize the resulting formula to the unit sphere, where the geometry allows a particularly simple representation. Finally, we analyze the associated spherical Fourier multipliers and use them to characterize positive definiteness and the corresponding native spaces.

\subsection{Construction of divergence-free kernels}

Assume that the coefficient functions in \eqref{eq:curlkernel_form} are chosen
so that the columns of the ambient kernel $\Phi_{\curl}$ are curl-free.
Substituting \eqref{eq:curlkernel_form} into
\eqref{eq:surf_div_op_intro} and using
\begin{equation*}
    \bX_{\bu}\bX_{\bv}^\top
    =
    (\bu\cdot\bv)\bfI-\bv\bu^\top,
    \quad
    \bX_{\bn_{\bx}}(\bx-\by)
    =
    \bn_{\bx}\times(\bx-\by),
\end{equation*}
we arrive at the explicit surface divergence-free kernel
\begin{equation}\label{eq:generalkernel_form}
    \bfK_\dive(\bx,\by)
    =
    \alpha(r)
    \bigl((\bn_{\bx}\cdot\bn_{\by})\bfI
    -
    \bn_{\by}\bn_{\bx}^\top \bigr)
    -
    \beta(r)
    \bigl(\bn_{\bx}\times(\bx-\by)\bigr)
    \bigl(\bn_{\by}\times(\bx-\by)\bigr)^\top,
\end{equation}
where $r=\|\bx-\by\|_2$. This construction applies to any smooth oriented embedded
surface.

For the unit sphere $\calS=\bbS^2$, the outward normal is
$\bn_{\bx}=\bx$. In this case,
\begin{equation*}
    \bigl(\bn_{\bx}\times(\bx-\by)\bigr)
    \bigl(\bn_{\by}\times(\bx-\by)\bigr)^\top
    =
    (\bx\times\by)(\bx\times\by)^\top.
\end{equation*}
Using the matrices $\bfQ$ and $\bfR$ from
Lemma~\ref{lem:vec_addition}, we obtain
\begin{equation}\label{eq:Sph_kernel}
    \bfK_\dive(\bx,\by)
    =
    \alpha(r)\,\bfR(\bx,\by)
    +
    \beta(r)\,\bfQ(\bx,\by),
\end{equation}
where the scalar radial functions are given by
\[
\beta(r)=\calD^m\varphi(r),
\quad
\alpha(r)=\calI\beta(r).
\]

On the unit sphere, we set $t:=\bx\cdot\by$,
$r^2=2-2t$ and let
$\phi(t):=\varphi\bigl(\sqrt{2-2t}\bigr)$.
Derivatives of $\varphi$ are taken with respect to $r$, whereas derivatives of
$\phi$ are taken with respect to $t$. Since
$\mathrm{d}t/\mathrm{d}r=-r$, the chain rule gives
\[
\calD\varphi(r)
=
-\frac{1}{r}\frac{\mathrm{d}}{\mathrm{d}r}\varphi(r)
=
\phi'(t).
\]
This identity is understood for $r>0$ and extended to $r=0$ by continuity
whenever the required radial derivatives exist.
Consequently, for every integer $m\ge1$,
\[
\beta(r)=\phi^{(m)}(t),
\quad
\alpha(r)=\phi^{(m-1)}(t),
\]
and \eqref{eq:Sph_kernel} becomes
\begin{equation}\label{eq:Matrix_kernel_kappa}
	\bfK_\dive(\bx,\by)
	=
	\phi^{(m-1)}(t)\,\bfR(\bx,\by)
	+
	\phi^{(m)}(t)\,\bfQ(\bx,\by).
\end{equation}
A central feature of this construction is the integer parameter $m$, which determines the order of differentiation applied to the underlying scalar zonal kernel. For $m=2$, \eqref{eq:Matrix_kernel_kappa} recovers the standard spherical surface divergence-free kernel; see, for example, \eqref{eq:psicurl_derivative_form} and \cite{Fuselier_2009MCoM_error,Fuselier_2009SINUM_stability}. The present framework, however, also admits the lower-order cases $m=1$ and $m=0$. For $m=0$, we set $\beta=\varphi(r)$ and $\alpha=\calI\beta(r)$. As we show below, these cases still produce positive definite divergence-free kernels while requiring fewer derivatives of the original scalar zonal kernel.

\subsection{Fourier characterization of divergence-free kernels}
\label{sec:FL-matrix-valued}
For an isotropic divergence-free kernel on $\bbS^2$, positivity is determined
by its vector spherical harmonic multipliers, while their asymptotic decay
determines the associated native space. The following theorem computes these
multipliers from the zonal generator.

\begin{theorem}
\label{thm:divergence-free-fourier-expansion}
Let $\phi\in C^m[-1,1]$, where $m\geq 1$ is an integer, and define
\begin{equation}
\label{eq:Legendre_expan_varphi}
    a_{\ell}^{(m-1)}
    =
    2\pi\int_{-1}^1 \phi^{(m-1)}(t)P_\ell(t)\,\mathrm{d}t .
\end{equation}
Suppose that $\sum_{\ell=1}^\infty \ell^2
\big|a_{\ell-1}^{(m-1)}-a_{\ell+1}^{(m-1)}\big|<\infty$.
Let $\bfK_\dive$ be the divergence-free matrix-valued kernel defined by
\eqref{eq:Matrix_kernel_kappa}. Then $\bfK_\dive$ admits the vector
spherical harmonic expansion \eqref{eq:Fourier_expan_Matrix}, with multipliers
\begin{equation}
\label{eq:Coeff_Kdiv_general}
 \kappa_\ell=\frac{\lambda_\ell}{2\ell+1}
    \Big(a_{\ell-1}^{(m-1)}-a_{\ell+1}^{(m-1)}\Big),\quad \ell\geq 1.
\end{equation}
If $\kappa_\ell>0$ for $\ell\ge1$, then $\bfK_\dive$ is
positive definite on tangent data at pairwise distinct nodes. 
\end{theorem}

\begin{proof}
First, we recall the standard identities for Legendre polynomials:
\[
    P_0(t)=P_1'(t),
    \quad
    (2\ell+1)P_\ell(t)=P_{\ell+1}'(t)-P_{\ell-1}'(t),
    \quad \ell\ge 1.
\]
Applying these identities to the Legendre expansion of
$\phi^{(m-1)}$ gives
\begin{equation}
\label{eq:bell}
\begin{split}
    \phi^{(m-1)}(t)
    &= \sum_{\ell=0}^{\infty}\frac{2\ell+1}{4\pi}a_\ell^{(m-1)}P_\ell(t)\\
    &=
    \frac{1}{4\pi}a_0^{(m-1)}P_1'(t)
    +
    \sum_{\ell=1}^{\infty}
    \frac{a_\ell^{(m-1)}}{4\pi}
    \bigl(P_{\ell+1}'(t)-P_{\ell-1}'(t)\bigr) \notag \\
    &=
    \sum_{\ell=1}^{\infty}
    \frac{1}{4\pi}
    \Big(
        a_{\ell-1}^{(m-1)}-a_{\ell+1}^{(m-1)}
    \Big)P_\ell'(t) \notag.
    \end{split}
\end{equation}
Since
$\|P_\ell'(t)\|_{L_\infty}\le C\ell^2$,
the Legendre series for $\phi^{(m-1)}(t)$ converge absolutely and uniformly. Differentiating \eqref{eq:bell} in the distributional sense on $(-1,1)$
yields
$$\phi^{(m)}(t)
    =
    \sum_{\ell=1}^{\infty} b_{\ell}^{(m-1)}P_\ell''(t),\quad b_{\ell}^{(m-1)}=\frac{1}{4\pi}
    \Big(
        a_{\ell-1}^{(m-1)}-a_{\ell+1}^{(m-1)}
    \Big).$$
Although this series need not converge uniformly by itself, the matrix
combination does. Indeed, Lemma~\ref{lem:vec_addition} and the
vector spherical harmonic addition formula imply
\[
\left\|
P_\ell'(t)\bfR(\bx,\by)
+
P_\ell''(t)\bfQ(\bx,\by)
\right\|
\le C\lambda_\ell
\le C\ell^2.
\]
The assumed summability therefore gives absolute and uniform convergence of
the matrix series. Its uniform limit agrees distributionally, and hence
pointwise, with
$\phi^{(m-1)}(t)\bfR+\phi^{(m)}(t)\bfQ$. Consequently, we have
\begin{equation}\label{eq:Kdiv_expan_RQ}
	\begin{aligned}
		\bfK_{\dive}(\bx,\by)
		&=
		\sum_{\ell=1}^{\infty}
		b_\ell^{(m-1)}
		\bigl(
		P_\ell'(t)\bfR
		+
		P_\ell''(t)\bfQ
		\bigr)\\
		&=
		\sum_{\ell=1}^{\infty}
		\frac{4\pi\lambda_\ell}{2\ell+1}
		b_\ell^{(m-1)}
		\sum_{k=1}^{2\ell+1}
		\mathbf y_{\ell,k}(\bx)
		\mathbf y_{\ell,k}(\by)^\top.
	\end{aligned}
\end{equation}
Equation \eqref{eq:Coeff_Kdiv_general} follows immediately.

It remains to verify positive definiteness. Let
$\bx_1,\ldots,\bx_J$ be pairwise distinct nodes and let
$\boldsymbol\eta_j\in T_{\bx_j}\bbS^2$. The quadratic form associated with
$\bfK_\dive$ is
\[
\sum_{i,j=1}^{J}
\boldsymbol\eta_i^\top
\bfK_\dive(\bx_i,\bx_j)
\boldsymbol\eta_j
=
\sum_{\ell=1}^{\infty}
\kappa_\ell
\sum_{k=1}^{2\ell+1}
\Big|
\sum_{j=1}^{J}
\boldsymbol\eta_j^\top
\mathbf y_{\ell,k}(\bx_j)
\Big|^2.
\]
If this expression vanishes, the positivity of the multipliers implies
\[
\sum_{j=1}^{J}
\boldsymbol\eta_j^\top
\mathbf y_{\ell,k}(\bx_j)
=0
\qquad
\text{for all }\ell,k.
\]
This gives $\boldsymbol\eta_j=0$ for all $j$. Hence $\bfK_\dive$ is positive definite.
\end{proof}


The classical
potential-based construction \cite{Drake_2021SISC_partition,Drake_2022SISC_implicit,
Fuselier_2009MCoM_error,Fuselier_2009SINUM_stability} corresponds to $m=2$. 
 In this case,
termwise differentiation of the scalar zonal expansion yields
$\phi'(t)
    =
    \sum_{\ell=1}^{\infty}
    \frac{2\ell+1}{4\pi}
    \widehat{\phi}(\ell)P_\ell'(t)$ and $
    \phi''(t)
    =
    \sum_{\ell=1}^{\infty}
    \frac{2\ell+1}{4\pi}
    \widehat{\phi}(\ell)P_\ell''(t).$
Hence, the associated divergence-free kernel is
\[
    \bfK_{\dive}(\bx,\by)
    =
    \sum_{\ell=1}^{\infty}
    \frac{2\ell+1}{4\pi}
    \widehat{\phi}(\ell)
    \bigl(
        P_\ell'(t)\bfR
        +
        P_\ell''(t)\bfQ
    \bigr).
\]
Applying the vector addition formula in
Lemma~\ref{lem:vec_addition}, we find that the corresponding Fourier coefficients satisfy $\kappa_\ell
    =
    \lambda_\ell\widehat{\phi}(\ell)$.
Thus, for a scalar zonal kernel with coefficients
$\widehat{\phi}(\ell)\asymp(1+\lambda_\ell)^{-\sigma}$, one obtains $\kappa_\ell
    \asymp\ell^{-2\sigma+2}$.
    
By contrast, Corollary~\ref{corol:kdiv_equivalence} below shows that the
lower-order construction corresponding to $m=1$ yields
$\kappa_\ell\asymp \ell^{-2\sigma}$. Thus, the choice $m=1$ improves the decay of the divergence-free Fourier
coefficients by a factor of order $\ell^{-2}$. Equivalently, for the same underlying scalar Fourier decay, the resulting native space is smoother by one Sobolev order; see, for example, \cite{Fuselier_2009SINUM_stability}.

\begin{corollary}
\label{corol:kdiv_equivalence}
Under the assumptions of Theorem~\ref{thm:divergence-free-fourier-expansion}
with $m=1$, one has
\begin{equation}
\label{eq:Kdivell_m1}
    \kappa_\ell
    =\frac{\lambda_\ell}{2\ell+1}\big(\widehat{\phi}(\ell-1)-\widehat{\phi}(\ell+1)\big),
    \quad \ell\ge 1.
\end{equation}
Hence the resulting spectral kernel is positive definite if $\widehat{\phi}(\ell-1)>\widehat{\phi}(\ell+1)$ for $\ell\geq 1$.

Moreover, suppose that $\widehat{\phi}(\ell)=c_\sigma(1+\lambda_\ell)^{-\sigma}$ for $\ell\ge 0$ with $\sigma>1$ and $c_\sigma>0$. Then the coefficients in \eqref{eq:Kdivell_m1}
define, through \eqref{eq:Fourier_expan_Matrix}, a continuous positive definite divergence-free kernel. Moreover,
\[
\kappa_\ell\asymp(1+\lambda_\ell)^{-\sigma},
\quad
\calN_{\bfK_\dive}\simeq\bH_\dive^\sigma(\bbS^2).
\]
\end{corollary}

\begin{proof}
    For $m=1$, \eqref{eq:Legendre_expan_varphi} gives
$a_\ell^{(0)}=\widehat{\phi}(\ell)$, so
\eqref{eq:Kdivell_m1} follows from
\eqref{eq:Coeff_Kdiv_general}. To prove positivity and determine the decay, define
\[
    h(x):=c_\sigma(1+x(x+1))^{-\sigma},
    \quad x\ge 0.
\]
By the mean value theorem, for each $\ell\ge 1$ there exists
$\xi_\ell\in(\ell-1,\ell+1)$ such that
\[
    \widehat{\phi}(\ell-1)-\widehat{\phi}(\ell+1)
    =h(\ell-1)-h(\ell+1)
    =
    -2h'(\xi_\ell).
\]
Since
\[
h'(x)=-c_\sigma\cdot\sigma(2x+1)(1+x(x+1))^{-\sigma-1}<0,
\]
we have
$\kappa_\ell>0$ for all $\ell\ge 1$, which implies that
$\bfK_\dive$ is positive definite on tangent data at pairwise distinct nodes.

For $\ell\ge2$, the relation $\xi_\ell\in(\ell-1,\ell+1)$ gives
$2\xi_\ell+1\asymp\ell$ and
$1+\xi_\ell(\xi_\ell+1)\asymp\ell^2$. Hence
\[
h(\ell-1)-h(\ell+1)\asymp\ell^{-2\sigma-1},
\]
and \eqref{eq:Kdivell_m1} yields
\[
\kappa_\ell\asymp
\frac{\ell^2}{\ell}\,\ell^{-2\sigma-1}
\asymp\ell^{-2\sigma}
\asymp(1+\lambda_\ell)^{-\sigma}.
\]
The degree $\ell=1$ is absorbed into the constants. Furthermore,
$\|\mathcal S_\ell(\bx,\by)\|\le C(2\ell+1)$, so the series
\eqref{eq:Fourier_expan_Matrix} converges absolutely and uniformly because $\sigma>1$. The native-space equivalence follows from the multiplier decay and the characterization in Section~2.2.
\end{proof}

\begin{remark}
\emph{Corollary~\ref{corol:kdiv_equivalence} establishes the one-order gain for the case $m=1$. It also motivates the formal $m=0$ antiderivative
	construction, but the present paper does not establish a general
	positive-definiteness or native-space theorem for that case. Such results
	require additional conditions on the underlying radial kernel; see \cite{Sun_2026_error}. Nevertheless, the numerical experiments reported below provide evidence for the effectiveness of these lower-order constructions.}
\end{remark}

\subsection{A multiplier-preserving construction}
The preceding $m=2$ relation contains a factor of $\lambda_\ell$. On the
sphere, this factor can be cancelled by first applying the inverse
Laplace--Beltrami operator to the mean-zero component of the scalar zonal
kernel. Let
\[
c_\phi:=\frac{\widehat{\phi}(0)}{4\pi},
\quad
u:=(-\Delta_*)^{-1}(\phi-c_\phi),
\quad
\widehat{u}(0)=0.
\]
Then, we have
$\widehat{u}(\ell)
=\frac{\widehat{\phi}(\ell)}{\lambda_\ell}$.
Writing $\psi:=u'$, the divergence-free kernel generated by $u$ is
\[
\bfK_{\dive}(\bx,\by)
=
\psi(t)\bfR(\bx,\by)
+
\psi'(t)\bfQ(\bx,\by).
\]
For a zonal function,
$\Delta_*u=((1-t^2)u')'$. Hence
$-\Delta_*u=\phi-c_\phi$ gives
\[
\psi(t)
=
\frac{1}{1-t^2}
\int_{-1}^{t}
\bigl(c_\phi-\phi(s)\bigr)\,\mathrm{d}s,
\quad
-1<t<1.
\]
The vector multipliers of this construction satisfy
\[
\kappa_\ell
=
\lambda_\ell\widehat{u}(\ell)
=
\widehat{\phi}(\ell),
\quad
\ell\ge1.
\]
This sphere-specific multiplier-preserving kernel is distinct from the
$m=1$ case. Explicit inverse potentials for the Poisson,
Gaussian, and selected Wendland kernels are given in
\cite[Table~1]{Sun_2026_vector}.

\begin{theorem}
	\label{thm:special-zonal-construction}
	Let $\phi\in C([-1,1])$ be a zonal kernel satisfying
	$\sum_{\ell=1}^{\infty}
	\ell\,
	|\widehat{\phi}(\ell)|
	<\infty$.
	Let $u$ be the mean-zero solution of
	$-\Delta_*u=\phi-\widehat{\phi}(0)/(4\pi)$, and set $\psi=u'$. Define
	\[
	\bfK_{\dive}(\bx,\by)
	:=
	\sum_{\ell=1}^{\infty}
	\widehat{\phi}(\ell)
	\sum_{k=1}^{2\ell+1}
	\mathbf y_{\ell,k}(\bx)
	\mathbf y_{\ell,k}(\by)^\top.
	\]
	This series converges absolutely and uniformly. For $-1<t<1$, it agrees with
	\[
	\bfK_{\dive}(\bx,\by)
	=
	\psi(t)\bfR(\bx,\by)
	+
	\psi'(t)\bfQ(\bx,\by),
	\]
	where the expression at $t=\pm1$ is defined by its continuous spectral
	extension. Consequently, if $\widehat{\phi}(\ell)>0$ for all $\ell\ge1$,
	then $\bfK_{\dive}$ is positive definite on tangent
	data at pairwise distinct nodes. If, in addition,
	\begin{equation}\label{eq:decay_property}
		\widehat{\phi}(\ell)
		\asymp
		(1+\lambda_\ell)^{-\sigma},
		\quad
		\ell\ge1,
		\quad
		\sigma>1,
	\end{equation}
	then the native space $\calN_{\bfK_\dive}$ is norm-equivalent to
	$\bH_\dive^{\sigma}(\bbS^2)$.
\end{theorem}

\begin{proof}
	The inverse Laplace--Beltrami relation gives
	\[
	\psi(t)
	=
	\sum_{\ell=1}^{\infty}
	\frac{2\ell+1}{4\pi\lambda_\ell}
	\widehat{\phi}(\ell)P_\ell'(t)
	\]
	in the distributional sense and pointwise wherever the corresponding scalar
	series converges. The integral formula above shows that
	$\psi\in C^1((-1,1))$. Applying the vector addition formula to
	$\psi\bfR+\psi'\bfQ$ in the interior gives the stated spectral expansion.
	Conversely, the bound
	\[
	\Big\|
	\sum_{k=1}^{2\ell+1}
	\mathbf y_{\ell,k}(\bx)
	\mathbf y_{\ell,k}(\by)^\top
	\Big\|
	\le C(2\ell+1)
	\]
	and the summability assumption show that this matrix-valued series converges
	absolutely and uniformly. Its continuous limit provides the extension at
	$t=\pm1$.
	
	If $\widehat{\phi}(\ell)$ is positive, positive definiteness
	follows from the argument used in
	Theorem~\ref{thm:divergence-free-fourier-expansion}. Finally,
	\eqref{eq:decay_property} and the multiplier characterization of the native
	space give
	$\calN_{\bfK_{\dive}}
	\simeq
	\bH_\dive^\sigma(\bbS^2)$.
\end{proof}

\section{Divergence-free interpolation}
Let $X=\{\bx_j\}_{j=1}^N\subset\mathbb S^2$ be a set of distinct
nodes. For each $j$, choose a matrix
$\calP_j=[\bfd_j,\bfe_j]\in\mathbb R^{3\times2}$ whose columns form an
orthonormal basis of $T_{\bx_j}\mathbb S^2$. Thus,
\[
\calP_j^\top \calP_j=\bfI_2,
\quad
\calP_j \calP_j^\top=\bfI_3-\bx_j\bx_j^\top.
\]
The obvious choices for the tangent vectors $\bfd_j$ and $\bfe_j$ are the standard meridional and zonal vectors, respectively:
\begin{equation*}
    \bfd_j=\frac{1}{\sqrt{1-x_{j,3}^2}}\begin{bmatrix}-x_{j,3}x_{j,1},-x_{j,3}x_{j,2},1-x_{j,3}^2\end{bmatrix}^\top, \quad \bfe_j=\frac{1}{\sqrt{1-x_{j,3}^2}}\begin{bmatrix}-x_{j,2},x_{j,1},0\end{bmatrix}^\top.
\end{equation*}

Let $\bff$ be a tangential vector field and define its local data
coordinates by
\[
\bgg_j:=\calP_j^\top\bff(\bx_j)\in\mathbb R^2.
\]
We seek an interpolant of the form
\begin{equation}\label{eq:interpolant_form}
I_X\bff(\bx)
=
\sum_{j=1}^N
\bfK_{\dive}(\bx,\bx_j)\calP_j\bc_j,
\quad
\bc_j\in\mathbb R^2.
\end{equation}
Imposing $I_X\bff(\bx_i)=\bff(\bx_i)$ and expressing the equations in
the local tangent frames gives
\begin{equation}\label{eq:linear_system}
\sum_{j=1}^N
\underbrace{
\calP_i^\top\bfK_{\dive}(\bx_i,\bx_j)\calP_j
}_{\displaystyle \calA_{ij}}
\bc_j
=
\bgg_i,
\quad 1\le i\le N.
\end{equation}
Accordingly, the interpolation matrix
$\calA_{\bfK_{\dive},X}\in\mathbb R^{2N\times2N}$ consists of the
$2\times2$ blocks $\calA_{ij}$. If $\bfK_{\dive}$ is positive
definite on tangential vectors, then $\calA_{\bfK_{\dive},X}$ is symmetric positive definite.

\begin{proposition}
	\label{prop:interpolation-projection}
	Let $X=\{\bx_1,\ldots,\bx_N\}\subset\bbS^2$ consist of pairwise
	distinct nodes, and let $\calP_j\in\mathbb R^{3\times2}$ be an
	orthonormal frame for $T_{\bx_j}\bbS^2$. Suppose that $\bfK_\dive$ is
	a continuous, symmetric, tangent-valued kernel that is strictly positive
	definite on tangent data and is the reproducing kernel of the
	vector-valued RKHS $\mathcal N_{\bfK_\dive}$.
	Then \eqref{eq:linear_system} has a unique solution for arbitrary tangent
	data. If $\bff\in\mathcal N_{\bfK_\dive}$ and
	\begin{equation}\label{eq:trialsp}
		\calV_{X,\bfK_\dive}
		:=
		\operatorname{span}
		\left\{
		\bfK_\dive(\cdot,\bx_j)\boldsymbol\eta:
		\boldsymbol\eta\in T_{\bx_j}\bbS^2,\ 
		1\le j\le N
		\right\},
	\end{equation}
	then $I_X\bff$ is the
	$\mathcal N_{\bfK_\dive}$-orthogonal projection of $\bff$ onto
	$\calV_{X,\bfK_\dive}$, and
	\begin{equation}\label{eq:minimumNorm}
		\|\bff-I_X\bff\|_{\mathcal N_{\bfK_\dive}}
		=
		\inf_{\bv\in\calV_{X,\bfK_\dive}}
		\|\bff-\bv\|_{\mathcal N_{\bfK_\dive}}.
	\end{equation}
	Moreover, if $\widetilde{\calP}_j=\calP_j\calW_j$, where
	$\calW_j\in O(2)$, and
	$\calW=\operatorname{diag}(\calW_1,\ldots,\calW_N)$, then
	\[
	\widetilde{\calA}_{\bfK_\dive,X}
	=
	\calW^\top\calA_{\bfK_\dive,X}\calW.
	\]
	Consequently, the eigenvalues, the spectral condition number, and the
	interpolated field are independent of the chosen orthonormal tangent frames.
\end{proposition}

\begin{proof}
	For a coefficient vector
	$\bc=(\bc_1^\top,\ldots,\bc_N^\top)^\top\ne0$, set
	$\boldsymbol\eta_j=\calP_j\bc_j\in T_{\bx_j}\bbS^2$. Since the columns
	of $\calP_j$ are orthonormal, the vectors $\boldsymbol\eta_j$ are not all
	zero. Positive definiteness gives
	\[
	\bc^\top\calA_{\bfK_\dive,X}\bc
	=
	\sum_{i,j=1}^N
	\boldsymbol\eta_i^\top
	\bfK_\dive(\bx_i,\bx_j)
	\boldsymbol\eta_j
	>0.
	\]
	Thus, $\calA_{\bfK_\dive,X}$ is symmetric positive definite, and
	\eqref{eq:linear_system} has a unique solution.
	
	The reproducing property and the interpolation conditions imply
	\[
	\left\langle
	\bff-I_X\bff,\,
	\bfK_\dive(\cdot,\bx_j)\boldsymbol\eta
	\right\rangle_{\mathcal N_{\bfK_\dive}}
	=
	\boldsymbol\eta^\top
	\bigl(\bff-I_X\bff\bigr)(\bx_j)
	=
	0.
	\]
	Hence $\bff-I_X\bff$ is orthogonal to
	$\calV_{X,\bfK_\dive}$, proving both the projection property and
	\eqref{eq:minimumNorm}.
	
	
	Finally, the transformed matrix blocks satisfy
	\[
	\widetilde{\calA}_{ij}
	=
	\calW_i^\top\calA_{ij}\calW_j,
	\]
	which yields
	$\widetilde{\calA}_{\bfK_\dive,X}
	=\calW^\top\calA_{\bfK_\dive,X}\calW$.
	The corresponding coefficient vectors satisfy
	$\widetilde{\bc}=\calW^\top\bc$,
	\[
	\calP_j\calW_j\widetilde{\bc}_j
	=
	\calP_j\bc_j.
	\]
	Therefore, changing the tangent frames leaves the interpolated field
	unchanged. Since $\calW$ is orthogonal, the two interpolation matrices are
	orthogonally similar and consequently have the same eigenvalues and spectral
	condition number.
\end{proof}

\section{Stability and error estimates}
In this section, we establish stability and error estimates for divergence-free interpolation with matrix-valued kernels on $\mathbb{S}^2$. We state the assumptions directly in terms of the Fourier multipliers and conduct the analysis in spherical Fourier space. The kernels need not possess ambient radial extensions; instead, stability is derived directly from their multipliers, making the argument technically distinct from existing approaches. Adapting ideas from \cite{Karvonen_2025_general}, we also establish a superconvergence result in the Hilbert scale generated by the kernel.

\begin{assumption}\label{ass:spectral-kernel}
    Assume that $\kappa_\ell$ satisfies the decay condition \eqref{eq:decay_condition_kernel} with $\sigma>1$. Suppose that the divergence-free kernel $\bfK_\dive$ has the expansion
    \begin{equation}\label{eq:general-spectral-kernel-section5}
\bfK_\dive(\bx,\by)
=\sum_{\ell=1}^\infty\kappa_\ell
\sum_{k=1}^{2\ell+1}
\mathbf y_{\ell,k}(\bx)\mathbf y_{\ell,k}(\by)^\top,
\quad \kappa_\ell>0.
\end{equation}
\end{assumption}

\subsection{Stability}
For a node set $X=\{\bx_1,\ldots,\bx_N\}\subset \bbS^2$, define the \emph{mesh norm}
(or \emph{fill distance}) and the \emph{separation distance}, respectively, by
 $$h_{X}:=\sup_{\bx\in\bbS^2}\min_{\bx_j\in X}\dist(\bx,\bx_j),\quad q_X:=\frac{1}{2}\min_{i\ne j}\dist(\bx_i,\bx_j),$$
where $\dist(\bx,\by):=\arccos(\bx\cdot\by)$ denotes the geodesic distance. Let $\bfK_\dive$ satisfy Assumption~\ref{ass:spectral-kernel}. The associated approximation space is defined as in \eqref{eq:trialsp}.
The following theorem provides a lower
bound for the smallest eigenvalue of $\calA_{\bfK_\dive,X}$. 


\begin{theorem}\label{thm:stability}
Let $X\subset\bbS^2$ be a finite set of distinct nodes with separation
distance $q_X$, and let $\bfK_\dive$ be the divergence-free
matrix-valued kernel satisfying
Assumption~\ref{ass:spectral-kernel}. Then there exist constants $C>0$ and
$\zeta_0>0$, independent of $q_X$ and $N$, such that
\begin{equation}\label{eq:stability_bound}
\lambda_{\min}(\calA_{\bfK_\dive,X})
\ge
C L^2 \min_{1\le \ell\le L}\kappa_\ell,
\quad
L:=\left\lceil \frac{\zeta_0}{q_X}\right\rceil .
\end{equation}
\end{theorem}

The proof uses two auxiliary results. The first is a localization estimate for smooth spectral filters on the sphere. Such estimates are standard for scalar spherical harmonics; see, for example,
\cite[Lemma 3.3]{Brown_2005JFA_approximation},
\cite[Theorem 2.6.7]{Dai_2013book_approximation},
\cite[Theorem 4.1]{Mhaskar_2010MCoM_bernstein} and
\cite{Mhaskar_2005ACHA_representation}. We require the corresponding vector spherical harmonic estimate. The second result is a packing bound for separated points on $\bbS^2$. The proofs are provided in the Appendix.

\begin{lemma}\label{lem:filtered-localization}
Let $\mathsf{w}\in C^{\infty}([0,\infty))$ be supported in $(1/4,7/8)$, satisfy $0\le \mathsf{w}\le 1$, and $\mathsf{w}(x)=1$ on $[1/2,3/4]$. For an integer $L\ge 1$, define the filtered matrix-valued kernel
\[
\mathbf K_{L,\mathsf{w}}(\bx,\by)
:=
\sum_{\ell=1}^{\infty}
\mathsf{w}\!\left(\frac{\ell}{L}\right)
\sum_{k=1}^{2\ell+1}
\mathbf y_{\ell,k}(\bx)\mathbf y_{\ell,k}(\by)^{\top},
\quad \bx,\by\in\mathbb S^2.
\]
Also define
\[
\calB_{L,\mathsf{w}}(t):=
\sum_{\ell=1}^{\infty}\mathsf{w}\!\left(\frac{\ell}{L}\right)
\frac{2\ell+1}{4\pi\lambda_\ell}P_\ell(t).
\]
Then, for any $\nu>0$, there exists a constant
$C_{\nu,\mathsf{w}}>0$ such that, with
$\theta=\dist(\bx,\by)$,
\[
|\calB_{L,\mathsf{w}}^{(j)}(\cos\theta)|
\le C_{\nu,\mathsf{w}}L^{2j}(1+L\theta)^{-\nu-2},
\quad j=1,2,
\]
and
\[
\|\mathbf K_{L,\mathsf{w}}(\bx,\by)\|_2
\le C_{\nu,\mathsf{w}} L^2(1+L\theta)^{-\nu},
\]
where $\|\mathbf K_{L,\mathsf{w}}\|_2:=\max_{\bc\neq \mathbf{0}}\frac{\|\mathbf K_{L,\mathsf{w}} \bc\|_2}{\|\bc\|_2}$ denotes the induced $2$-norm of the matrix $\mathbf K_{L,\mathsf{w}}$.
\end{lemma}

\begin{lemma}\label{lem:packing-estimate}
Let $X=\{\bx_1,\ldots,\bx_N\}\subset \mathbb S^2$ have pairwise distances of at least $2q_X$. There exists a constant $C>0$ such that,
for any $i$ and integer $k\ge 1$,
$$\#\big\{
j:\; k q_X\le \dist(\bx_i,\bx_j)<(k+1)q_X
\big\}\le C(k+1).$$
Consequently, for any $\nu>2$, there exists a constant $C>0$,
independent of $L$, $N$, and $q_X$, such that, whenever $Lq_X\ge 1$,
$$
\sup_{1\le i\le N}\sum_{j\ne i}
\left(1+L\dist(\bx_i,\bx_j)\right)^{-\nu}\le C(Lq_X)^{-\nu}.
$$
\end{lemma}

\begin{proof}[\textbf{Proof of Theorem \ref{thm:stability}}]
Let $\bc=(\bc_1,\ldots,\bc_N)^\top\in\mathbb R^{2N}$ be arbitrary, and set $\bfeta_j=\calP_j\bc_j\in T_{\bx_j}\bbS^2$. Then, $\sum_j\|\bfeta_j\|_2^2=\sum_j\|\bc_j\|_2^2$. The vector spherical harmonic expansion of $\bfK_{\dive}$ yields 
\begin{equation}\label{eq:stability-quad}
\bc^\top\calA_{\bfK_\dive,X}\bc
=\sum_{\ell=1}^{\infty}\kappa_\ell\sum_{k=1}^{2\ell+1}
\Big|\sum_{j=1}^N
\bfeta_j^\top \mathbf{y}_{\ell,k}(\bx_j)\Big|^2 .
\end{equation}
Since $\kappa_\ell>0$, truncation at any degree $L\ge1$ gives
$$
\bc^\top\calA_{\bfK_\dive,X}\bc
\ge\Big(\min_{1\le \ell\le L}\kappa_\ell\Big)
\sum_{\ell=1}^{L}\sum_{k=1}^{2\ell+1}
\Big|\sum_{j=1}^N
\bfeta_j^\top \mathbf{y}_{\ell,k}(\bx_j)\Big|^2.
$$

Let $\mathsf{w}$ be a smooth filter satisfying the assumptions in Lemma~\ref{lem:filtered-localization}. Since $\mathsf{w}(\ell/L)\le1$ for all $\ell$, it follows that
\begin{equation}\label{eq:Quadratic_form}
    \begin{aligned}
\sum_{\ell=1}^{L}\sum_{k=1}^{2\ell+1}
\Big|
\sum_{j=1}^{N}
\bfeta_j^{\top}\mathbf y_{\ell,k}(\bx_j)
\Big|^2\ge &~
\sum_{\ell=1}^{\infty}
\mathsf{w}\Big(\frac{\ell}{L}\Big)
\sum_{k=1}^{2\ell+1}
\Big|
\sum_{j=1}^{N}
\bfeta_j^{\top}\mathbf y_{\ell,k}(\bx_j)
\Big|^2\\
= &~
\sum_{i,j=1}^{N}
\bfeta_i^{\top}\mathbf K_{L,\mathsf{w}}(\bx_i,\bx_j)\bfeta_j .
\end{aligned}
\end{equation}
It remains to bound the resulting localized quadratic form from below.
We separate this form into its diagonal and off-diagonal parts. Setting $\by=\bx$ in \eqref{eq:K-Lw-from-H} gives
\[
\bfK_{L,\mathsf{w}}(\bx,\bx)=\mathsf{d}_L\bfR(\bx,\bx),\quad \mathsf{d}_L:=\frac{1}{8\pi}\sum_{\ell=1}^{\infty}
\mathsf{w}\left(\frac{\ell}{L}\right)(2\ell+1),
\]
where we used $P_\ell'(1)=\lambda_\ell/2$ and
$\bfQ(\bx,\bx)=0$. Since $\mathsf{w}(x)\equiv 1$ on $[1/2,3/4]$, there exist constants
$a_\mathsf{w}>0$ and $L_\mathsf{w}\ge1$, depending only on $\mathsf{w}$, such that
\[
\mathsf{d}_L \ge \frac{1}{8\pi}\sum_{\ell=\lceil L/2\rceil}^{\lfloor 3L/4\rfloor}(2\ell+1) \ge a_\mathsf{w} L^2, \quad \text{for all } L \ge L_\mathsf{w}.
\]
Moreover, $\bfR(\bx,\bx)=\bfI-\bx\bx^\top$ is the orthogonal
projection onto $T_{\bx}\bbS^2$. Hence
$\bfR(\bx_j,\bx_j)\bfeta_j=\bfeta_j$, and the diagonal summation satisfies
$$
\sum_{j=1}^{N}
\bfeta_j^{\top}\mathbf K_{L,\mathsf{w}}(\bx_j,\bx_j)\bfeta_j
=
\mathsf{d}_L\sum_{j=1}^{N}\|\bfeta_j\|_2^2
\ge
a_\mathsf{w} L^2
\|\bfeta\|_2^2.
$$

We next estimate the off-diagonal part. Fix $\nu>2$.
Lemmas~\ref{lem:filtered-localization} and
\ref{lem:packing-estimate} imply that, for some $C_{\nu,\mathsf{w}}>0$ and
whenever $Lq_X\ge1$,
$$
\sup_{1\le i\le N}
\sum_{j\ne i}
\left\|\mathbf K_{L,\mathsf{w}}(\bx_i,\bx_j)\right\|_{2}
\le
C_{\nu,\mathsf{w}} L^2 (Lq_X)^{-\nu}.
$$
Choose $\zeta_0\geq \pi L_\mathsf{w}/2$ sufficiently large so that
$C_{\nu,\mathsf{w}}\zeta_0^{-\nu}\le a_\mathsf{w}/2$. If $Lq_X\ge\zeta_0$, then
$L\ge \zeta_0/q_X\ge L_\mathsf{w}$ and
\[
\sup_{1\le i\le N} \sum_{\substack{j=1\\ j\ne i}}^N \left\|\mathbf K_{L,\mathsf{w}}(\bx_i,\bx_j)\right\|_{2} \le \frac{1}{2}a_\mathsf{w}L^2.
\]
The symmetry
$\bfK_{L,\mathsf{w}}(\bx,\by)=\bfK_{L,\mathsf{w}}(\by,\bx)^{\top}$ and the inequality $2ab\le a^2+b^2$ therefore give
$$
\begin{aligned}
\Big|
\sum_{i\ne j}
\bfeta_i^{\top}\bfK_{L,\mathsf{w}}(\bx_i,\bx_j)\bfeta_j
\Big|
&\le
\sum_{i\ne j}
\left\|\bfK_{L,\mathsf{w}}(\bx_i,\bx_j)\right\|_2
\|\bfeta_i\|_2\|\bfeta_j\|_2                                      \\
&\le
\sum_{i=1}^{N}
\|\bfeta_i\|_2^2
\sum_{j\ne i}
\left\|\bfK_{L,\mathsf{w}}(\bx_i,\bx_j)\right\|_2 \le
\frac{1}{2}a_\mathsf{w}L^2
\|\bfeta\|_2^2.
\end{aligned}
$$

Combining the diagonal and off-diagonal estimates yields
$$
\begin{aligned}
\sum_{i,j=1}^{N}
\bfeta_i^{\top}\bfK_{L,\mathsf{w}}(\bx_i,\bx_j)\bfeta_j
&\ge a_\mathsf{w} L^2
\|\bfeta\|_2^2
-\frac{a_\mathsf{w}}{2}L^2
\|\bfeta\|_2^2=\frac{1}{2}a_\mathsf{w}L^2
\|\bfeta\|_2^2.
\end{aligned}
$$
Substituting this lower bound into \eqref{eq:Quadratic_form} gives
$$
\bc^\top\calA_{\bfK_\dive,X}\bc
\ge
\frac{1}{2}a_\mathsf{w}L^2
\left(\min_{1\le \ell\le L}\kappa_\ell\right)
\|\bc\|_2^2 .
$$
Taking the infimum of the corresponding Rayleigh quotient over all nonzero coordinate vectors $\bc$ yields
\[
\lambda_{\min}\big(\mathcal A_{\mathbf K_{\mathrm{div}},X}\big)
\ge \frac{1}{2}a_\mathsf{w} L^2 \min_{1\le \ell\le L}\kappa_\ell,
\]
with $L=\lceil\zeta_0/q_X\rceil$. This completes the proof.
\end{proof}

\begin{corollary}\label{cor:smallest-eigenvalue-rate}
Under Assumption~\ref{ass:spectral-kernel}, there exists a constant $C>0$, independent of $q_X$, such that
$$ \lambda_{\min}(\calA_{\bfK_\dive, X}) \ge C q_X^{2\sigma-2}. $$
\end{corollary}

\begin{proof}
Since $\lambda_\ell=\ell(\ell+1)\asymp\ell^2$, the decay condition implies
\[
\min_{1\le\ell\le L}\kappa_\ell\ge cL^{-2\sigma}.
\]
Choose $L=\lceil\zeta_0/q_X\rceil$ as in
Theorem~\ref{thm:stability}. Because $q_X\le\pi/2$, there exists a
constant $C$, independent of $X$, such that $L\le Cq_X^{-1}$.
Since $2-2\sigma<0$, Theorem~\ref{thm:stability} yields
\[
\lambda_{\min}(\calA_{\bfK_\dive,X})
\ge cL^{2-2\sigma}
\ge Cq_X^{2\sigma-2}.
\]
This completes the proof.
\end{proof}

\begin{corollary}
\label{corol:condition-number}
Under Assumption~\ref{ass:spectral-kernel}, we have
\[
\lambda_{\max}(\calA_{\bfK_\dive,X})
\le2N\gamma_{\bfK}\le Cq_X^{-2},
\]
where $\gamma_{\bfK}:=
\frac1{8\pi}\sum_{\ell=1}^\infty(2\ell+1)\kappa_\ell$. Moreover, the following estimate holds
\begin{equation}\label{eq:condition-number-bound}
\operatorname{cond}_2(\calA_{\bfK_\dive,X})
\le Cq_X^{-2\sigma}.
\end{equation}
Let
$\bgg=(\bgg_1^\top,\ldots,\bgg_N^\top)^\top$, where
$\bgg_j=\calP_j^\top\bff(\bx_j)$. For the linear system
$\calA_{\bfK_\dive,X}\bc=\bgg$, one also has
\begin{equation}\label{eq:coefficient-stability}
\|\bc\|_2\le Cq_X^{2-2\sigma}\|\bgg\|_2,
\quad
\|I_X\bff\|_{\mathcal N_{\bfK_\dive}}
\le Cq_X^{1-\sigma}\|\bgg\|_2.
\end{equation}
\end{corollary}

\begin{proof}
Since $\bfK_\dive(\bx,\bx)
=\gamma_{\bfK}(\bfI-\bx\bx^\top)$, we have
$\calP_i^\top\bfK_\dive(\bx_i,\bx_i)\calP_i=\gamma_{\bfK}\bfI_2$.
Because the interpolation matrix is positive definite,
\[
\lambda_{\max}(\calA_{\bfK_\dive,X})
\le\operatorname{trace}(\calA_{\bfK_\dive,X})
=2N\gamma_{\bfK}.
\]
The geodesic caps $B(\bx_j,q_X)$ are pairwise disjoint, and comparison with the area of $\mathbb S^2$ yields $N\le Cq_X^{-2}$. Combining this estimate with Corollary~\ref{cor:smallest-eigenvalue-rate} proves \eqref{eq:condition-number-bound}. The first bound in
\eqref{eq:coefficient-stability} follows from the inverse-matrix norm estimate. Finally,
\[
\|I_X\bff\|_{\mathcal N_{\bfK_\dive}}^2
=\bc^\top\calA_{\bfK_\dive,X}\bc
=\bc^\top\bgg
\le\|\bc\|_2\|\bgg\|_2,
\]
which proves the second bound.
\end{proof}

\subsection{Pointwise error estimates}
Interpolation error estimates are naturally derived first in the native space, where the appropriate tool is the \emph{power function}. For $\bxi\in T_{\bx}\bbS^2$, the reproducing property and native-space orthogonality yield
\begin{equation}\label{eq:pointwise-power-bound}
\left|\bxi^T\big(\bff(\bx)-I_X\bff(\bx)\big)\right|
\le
P_{\bfK_\dive,X,\bxi}(\bx)
\,
\|\bff\|_{\calN_{\bfK_\dive}},
\end{equation}
where the power function is defined by
$$P_{\bfK_\dive,X,\bxi}(\bx)
:=
\inf_{\bchi\in\calV_{X,\bfK_\dive}}
\|\bfK_\dive(\cdot,\bx)\bxi-\bchi\|_{\calN_{\bfK_\dive}}.$$
Thus, it suffices to estimate the power function.

For $L\ge1$, define the polynomial space
\[
\Pi_L^{\rm div}
:=\operatorname{span}
\{\mathbf y_{\ell,k}:1\le\ell\le L,\ 1\le k\le2\ell+1\}.
\]

\begin{lemma}\label{lem:power-function-bound}
	Let $X\subset\bbS^2$ be a finite set. Let $\bxi\in T_{\bx}\mathbb S^2$ with $\|\bxi\|_2=1$, $\bx\in\mathbb S^2$. Then, there exist constants $\zeta_0$, $C>0$, depending only on the sphere, such
	that for $h_X\le \zeta_0/L$,
	\begin{equation}\label{eq:power-function-bound}
		P_{\bfK_\dive,X,\bxi}(\bx)^2
		\le C\sum_{\ell>L}(2\ell+1)\kappa_\ell.
	\end{equation}
\end{lemma}

\begin{proof}
	For a vector field
	$\bp=(p_1,p_2,p_3)^\top$, let
	$D_*\bp(\bz):T_{\bz}\bbS^2\rightarrow\mathbb R^3$
	denote its componentwise surface differential, and define
	\[
	\|\bp\|_{L^\infty}
	:=
	\sup_{\bz\in\bbS^2}\|\bp(\bz)\|_2,
	\quad
	\|D_*\bp\|_{L^\infty}
	:=
	\sup_{\bz\in\bbS^2}\|D_*\bp(\bz)\|_{\mathrm{op}}.
	\]
	
	For every $\bp\in\Pi_L^{\rm div}$ and every fixed
	$\mathbf a\in\mathbb R^3$, the scalar function
	$q_{\mathbf a}:=\mathbf a^\top\bp$ is a spherical polynomial of degree at most $L$.
	Indeed, for a vector spherical harmonic
	$\mathbf y_{\ell,k}=\lambda_\ell^{-1/2}\bL_*Y_{\ell,k}$,
	\[
	\mathbf a^\top\bL_*Y_{\ell,k}
	=
	(\mathbf a\times\bx)\cdot\nabla_*Y_{\ell,k},
	\]
	and the differential operator on the right is an infinitesimal rotation.
	It commutes with the scalar Laplace--Beltrami operator and therefore
	preserves the degree-$\ell$ eigenspace.
	Applying the scalar Bernstein inequality
	\cite{Dai_2013book_approximation,Mhaskar_2010MCoM_bernstein} to
	$q_{\mathbf a}$ gives, for every unit vector $\mathbf a$,
	\[
	\|\nabla_*q_{\mathbf a}\|_{L^\infty}
	\le
	C_B L\|q_{\mathbf a}\|_{L^\infty}
	\le
	C_B L\|\bp\|_{L^\infty}.
	\]
	Taking the supremum over the output direction $\mathbf a$ and over unit tangent
	directions yields
	\begin{equation}\label{eq:vector-bernstein}
		\|D_*\bp\|_{L^\infty}
		\le
		C_B L\|\bp\|_{L^\infty},
		\quad
		\bp\in\Pi_L^{\rm div},
	\end{equation}
	where $C_B$ depends only on the sphere.
	
	Choose $\zeta_0>0$ such that $C_B\zeta_0\le1/2$. Let $\bx_*\in\bbS^2$ satisfy
	$\|\bp(\bx_*)\|_2=\|\bp\|_{L^\infty}$.
	By the definition of the fill distance, there exists $\bx_j\in X$ such that
	$\operatorname{dist}(\bx_*,\bx_j)\le h_X$. Consequently,
	\[
	\begin{aligned}
		\|\bp\|_{L_\infty}
		&\le
		\|\bp(\bx_j)\|_2
		+\|\bp(\bx_*)-\bp(\bx_j)\|_2 \\
		&\le
		\max_{1\le i\le N}\|\bp(\bx_i)\|_2
		+C_BLh_X\|\bp\|_{L^\infty}.
	\end{aligned}
	\]
	If $h_X\le \zeta_0/L$, it follows that
	\begin{equation}\label{eq:vector-norming-set}
		\|\bp\|_{L^\infty}
		\le
		2\max_{1\le j\le N}\|\bp(\bx_j)\|_2,
		\quad
		\bp\in\Pi_L^{\rm div}.
	\end{equation}
	
	Equip
	$\mathcal Y_X
	:=
	\bigoplus_{j=1}^N T_{\bx_j}\bbS^2
	$
	with the block maximum norm
	$\|(\bv_1,\ldots,\bv_N)\|_{\infty,2}
	:=
	\max_{1\le j\le N}\|\bv_j\|_2$.
	The sampling operator
	\[
	E_X:\Pi_L^{\rm div}\longrightarrow\mathcal Y_X,
	\quad
	E_X\bp
	:=
	\bigl(\bp(\bx_1),\ldots,\bp(\bx_N)\bigr),
	\]
	is injective by \eqref{eq:vector-norming-set}. The functional
	$F(E_X\bp):=\bxi^\top\bp(\bx)$
	is therefore well defined on $\operatorname{ran}(E_X)$ and satisfies
	\[
	|F(E_X\bp)|
	\le
	\|\bp\|_{L^\infty}
	\le
	2\|E_X\bp\|_{\infty,2}.
	\]
	The Hahn--Banach theorem extends $F$ to $\mathcal Y_X$ without increasing
	its norm. The dual of the block maximum norm is the block $\ell_1$-norm,
	so the blockwise Riesz representation theorem gives vectors
	$\boldsymbol\eta_j\in T_{\bx_j}\bbS^2$ such that for every $\bp\in\Pi_L^{\rm div}$,
	\begin{equation}\label{eq:low-mode-reproduction}
		\bxi^\top\bp(\bx)
		=
		\sum_{j=1}^N
		\boldsymbol\eta_j^\top\bp(\bx_j),
		\quad
		\sum_{j=1}^N\|\boldsymbol\eta_j\|_2\le2.
	\end{equation}
	
	Lemma~\ref{lem:vec_addition} gives
	$\calS_\ell(\bz,\bz)
	=
	\frac{2\ell+1}{8\pi}
	\bigl(\bfI-\bz\bz^\top\bigr)$.
	Hence, by homogeneity, every
	$\boldsymbol\eta\in T_{\bz}\bbS^2$ satisfies
	\begin{equation}\label{eq:direct-vsh-diagonal}
		\sum_{k=1}^{2\ell+1}
		\left|
		\boldsymbol\eta^\top\mathbf y_{\ell,k}(\bz)
		\right|^2
		=
		\frac{2\ell+1}{8\pi}
		\|\boldsymbol\eta\|_2^2.
	\end{equation}
	
	Set
	$\bchi
	:=
	\sum_{j=1}^N
	\bfK_\dive(\cdot,\bx_j)\boldsymbol\eta_j
	\in\calV_{X,\bfK_\dive}$.
	Since $\mathbf y_{\ell,k}\in\Pi_L^{\rm div}$ whenever $\ell\le L$,
	\eqref{eq:low-mode-reproduction} cancels every mode of degree at most $L$.
	By the characterization of the power function as the distance from
	$\bfK_\dive(\cdot,\bx)\bxi$ to $\calV_{X,\bfK_\dive}$,
	\[
	\begin{aligned}
		P_{\bfK_\dive,X,\bxi}(\bx)^2
		\le
		\left\|
		\bfK_\dive(\cdot,\bx)\bxi-\bchi
		\right\|_{\mathcal N_{\bfK_\dive}}^2=
		\sum_{\ell>L}\kappa_\ell
		\sum_{k=1}^{2\ell+1}
		\big|
		\bxi^\top\mathbf y_{\ell,k}(\bx)
		-
		\sum_{j=1}^N
		\boldsymbol\eta_j^\top
		\mathbf y_{\ell,k}(\bx_j)
		\big|^2.
	\end{aligned}
	\]
	
	For each fixed $\ell$, Minkowski's inequality,
	\eqref{eq:low-mode-reproduction}, and
	\eqref{eq:direct-vsh-diagonal} yield
	\[
	\begin{aligned}
		&\Big(
		\sum_{k=1}^{2\ell+1}
		\Big|
		\bxi^\top\mathbf y_{\ell,k}(\bx)
		-
		\sum_{j=1}^N
		\boldsymbol\eta_j^\top\mathbf y_{\ell,k}(\bx_j)
		\Big|^2
		\Big)^{1/2}
		\\
		\le&~
		\Big(
		\sum_{k=1}^{2\ell+1}
		|\bxi^\top\mathbf y_{\ell,k}(\bx)|^2
		\Big)^{1/2}
		+
		\sum_{j=1}^N
		\Big(
		\sum_{k=1}^{2\ell+1}
		|\boldsymbol\eta_j^\top
		\mathbf y_{\ell,k}(\bx_j)|^2
		\Big)^{1/2}
		\\
		=&~
		\Big(
		1+\sum_{j=1}^N\|\boldsymbol\eta_j\|_2
		\Big)
		\Big(\frac{2\ell+1}{8\pi}\Big)^{1/2}
		\le
		3\Big(\frac{2\ell+1}{8\pi}\Big)^{1/2}.
	\end{aligned}
	\]
	Combining the above two estimates yields
	\[
	P_{\bfK_\dive,X,\bxi}(\bx)^2
	\le
	\frac{9}{8\pi}
	\sum_{\ell>L}(2\ell+1)\kappa_\ell.
	\]
	This completes the proof.
\end{proof}

\begin{theorem}
\label{thm:native-pointwise-error}
Under Assumption \ref{ass:spectral-kernel}, there exist $h_0,~C>0$ such
that, whenever $h_X\le h_0$ and
$\bff\in\mathcal N_{\bfK_\dive}$,
\[
\|\bff-I_X\bff\|_{\bL_\infty(\mathbb S^2)}
\le Ch_X^{\sigma-1}
\|\bff\|_{\mathcal N_{\bfK_\dive}}.
\]
\end{theorem}

\begin{proof}
Choose $h_0$ so that $L\ge1$ and $L\asymp h_X^{-1}$. Applying Lemma
\ref{lem:power-function-bound} and the spectral decay assumption \eqref{eq:decay_condition_kernel}, we can obtain
\[
P_{\bfK_\dive,X,\bxi}(\bx)^2
\le C\sum_{\ell>L}\ell^{1-2\sigma}
\le CL^{2-2\sigma}.
\]
Combining this estimate with \eqref{eq:pointwise-power-bound} and taking the supremum over $\bx$ and all unit tangent vectors proves the result.
\end{proof}

\subsection{Sobolev error estimates}
\label{subsec:sobolev-error-estimates}

Classical Sobolev estimates for the $m=2$ surface kernel were established in \cite{Fuselier_2009MCoM_error,Fuselier_2009SINUM_stability}. Those estimates, however, were stated only for integer-order Sobolev spaces. The fractional-order sampling inequality in
\cite[Theorem~3.2]{arcangeli-2012NumerMath-extension}   allows us to extend the corresponding estimates to noninteger orders. 


\begin{lemma}
\label{lem:vector-zeros-estimate}
Let $\beta>1$ and $0\leq \alpha\leq\beta$. There exist constants
$h_0,C>0$ such that, whenever $h_X\leq h_0$ and
$\bff\in\bH^\beta(\bbS^2)$ satisfying
$\bff|_X=\bld{0}$, one has
\[
\|\bff\|_{\bH^\alpha(\bbS^2)}
\leq C h_X^{\beta-\alpha}
\|\bff\|_{\bH^\beta(\bbS^2)}.
\]
\end{lemma}

\begin{proof}
Let $k=\lfloor\beta\rfloor$. Since $\beta>1$, point evaluation is
well defined on $H^\beta(\bbS^2)$. By mapping of bounded domains to $\bbR^3$ via charts \cite{adams_2003_sobolev,Fuselier_2009MCoM_error}, we can obtain the analogous scalar zeros estimate
\cite[Theorem~3.2]{arcangeli-2012NumerMath-extension} on the sphere,
for every $f\in H^\beta(\bbS^2)$ satisfying $f|_X=0$,
\begin{equation}\label{eq:zero-lower-range}
    \|f\|_{H^s(\bbS^2)}
\leq C h_X^{\beta-s}\|f\|_{H^\beta(\bbS^2)},
\quad 0\leq s\leq k.
\end{equation}
If $\beta$ is an integer, then $k=\beta$, and
the above inequality includes whole
range $0\leq\alpha\leq\beta$.

Suppose now that $\beta$ is not an integer, so that $k<\beta$, and let
$k<\alpha\leq\beta$. Define
\[
\alpha=(1-\vartheta)k+\vartheta\beta, \quad \vartheta=\frac{\alpha-k}{\beta-k}\in(0,1].
\]
Applying the Gagliardo-Nirenberg interpolation inequality \cite[Theorem~1]{brezis2018gagliardo}, we have
\[
\|f\|_{H^\alpha(\bbS^2)}
\leq C
\|f\|_{H^k(\bbS^2)}^{1-\vartheta}
\|f\|_{H^\beta(\bbS^2)}^\vartheta.
\]
Using \eqref{eq:zero-lower-range} with $s=k$, we obtain
\[
\begin{aligned}
\|f\|_{H^\alpha(\bbS^2)}
&\leq
C\left(
h_X^{\beta-k}\|f\|_{H^\beta(\bbS^2)}
\right)^{1-\vartheta}
\|f\|_{H^\beta(\bbS^2)}^\vartheta \\
&=
C h_X^{(\beta-k)(1-\vartheta)}
\|f\|_{H^\beta(\bbS^2)}\\
&= C h_X^{\beta-\alpha}
\|f\|_{H^\beta(\bbS^2)},
\end{aligned}
\]
since $(\beta-k)(1-\vartheta)=\beta-\alpha$.

Finally, applying this estimate to the Cartesian components of $\bff$ and using the equivalence between the componentwise Sobolev norm and the Sobolev norm for tangential vector fields completes the proof.
\end{proof}

\begin{lemma}
\label{lem:divergence-free-surrogate}
Let $1<\tau\le\sigma$, and let
$X=\{\bx_j\}_{j=1}^N$ consist of $N\ge2$ distinct nodes with
separation distance $q_X$.
For any $\bff\in\bH_\dive^\tau(\mathbb S^2)$, there exists
$L\le Cq_X^{-1}$ and $\bp\in\Pi_L^{\rm div}$ such that
\begin{equation}\label{eq:div-free-polynomial}
    \bp|_X=\bff|_X,
\quad
\|\bff-\bp\|_{\bH^\tau(\mathbb S^2)}
\le C\|\bff\|_{\bH^\tau(\mathbb S^2)},
\quad
\|\bp\|_{\bH^\sigma(\mathbb S^2)}
\le Cq_X^{\tau-\sigma}\|\bff\|_{\bH^\tau(\mathbb S^2)}.
\end{equation}
\end{lemma}

\begin{proof}
Choose
$L=\lceil\zeta_\tau/q_X\rceil$, where $\zeta_\tau$ is the constant in
\cite[Theorem~4.1]{Fuselier_2009MCoM_error}. That theorem, which is
stated for any real $\tau>1$, provides
$\bp\in\Pi_L^{\rm div}$ with $\bp|_X=\bff|_X$ and
\[
\|\bff-\bp\|_{\bH^\tau}
\le C\operatorname{dist}_{\bH^\tau}
(\bff,\Pi_L^{\rm div})
\le C\|\bff\|_{\bH^\tau}.
\]
The triangle inequality therefore gives
$\|\bp\|_{\bH^\tau}\le C\|\bff\|_{\bH^\tau}$. Since
$\bp\in\Pi_L^{\rm div}$ and $\sigma\ge\tau$, the spectral definition
gives directly
\[
\begin{aligned}
\|\bp\|_{\bH^\sigma}^2
&=\sum_{\ell\le L}\sum_{k=1}^{2\ell+1}
(1+\lambda_\ell)^\sigma|\widehat\bp_{\ell,k}|^2\\
&\le(1+\lambda_L)^{\sigma-\tau}
\sum_{\ell\le L}\sum_{k=1}^{2\ell+1}
(1+\lambda_\ell)^\tau|\widehat\bp_{\ell,k}|^2\\
&=(1+\lambda_L)^{\sigma-\tau}\|\bp\|_{\bH^\tau}^2.
\end{aligned}
\]
Because $1+\lambda_L\asymp(1+L)^2$, taking square roots yields
\[
\|\bp\|_{\bH^\sigma}
\le C(1+L)^{\sigma-\tau}\|\bp\|_{\bH^\tau}
\le Cq_X^{\tau-\sigma}\|\bff\|_{\bH^\tau}.
\]
This completes the proof.
\end{proof}

\begin{theorem}\label{thm:Outside_native_space}
Let $\bff\in\bH_\dive^\tau(\mathbb S^2)$ and $\bfK_\dive$ satisfy Assumption~\ref{ass:spectral-kernel} with
$1<\tau\le\sigma$. Let $X=\{\bx_j\}_{j=1}^N\subset\mathbb S^2$ be a
set of $N$ distinct points with $h_X\le h_0$, separation distance $q_X$,
and mesh ratio $\rho_X=h_X/q_X$. Then, for any real
$0\le s\le\tau$, we have
\[
\|\bff-I_X\bff\|_{\bH^{s}(\bbS^2)}\leq C \rho_X^{\sigma-\tau} h_X^{\tau-s}\|\bff\|_{\bH^{\tau}(\bbS^2)}.
\]
\end{theorem}
\begin{proof}
Lemma~\ref{lem:divergence-free-surrogate} provides a divergence-free
spherical polynomial $\bp$ of degree $L\lesssim q_X^{-1}$ satisfying \eqref{eq:div-free-polynomial}.
Because $\bp$ and $\bff$ have the same values on $X$,
$I_X\bp=I_X\bff$. Since $(\bff-I_X\bff)|_X=0$,
Lemma~\ref{lem:vector-zeros-estimate} gives
\begin{equation}\label{eq:zeros-estimate1}
	\|\bff-I_X\bff\|_{\bH^s}
	\le Ch_X^{\tau-s}\|\bff-I_X\bff\|_{\bH^\tau}.
\end{equation}

Moreover, by the triangle inequality, we have
\[
\|\bff-I_X\bff\|_{\bH^\tau}
\le\|\bff-\bp\|_{\bH^\tau}
+\|\bp-I_X\bp\|_{\bH^\tau}.
\]
Applying the zeros estimate to the second term, using
$\calN_{\bfK_\dive}\simeq\bH_\dive^\sigma$, and using the
native-space best-approximation property, we obtain
\[
\begin{aligned}
\|\bp-I_X\bp\|_{\bH^\tau}
&\le Ch_X^{\sigma-\tau}
\|\bp-I_X\bp\|_{\bH^\sigma}\\
&\le Ch_X^{\sigma-\tau}
\|\bp-I_X\bp\|_{\calN_{\bfK_\dive}}\\
&\le Ch_X^{\sigma-\tau}\|\bp\|_{\calN_{\bfK_\dive}}\\
&\le Ch_X^{\sigma-\tau}q_X^{\tau-\sigma}
\|\bff\|_{\bH^\tau}
=C\rho_X^{\sigma-\tau}\|\bff\|_{\bH^\tau}.
\end{aligned}
\]
Combining this with \eqref{eq:div-free-polynomial} implies
$\|\bff-I_X\bff\|_{\bH^\tau}\le
C\rho_X^{\sigma-\tau}\|\bff\|_{\bH^\tau}$. Substituting it into \eqref{eq:zeros-estimate1} completes the proof.
\end{proof}

\subsection{Superconvergence}
We next derive a Sobolev-scale error estimate from the preceding estimate. The proof is based on a standard interpolation argument for Hilbert scales. We include
the details to show that the divergence-free constraint causes no additional difficulty. 

For $\vartheta\geq 0$, we define
\begin{equation}\label{eq:PowerSpace}
    \calF^{\vartheta}_\dive(\bbS^2):=
    \left\{
    \bff \in \bH_{\dive}^{0}(\bbS^2)
    :\sum_{\ell=1}^{\infty}\sum_{k=1}^{2\ell+1}
    \kappa_\ell^{-\vartheta}
    |\widehat{\bff}_{\ell,k}|^2 < \infty
    \right\}.
\end{equation}
Then, $\calF^{0}_\dive(\bbS^2)=\bH_\dive^0(\bbS^2)$ and $\calF^{1}_\dive(\bbS^2)=\calN_{\bfK_\dive}$ with the exact native inner product. Moreover, $\kappa_\ell\asymp (1+\lambda_\ell)^{-\sigma}$ implies $\calF^\vartheta_\dive(\bbS^2)\simeq \bH_{\dive}^{\vartheta\sigma}(\bbS^2)$. 

The following lemma turns an $\bL_2$-error estimate into an estimate in the stronger norm $\calF^{1}_\dive(\bbS^2)$, provided the target function has
additional smoothness \cite[Corollary 15]{Karvonen_2025_general}.

\begin{lemma}\label{lem:hilbert-scale-superconvergence}
Let $\varepsilon>0$, and let
$\calQ:\calF^{1}_\dive(\bbS^2)\to \calF^{1}_\dive(\bbS^2)$
be an $\calF^{1}_\dive(\bbS^2)$-orthogonal projection. Suppose that
\begin{equation}\label{eq:Hscale-L2error}
\|\bff-\calQ\bff\|_{\calF^{0}_\dive(\bbS^2)}
\le
\varepsilon \|\bff\|_{\calF^{1}_\dive(\bbS^2)},
\quad \bff\in\calF^{1}_\dive(\bbS^2).
\end{equation}
Then, for any $0\le \vartheta\le 1$ and 
$\bgg\in \calF^{1+\vartheta}_\dive(\bbS^2)$, we have
\begin{equation}\label{eq:Hilbert-superconver2}
\|\bgg-\calQ\bgg\|_{\calF^{1}_\dive(\bbS^2)}
\le
\varepsilon^{\vartheta}
\|\bgg\|_{\calF^{1+\vartheta}_\dive(\bbS^2)} .
\end{equation}
\end{lemma}

\begin{proof}
Let $\calE=I-\mathcal Q$. The case $\vartheta=0$ follows from the contractivity of this orthogonal projection. To prove the case $\vartheta=1$, define the positive self-adjoint operator
\[
\calA\mathbf y_{\ell,k}=\kappa_\ell^{-1}\mathbf y_{\ell,k}.
\]
Then $\|\bgg\|_{\calF_{\dive}^{\,1}}^2=(\calA \bgg,\bgg)_{\calF_{\dive}^{\,0}}$ and
$\|\calA\bgg\|_{\calF_{\dive}^{\,0}}
=\|\bgg\|_{\calF_{\dive}^{\,2}}$. For
$\bgg\in\calF_{\dive}^{\,2}$, orthogonality gives
\[
\begin{aligned}
\|\calE\bgg\|_{\calF_{\dive}^{\,1}}^2
=\langle\calE\bgg,\bgg\rangle_{\calF_{\dive}^{\,1}}=\langle\calE\bgg,\calA\bgg\rangle_{\calF_{\dive}^{\,0}}\le\|\calE\bgg\|_{\calF_{\dive}^{\,0}}
\|\bgg\|_{\calF_{\dive}^{\,2}}.
\end{aligned}
\]
Since $\calE^2=\calE$, the assumed estimate \eqref{eq:Hscale-L2error} applied to
$\calE\bgg$ yields
$\|\calE\bgg\|_{\calF_{\dive}^{\,0}}
\le\varepsilon\|\calE\bgg\|_{\calF_{\dive}^{\,1}}$.
Consequently, we have
\[
\|\calE\bgg\|_{\calF_{\dive}^{\,1}}
\le\varepsilon\|\bgg\|_{\calF_{\dive}^{\,2}}.
\]

We then obtain the intermediate cases by interpolation. Let
$[\cdot,\cdot]_\vartheta$ denote complex interpolation. The vector
spherical harmonic coefficient map is an isometry from
$\calF_{\dive}^{\vartheta}$ onto the weighted sequence space
$\ell_2\bigl(\{\kappa_\ell^{-\vartheta}\}_{\ell,k}\bigr)$.
Hence, the interpolation formula for weighted Hilbert spaces gives
\[
[\calF_{\dive}^{\,a},\calF_{\dive}^{\,b}]_\vartheta
=
\calF_{\dive}^{\,(1-\vartheta)a+\vartheta b}.
\]
This is an equality of spaces with equivalent norms; for the canonical
weighted-sequence norms used here, the interpolation norm equals the
$\calF_{\dive}^{\,(1-\vartheta)a+\vartheta b}$ norm. The endpoint bounds are
$\|\calE\bgg\|_{\calF_{\dive}^{\,1}}\le \|\bgg\|_{\calF_{\dive}^{\,1}}$
and
$\|\calE \bgg\|_{\calF_{\dive}^{\,1}}
\le\varepsilon \|\bgg\|_{\calF_{\dive}^{\,2}}$. The complex interpolation theorem for linear operators therefore yields
\[
\|\calE \bgg\|_{
\calF_{\dive}^{\,1}}
\le\varepsilon^\vartheta \|\bgg\|_{\calF_{\dive}^{\,1+\vartheta}},
\]
which proves \eqref{eq:Hilbert-superconver2}. This is precisely the projection principle
used in \cite[Corollary~15]{Karvonen_2025_general}.
\end{proof}

\begin{theorem}\label{thm:hilbert-scale-superconvergence}
Assume that Assumption~\ref{ass:spectral-kernel} holds with $\sigma>1$
and that the hypotheses of Theorem~\ref{thm:Outside_native_space} are
satisfied. For $0\leq\vartheta\leq1$, let $\calF_\dive^{1+\vartheta}(\bbS^2)$ denote the space defined in \eqref{eq:PowerSpace}.
Then, for any
$\bgg\in\calF_\dive^{1+\vartheta}(\bbS^2)$, we have
\begin{equation}\label{eq:Hilbert-scale-superconvergence}
\|\bgg-I_X\bgg\|_{\calN_{\bfK_\dive}}
\le
C h_X^{\vartheta\sigma}
\|\bgg\|_{\calF_{\dive}^{1+\vartheta}(\bbS^2)} .
\end{equation}
Moreover, there exists a constant $C>0$ independent of $h_X$ such that
\begin{equation}\label{eq:Hs-superconvergence}
\|\bgg-I_X\bgg\|_{\bH^s(\bbS^2)}
\le
C h_X^{(1+\vartheta)\sigma-s}
\|\bgg\|_{\calF_\dive^{1+\vartheta}(\bbS^2)},
\quad 0\le s\le  \sigma .
\end{equation}
\end{theorem}

\begin{proof}
By the estimate from Theorem~\ref{thm:Outside_native_space},
there exists a constant $C>0$ such that
\begin{equation}\label{eq:L2error}
\|\bff-I_X\bff\|_{\bL_2(\bbS^2)}
\le
C h_X^{\sigma}
\|\bff\|_{\calF_{\dive}^1(\bbS^2)},
\quad
\bff\in\calF_{\dive}^1(\bbS^2).
\end{equation}
The interpolation operator $I_X$ is the
$\calF_{\dive}^1(\bbS^2)$-orthogonal projection onto the kernel-based trial space.
Therefore Lemma~\ref{lem:hilbert-scale-superconvergence}, applied with
$\varepsilon=C h_X^\sigma$, gives
\[
\|\bgg-I_X\bgg\|_{\calF_{\dive}^1(\bbS^2)}
\le
(C h_X^\sigma)^\vartheta
\|\bgg\|_{\calF_{\dive}^{1+\vartheta}(\bbS^2)} .
\]
Absorbing $C^\vartheta$ into the generic constant $C$ proves
\eqref{eq:Hilbert-scale-superconvergence}.

The Sobolev zeros lemma \ref{lem:vector-zeros-estimate} gives
\[
\begin{aligned}
\|\bgg-I_X\bgg\|_{\bH^s(\bbS^2)}
&\le
C h_X^{\sigma-s}
\|\bgg-I_X\bgg\|_{\bH^\sigma(\bbS^2)} \\
&\le
C h_X^{\sigma-s}
h_X^{\vartheta\sigma}
\|\bgg\|_{\calF_{\dive}^{1+\vartheta}(\bbS^2)} \\
&=
C h_X^{(1+\vartheta)\sigma-s}
\|\bgg\|_{\calF_{\dive}^{1+\vartheta}(\bbS^2)} .
\end{aligned}
\]
This completes the proof of \eqref{eq:Hs-superconvergence}.
\end{proof}

The theorem shows that additional smoothness of the target field beyond the native space regularity improves convergence in the native norm. The parameter
$\vartheta$ measures this extra smoothness. When $\vartheta=0$,
 the estimate gives the basic bound in $\calF_\dive^1$; when $\vartheta=1$, it gives the full improvement provided by the Hilbert-scale argument. In many applications, the target function is smoother than the functions in the native space. In this setting, one can obtain a standard doubling argument \cite{Fuselier_2012SINUM_scattered,Morton_2002JAT_error,Narcowich_2007FoCM_direct,Schaback_1999MCoM_improved,Sun_2022SISC_kernel,Wendland_2009SINUM_divergence}.

\begin{corollary}\label{cor:Inside_native_space}
    Under the assumptions of Theorem \ref{thm:Outside_native_space}, if $\bff\in \bH_{\dive}^{2\sigma-s}(\bbS^2)$ with $0\leq s\leq  \sigma$, then
    $$\|\bff-I_X\bff\|_{\bH^{s}(\bbS^2)}\leq C h_X^{2(\sigma-s)}\|\bff\|_{\bH^{2\sigma-s}(\bbS^2)}.$$
\end{corollary}
\begin{proof}
    In Theorem~\ref{thm:hilbert-scale-superconvergence}, set
$\vartheta=1-s/\sigma$. Then
$(1+\vartheta)\sigma=2\sigma-s$ and
$(1+\vartheta)\sigma-s=2(\sigma-s)$, which gives the stated estimate.
\end{proof}

\section{Numerical examples}
\label{sec:Numer_Examp}

In this section, we present several numerical experiments to verify the theoretical results for the convergence rates and stability of interpolation with matrix-valued kernels on surfaces.  
We compare the classical potential-based construction with two alternative
constructions derived from the more general isotropic ansatz
\begin{equation*}
	\mathbf K_{\mathrm{div}}(\bx,\by)
	=
	\alpha(r)
	\bigl((\bn_{\bx}\cdot\bn_{\by})\mathbf I
	-\bn_{\by}\bn_{\bx}^{\top}\bigr)
	-
	\beta(r)
	\bigl(\bn_{\bx}\times(\bx-\by)\bigr)
	\bigl(\bn_{\by}\times(\bx-\by)\bigr)^{\top}.
\end{equation*}
For conciseness, we characterize each kernel by its choice of $\beta(r)$.
Once $\beta$ is specified, the divergence-free constraint uniquely determines
$\alpha$ and hence the corresponding kernel. We consider the following three
cases:
\begin{align*}
	&\mathbf K_{\mathrm{div}}^{0}: \alpha_0(r)=\calI \varphi(r), \quad \beta_0(r)=\varphi(r),\\
	&\mathbf K_{\mathrm{div}}^{1}: \alpha_1(r) = \varphi(r),\quad \beta_1(r)=\mathcal D\varphi(r),\\
	&\mathbf K_{\mathrm{div}}^{2}: \alpha_2(r)=\calD \varphi(r),\quad  \beta_2(r)=\mathcal D^2\varphi(r).
\end{align*}

We use two standard families of scalar kernels obtained by restricting radial
basis functions to the surface: the Mat\'ern (MA) and Wendland (WE)
families. A Mat\'ern kernel is defined by
\begin{equation*}
	\varphi_{\nu}(r)
	=
	\frac{2^{1-\nu}}{\Gamma(\nu)}
	(\varepsilon r)^\nu K_\nu(\varepsilon r),
\end{equation*}
where $K_\nu$ denotes the modified Bessel function of the second kind of order $\nu$, and $\varepsilon>0$ is the shape parameter. Mat\'{e}rn kernels are strictly positive definite on $\mathbb R^d$, and the parameter $\nu$ determines their smoothness. For example, the associated radial kernel belongs to $C^2(\mathbb R^d)$ when $\nu=(d+3)/2$, whereas it belongs to $C^4(\mathbb R^d)$ when $\nu=(d+5)/2$. For $\nu=7/2$, the Mat\'{e}rn profile has the explicit form
\begin{equation*}
	\varphi_{7/2}(r)
	=
	e^{-\varepsilon r}
	\left(
	1+\varepsilon r
	+\frac{2}{5}(\varepsilon r)^2
	+\frac{1}{15}(\varepsilon r)^3
	\right).
\end{equation*}

As a compactly supported alternative, we use the Wendland functions
$\varphi_{d,\ell}:[0,\infty)\to\mathbb R$ introduced in
\cite{Wendland_1995Adv_piecewise,Wendland_2004book_scattered}. The corresponding
radial kernels are strictly positive definite on $\mathbb{R}^d$ and belong to
$C^{2\ell}(\mathbb{R}^d)$. In most of our numerical experiments, we use
\begin{equation*}
	\varphi_{3,3}(r)
	=
	(1-\varepsilon r)_+^8
	\left(
	32(\varepsilon r)^3
	+25(\varepsilon r)^2
	+8\varepsilon r+1
	\right).
\end{equation*}
The stability analysis additionally includes $\varphi_{5,3}$ and $\varphi_{7,3}$.

\subsection{Convergence test}

We investigate the convergence and stability of the three kernel constructions
using two vector fields on $\mathbb{S}^2$ with different levels of regularity
\cite{Fuselier_2009SINUM_stability}: a smooth field in
$C^\infty(\mathbb{S}^2)$ (Field~1) and a field of limited regularity belonging
to $\bH^\beta(\mathbb{S}^2)$ with $\beta<2$ (Field~2). Both fields are
generated from scalar stream functions. Given a stream function $s$, the
corresponding vector field is defined by
$\bff=\bL_*s$,
where $\bL_*$ is the surface curl operator.

\smallskip
\textbf{Field 1.} 
The stream function is defined by
\[s_1(\bx)=-\frac1{\sqrt{3}}Y_{1,0}(\bx)+\frac{8\sqrt{2}}{3\sqrt{385}}Y_{5,4}(\bx).\]
Since $s_1$ is a finite linear combination of spherical harmonics, the
resulting field $\bff_1=\bL_* s_1$ belongs to
$C^\infty(\mathbb{S}^2)$. Figure~\ref{field1divcon}(a) shows the stream
function $s_1$.
\smallskip

\textbf{Field 2.} For $\bx_c\in\mathbb{S}^2$ with spherical coordinates $(\theta_c,\lambda_c)$, set $t = \bx\cdot\bx_c$ and $a = 1-t$. Define
\[
g(\bx; \theta_c, \lambda_c) = -\frac{1}{2}\Big((3t + 3\sqrt{2}a^{3/2} - 4) + (3t^2 - 4t + 1)\log a + (3t - 1)a\log\big(\sqrt{2}a + a\big)\Big).
\]
The stream function is
\[
s_2(\bx) = \int_{-\pi/2}^{\theta} \sin^{14}(2\zeta)\,\mathrm{d}\zeta - 3g(\bx; \pi/4, -\pi/12),
\]
where $\theta$ is the latitudinal coordinate of $\bx$. The
corresponding divergence-free field
$\bff_2=\bL_* s_2$ belongs to
$\bH_{\mathrm{div}}^\beta(\mathbb{S}^2)$ with $\beta<2$
\cite{Fuselier_2009MCoM_error}. Figure~\ref{field1divcon}(d) shows the stream
function $s_2$.
\smallskip

We first consider minimum-energy (ME) node sets on the sphere. The interpolation sets contain
$N\in\{529,1025,1849,3137,5041\}$
nodes, while a fixed maximum-determinant (MD) node set
$Y\subset\mathbb{S}^2$ with $M=20164$ nodes is used for error evaluation. We measure the accuracy using the discrete relative $\ell_2$-error
\begin{equation}
	\mathcal{E}(X)
	:=
	\frac{\|\bu-I_X\bu\|_{\ell_2(Y)}}
	{\|\bu\|_{\ell_2(Y)}},
	\quad 
	\|\bu\|_{\ell_2(Y)}^2
	:=
	\sum_{j=1}^{|Y|}w_j\|\bu(\by_j)\|_2^2,
	\label{eq:discrete-relative-error}
\end{equation}
where $I_X\boldsymbol{u}$ denotes the interpolant constructed from the node set $X$, and $\{w_j\}$ are the quadrature weights associated with the MD nodes.

The experiments use the Mat\'{e}rn kernel $\mathrm{MA}_{7/2}$ and the compactly supported Wendland kernel $\mathrm{WE}_{3,3}$. For each kernel, the shape parameter is held fixed as the number of interpolation nodes increases. For Field~1, the shape parameters for the Mat\'{e}rn kernel are $\varepsilon_1=9$, $\varepsilon_2=7$, and $\varepsilon_3=4$, while the corresponding Wendland
shape parameters are $\varepsilon_1=11/10$, $\varepsilon_2=4/3$, and $\varepsilon_3=5/3$. The same parameter choices are used for Field~2.

Table~\ref{tab:Conv_MA} and Table~\ref{tab:Conv_WE} report the errors and convergence rates for the smooth Field~1 obtained using the Mat\'{e}rn and Wendland kernels, respectively. The errors decay approximately as $\mathcal{O}(h_X^{11})$, $\mathcal{O}(h_X^{9})$, and $\mathcal{O}(h_X^{7})$ for $\mathbf{K}_{\dive}^{0}$, $\mathbf{K}_{\dive}^{1}$, and $\mathbf{K}_{\dive}^{2}$, respectively. These
rates are consistent with the Hilbert-scale superconvergence estimate
established in Theorem~\ref{thm:hilbert-scale-superconvergence}. In
particular, the proposed constructions $\mathbf{K}_{\dive}^{0}$ and
$\mathbf{K}_{\dive}^{1}$ attain higher convergence orders than the classical
construction $\mathbf{K}_{\dive}^{2}$. These results support the predicted
accuracy gains for sufficiently smooth target fields.
For Field~2, the limited regularity of the target field restricts the attainable convergence order. Nevertheless, Figure~\ref{fig.Err_Field2} shows that the errors decrease consistently as the nodes are refined. For the classical construction $\mathbf{K}_{\dive}^{2}$, the kernels and corresponding shape parameters coincide with those used in \cite{Fuselier_2009SINUM_stability}, and the resulting numerical values reproduce those reported therein.

Figure~\ref{field1divcon} compares the two target vector fields with their
divergence-free interpolants. The top and bottom rows show Field~1 and
Field~2, respectively, together with their reconstructions obtained using
$\mathbf{K}_{\dive}^{0}$. Both interpolants use $N=1849$ ME nodes. The close agreement between the target and reconstructed
fields indicates that $\mathbf{K}_{\dive}^{0}$ accurately approximates both
fields. Figure~\ref{field1divallerr} further shows the pointwise interpolation errors
for Field~1 obtained using the three kernels $\mathbf{K}_{\dive}^{0}$,
$\mathbf{K}_{\dive}^{1}$, and $\mathbf{K}_{\dive}^{2}$. Among them,
$\mathbf{K}_{\dive}^{0}$ yields the smallest errors across the sphere,
whereas $\mathbf{K}_{\dive}^{2}$ produces the largest.

\begin{table}
	\centering
	\caption{Convergence results for the approximation of Field 1 using the restricted \textbf{MA$_{7/2}$} kernel. The three methods use the shape parameters $\varepsilon_1=9$, $\varepsilon_2=7$ and $\varepsilon_3=4$, respectively. Errors are
		evaluated on $N=20164$ MD nodes.}
	\begin{tabular}{c *{6}{c}} 
		\toprule[1pt] 
		$N$ & $\mathbf{K}_\dive^{0}$ & rate & $\mathbf{K}_\dive^{1}$ & rate & $\mathbf{K}_\dive^{2}$ & rate \\
		\midrule 
        529 & 5.833e-06 & & 8.887e-06 & & 1.013e-05 & \\
        1025 & 1.435e-07 & 11.20 & 4.164e-07 & 9.25 & 8.605e-07 & 7.46 \\
         1849 & 6.316e-09 & 10.59 & 2.926e-08 & 9.00 & 1.034e-07 & 7.18 \\
       3137 & 3.405e-10 & 11.05 & 2.511e-09 & 9.29 & 1.524e-08 & 7.24 \\
       5041 & 2.457e-11 & 11.09 & 3.059e-10 & 8.88 & 2.890e-09 & 7.01 \\
		\bottomrule[1pt] 
	\end{tabular}
	\label{tab:Conv_MA}
\end{table}

\begin{table}
	\centering
	\caption{Convergence results for the approximation of Field~1 using the restricted \textbf{WE$_{3,3}$} kernel. The three methods use the shape parameters $\varepsilon_1=11/10$, $\varepsilon_2=4/3$, and $\varepsilon_3=5/3$, respectively. Errors are evaluated at $N=20164$ MD nodes.}
	\begin{tabular}{c *{6}{c}} 
		\toprule[1pt] 
		$N$ & $\mathbf{K}_\dive^{0}$ & rate & $\mathbf{K}_\dive^{1}$ & rate & $\mathbf{K}_\dive^{2}$ & rate \\
		\midrule 
		529 & 1.222e-05 & & 7.695e-06 & & 1.527e-05 & \\
        1025 & 2.419e-07 & 11.86 & 3.546e-07 & 9.31 & 1.475e-06 & 7.07 \\
       1849 & 9.649e-09 & 10.92 & 2.418e-08 & 9.10 & 1.799e-07 & 7.13 \\
       3137 & 5.280e-10 & 10.99 & 2.069e-09 & 9.30 & 2.621e-08 & 7.29 \\
       5041 & 3.617e-11 & 11.30 & 2.514e-10 & 8.89 & 5.263e-09 & 6.77 \\
		\bottomrule[1pt] 
	\end{tabular}
	\label{tab:Conv_WE}
\end{table}

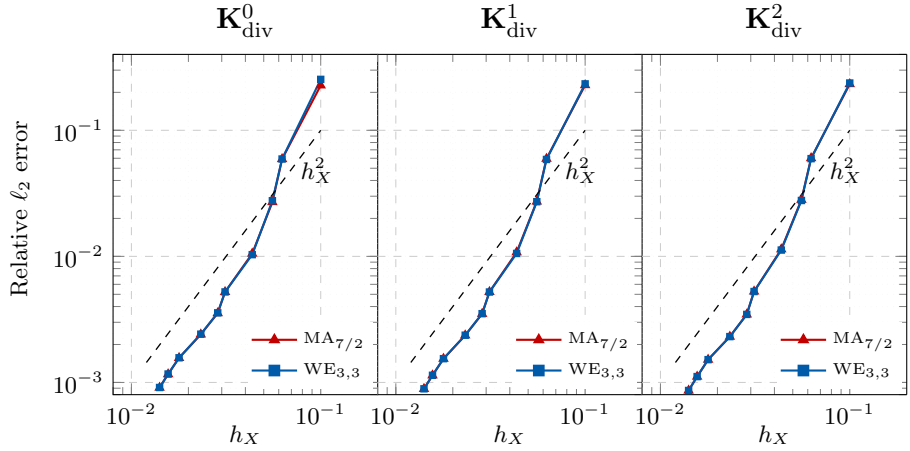
\begin{figure}
	\centering
	\begin{tikzpicture}
		\begin{groupplot}[
			group style={
				group size=3 by 1, 
				horizontal sep=0pt, 
				vertical sep=18pt,
			},
			width=3.5cm, height=4.5cm, 
			scale only axis,           
			xmode=log, ymode=log,
			xmin=8e-3, xmax=0.2,
			ymin=8e-4, ymax=4e-1,
			grid=both,
			major grid style={line width=0.2pt, draw=gray!40, dashed},
			minor grid style={line width=0.1pt, draw=gray!15, dotted},
			minor x tick num=3, minor y tick num=3,
			ticklabel style={font=\small},
			xlabel={$h_X$},
			xlabel style={font=\small, yshift=7pt},
			legend style={
				at={(0.99,0.01)}, 
				anchor=south east,
				font=\tiny, 
				nodes={inner ysep=3pt},
				cells={anchor=west},
				legend columns=1,
				inner sep=1pt,
				outer sep=1pt,
				draw=none, 
				fill opacity=0.8,
				text opacity=1
			},
			legend image post style={mark size=1.8pt}, 
			cycle list={
				{color={myred}, mark=triangle*, line width=0.8pt,mark size=1.5pt},
				{color={myblue}, mark=square*, line width=0.8pt,mark size=1pt},
			},
			]
			
			\nextgroupplot[
			ylabel={Relative $\ell_2$ error}, 
			ylabel style={font=\small, xshift=6pt},
			title={$\mathbf{K}_\dive^{0}$},
			title style={font=\large, yshift=-3pt},	]
			
			\addplot table[x index=0, y index=1, col sep=space] {data_surface/methodfield3diverrork0.txt};
			\addlegendentry{MA$_{7/2}$}

            \addplot table[x index=0, y index=2, col sep=space] {data_surface/methodfield3diverrork0.txt};
			\addlegendentry{WE$_{3,3}$}

			\addplot[
			domain=0.012:0.1,
			samples=2,
			color=black,
			line width=0.5pt,  %
			dashed,
			forget plot,
			] {10* x^2};
			
			\node[anchor=south west, font=\small] at (axis cs:0.07,0.032) {$h_X^{2}$};
			
			\nextgroupplot[
			title={$\mathbf{K}_\dive^{1}$},
			title style={font=\large, yshift=-3pt},
			yticklabels={}, 
            ylabel={},
			axis y line*=right,]
			
			\addplot table[x index=0, y index=1, col sep=space] {data_surface/methodfield3diverrork1.txt};
			\addlegendentry{MA$_{7/2}$}

            \addplot table[x index=0, y index=2, col sep=space] {data_surface/methodfield3diverrork1.txt};
			\addlegendentry{WE$_{3,3}$}

			\addplot[
			domain=0.012:0.1,
			samples=2,
			color=black,
			line width=0.5pt,  %
			dashed,
			forget plot,
			] {10* x^2};
			
			\node[anchor=south west, font=\small] at (axis cs:0.07,0.032) {$h_X^{2}$};

			\nextgroupplot[
			title={$\mathbf{K}_\dive^{2}$},
			title style={font=\large, yshift=-3pt},
			yticklabels={}, 	
            ylabel={},
			axis y line*=right, ]

			\addplot table[x index=0, y index=1, col sep=space] {data_surface/methodfield3diverrork2.txt};
			\addlegendentry{MA$_{7/2}$}

            \addplot table[x index=0, y index=2, col sep=space] {data_surface/methodfield3diverrork2.txt};
			\addlegendentry{WE$_{3,3}$}

			\addplot[
			domain=0.012:0.1,
			samples=2,
			color=black,
			line width=0.5pt,  %
			dashed,
			forget plot,
			] {10* x^2};
			
			\node[anchor=south west, font=\small] at (axis cs:0.07,0.032) {$h_X^{2}$};

			
			
		\end{groupplot}
	\end{tikzpicture}
	
	\captionsetup{font=normalsize}
	\caption{Relative $\ell_2$ errors for Field~2 obtained via matrix-valued kernel interpolation with \textbf{MA$_{7/2}$} and \textbf{WE$_{3,3}$} kernels on ME node sets.}
	\label{fig.Err_Field2}
\end{figure}

\begin{figure}
	\centering
	\begin{tikzpicture}
		\matrix (M) [
		matrix of nodes,
		nodes={inner sep=0pt, anchor=center}, 
		column sep=0.2cm, 
		row sep=0.5cm,   
		] {
			\includegraphics[width=0.3\textwidth, trim=1cm 1cm 0.95cm 0.5cm, clip]{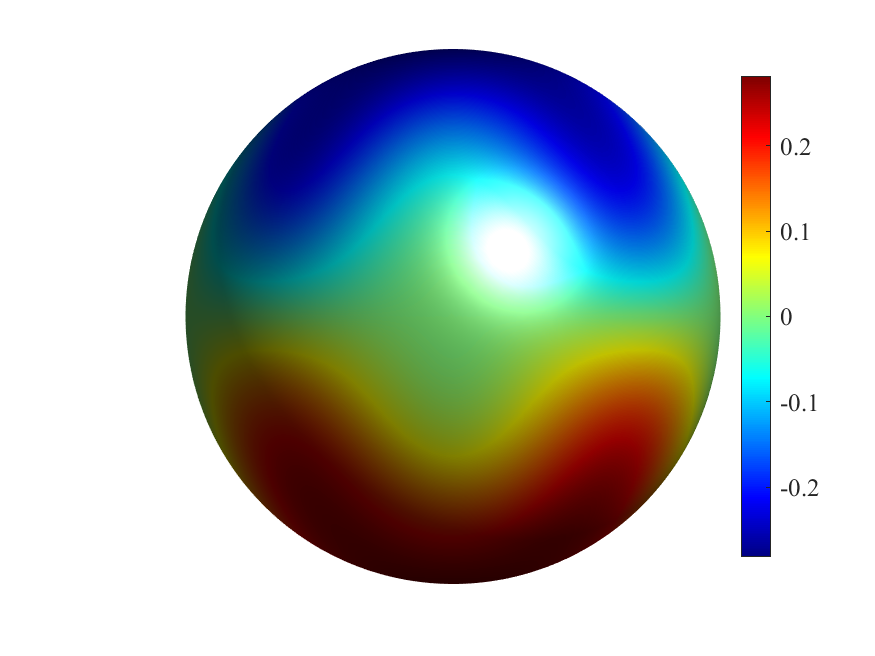} &
			\includegraphics[width=0.3\textwidth, trim=1cm 1cm 1cm 0.5cm, clip]{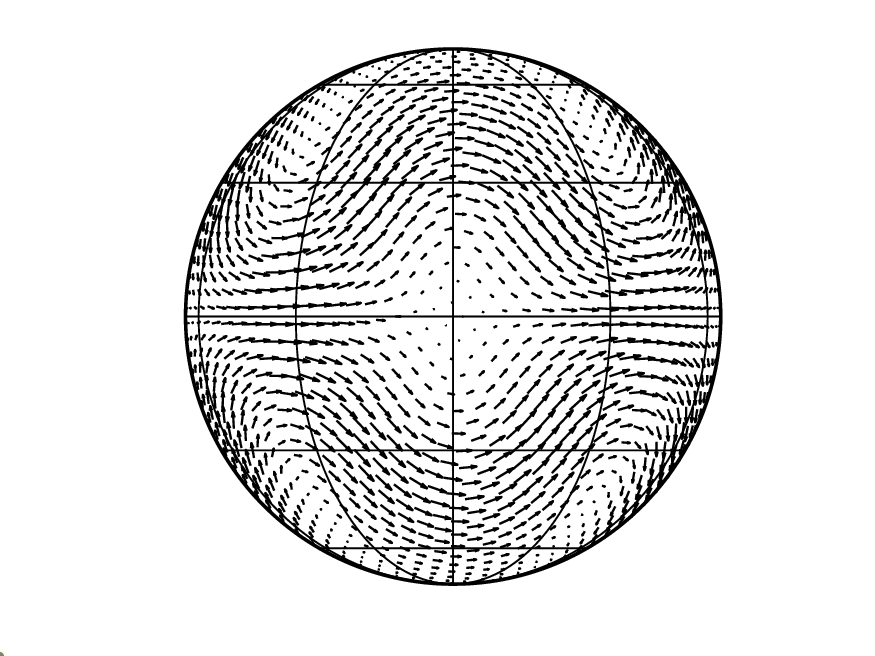}&
			\includegraphics[width=0.3\textwidth, trim=1cm 1cm 1cm 0.5cm, clip]{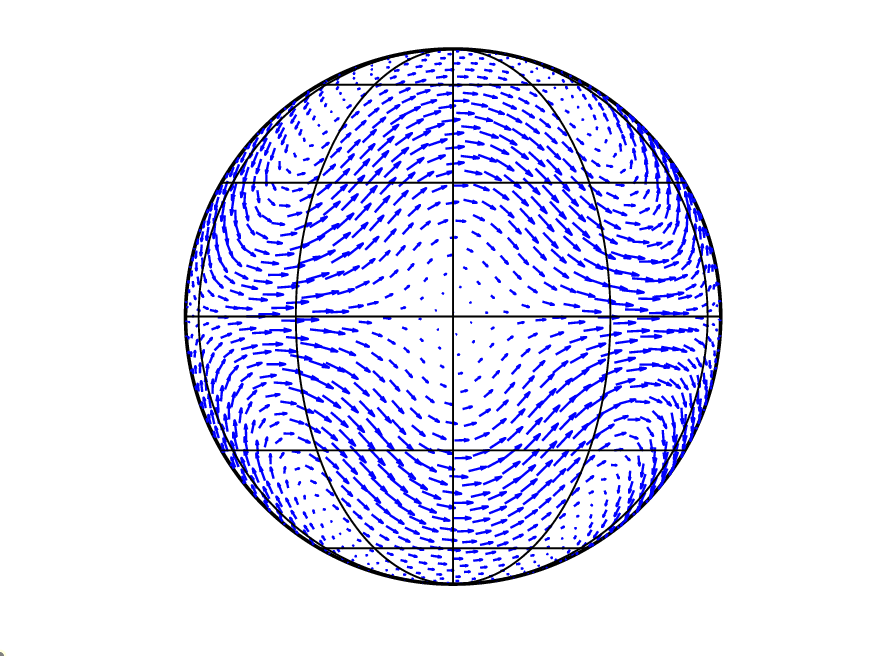}\\
		};
		\node [anchor=north, inner sep=1pt, yshift=-1mm] at (M-1-1.south) {(a)  Stream function  $s_1$};
		
		\node [anchor=north, inner sep=1pt, yshift=-1mm] at (M-1-2.south) {(b) vector field  $\bL_* s_1$};
		
		\node [anchor=north, inner sep=1pt, yshift=-1mm] at (M-1-3.south) {(c) Reconstructed vector field} ;

	\end{tikzpicture}\\
	\hspace{0.3cm}
	
	\begin{tikzpicture}
		\matrix (M) [
		matrix of nodes,
		nodes={inner sep=0pt, anchor=center}, 
		column sep=0.2cm, 
		row sep=0.5cm,   
		] {
			\includegraphics[width=0.3\textwidth, trim=1cm 1cm 0.95cm 0.5cm, clip]{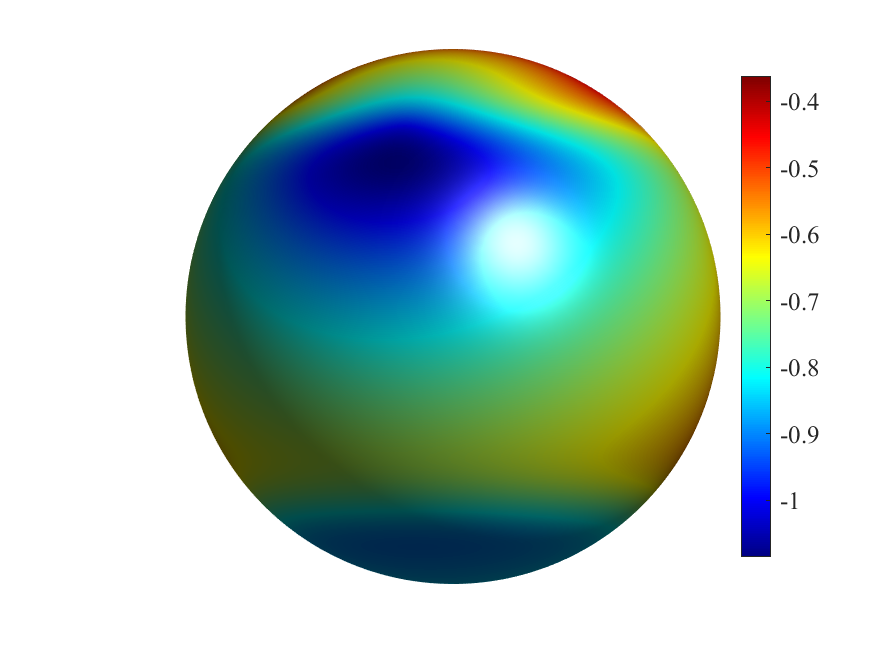} &
			\includegraphics[width=0.3\textwidth, trim=1cm 1cm 1cm 0.5cm, clip]{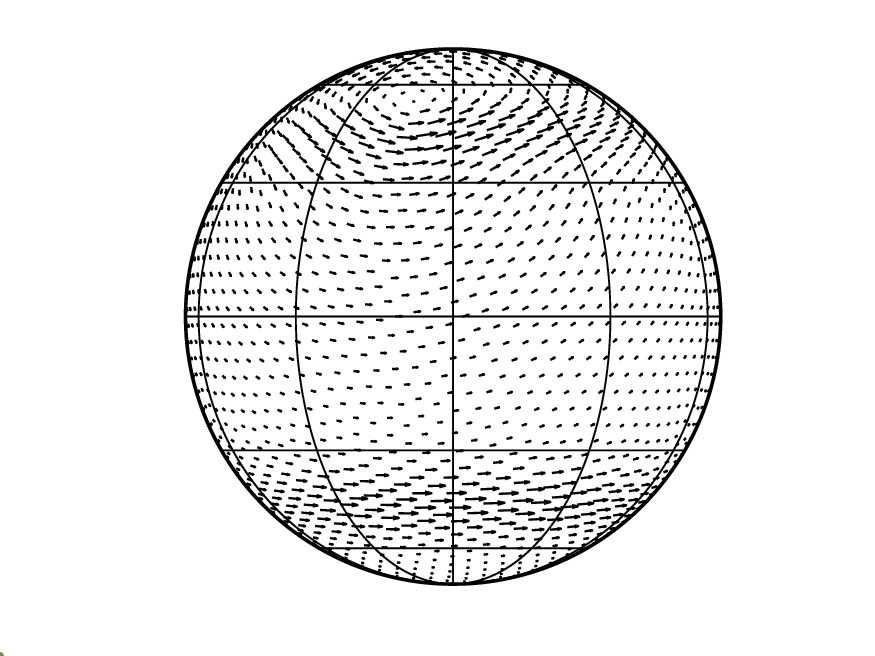}&
			\includegraphics[width=0.3\textwidth, trim=1cm 1cm 1cm 0.5cm, clip]{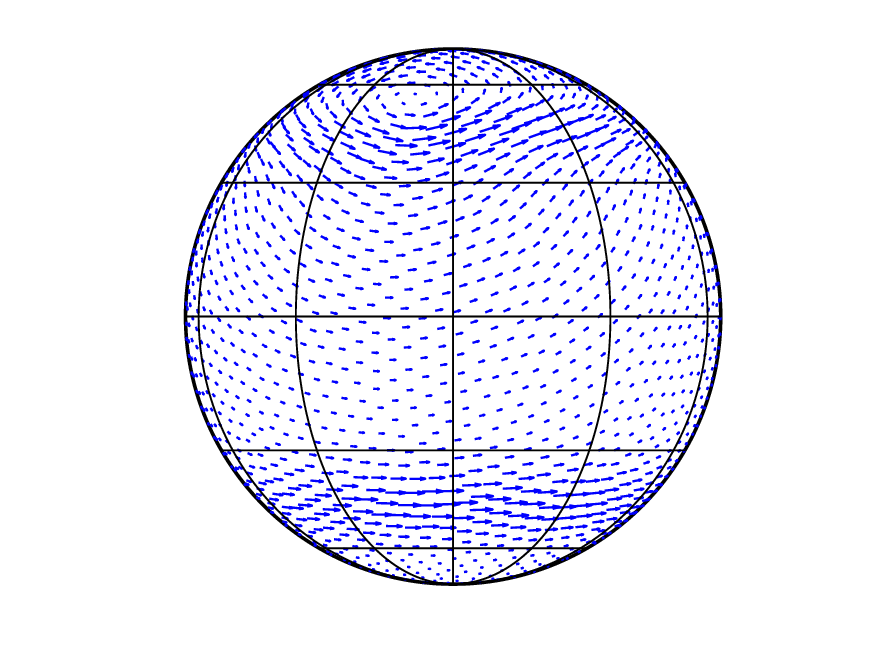}\\
		};
		\node [anchor=north, inner sep=1pt, yshift=-1mm] at (M-1-1.south) {(d) Stream function  $s_2$};
		
		\node [anchor=north, inner sep=1pt, yshift=-1mm] at (M-1-2.south) {(e) vector field  $\bL_* s_2$};
		
		\node [anchor=north, inner sep=1pt, yshift=-1mm] at (M-1-3.south) {(f) Reconstructed vector field } ;
		
	\end{tikzpicture}
	
	\captionsetup{font=normalsize}
\caption{Comparison of the target and reconstructed fields. \textbf{Top row}: Field~1. (a) Stream function  $s_1$; (b) divergence-free vector field $\bL_* s_1$; and (c) reconstructed field using $\mathbf{K}_{\dive}^{0}$ with $N=1849$ ME nodes. \textbf{Bottom row}: Field~2. (d) Stream function  $s_2$; (e) divergence-free vector field $\bL_* s_2$; and (f) reconstructed field using $\mathbf{K}_{\dive}^{0}$ with $N=1849$ ME nodes.}
	\label{field1divcon}
\end{figure}

\begin{figure}
	\centering
	\begin{tikzpicture}
		\matrix (M) [
		matrix of nodes,
		nodes={inner sep=0pt, anchor=center}, 
		column sep=0.2cm, 
		row sep=0.5cm,   
		] {
			\includegraphics[width=0.3\textwidth, trim=3cm 2cm 3cm 2cm, clip]{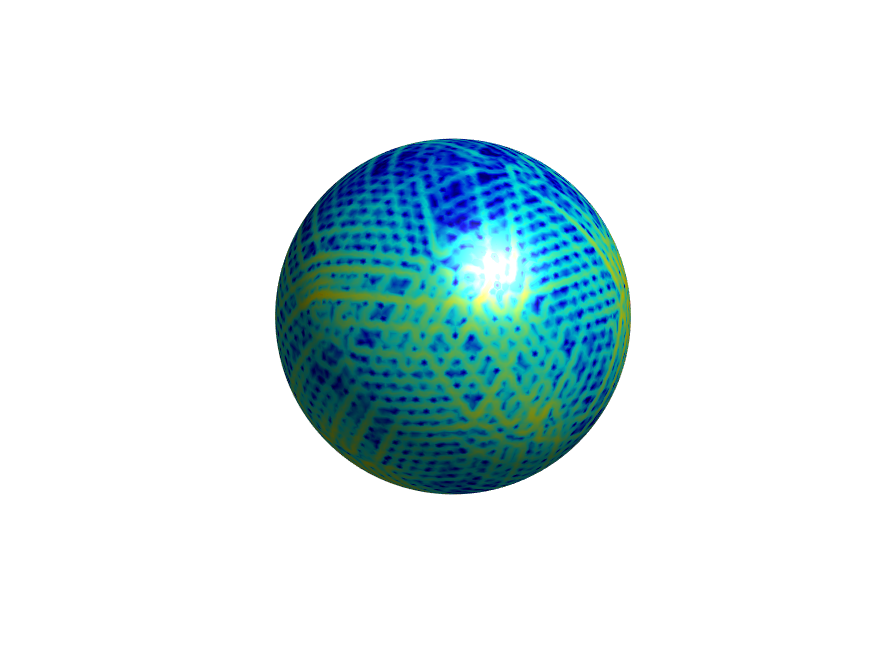} &
			\includegraphics[width=0.3\textwidth, trim=3cm 2cm 3cm 2cm, clip]{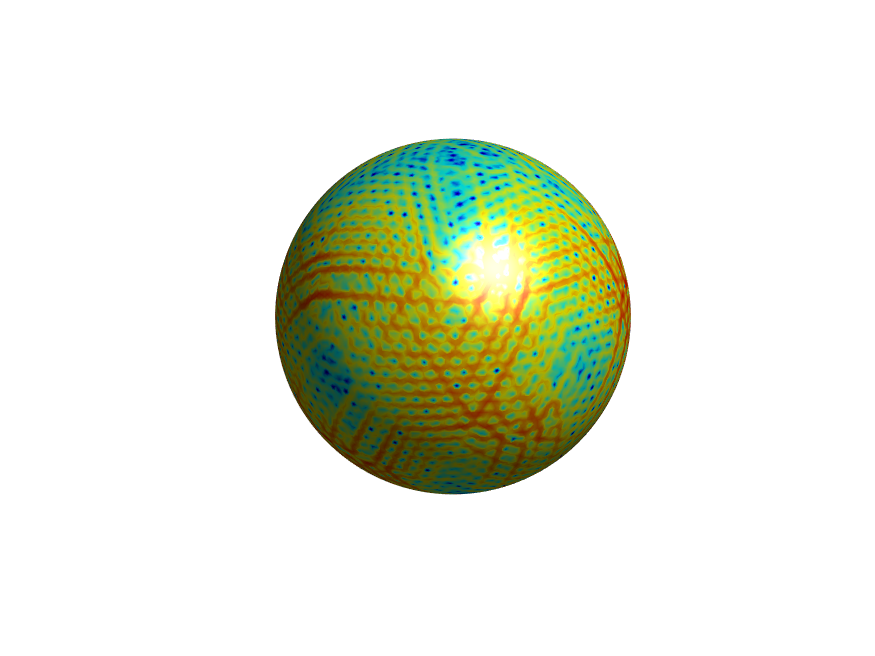}&
			\includegraphics[width=0.3\textwidth, trim=3.8cm 2cm 2.2cm 2cm, clip]{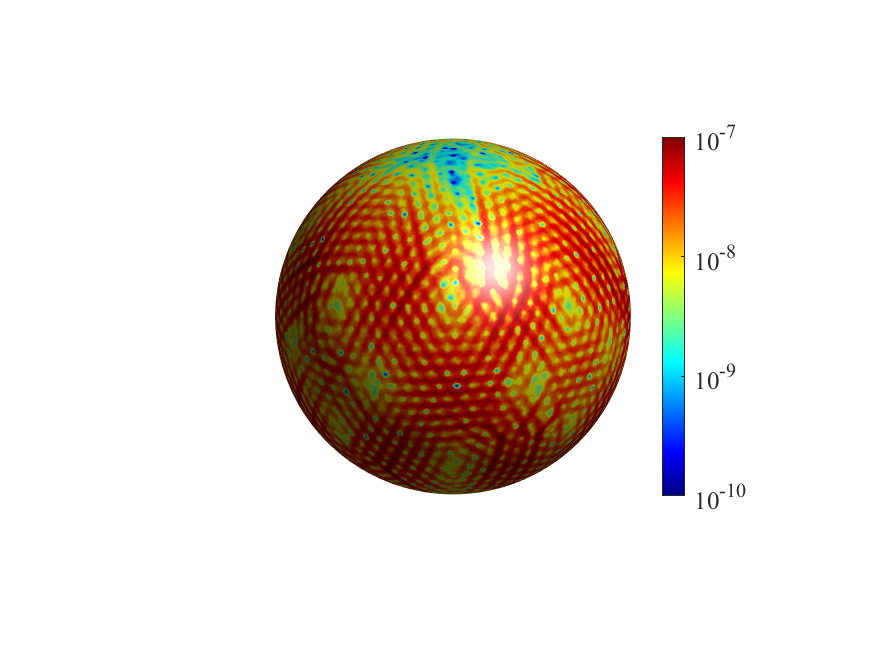}\\
		};
		\node [anchor=north, inner sep=1pt, yshift=-1mm] at (M-1-1.south) {(a) $\mathbf{K}_\dive^{0}$};
		
		\node [anchor=north, inner sep=1pt, yshift=-1mm] at (M-1-2.south) {(b)$\mathbf{K}^{1}_\dive$};
		
		\node [anchor=north, inner sep=1pt, yshift=-1mm] at (M-1-3.south) {(c) $\mathbf{K}_\dive^2$} ;
	\end{tikzpicture}
	
	\captionsetup{font=normalsize}
	\caption{Pointwise errors for approximating Field~1 using the three matrix-valued kernels. All experiments use $N=1849$ ME interpolation nodes and $M=20164$ evaluation nodes.}
	\label{field1divallerr}
\end{figure}

\subsection{Spectral stability of the interpolation matrices}

We assess the spectral stability of the interpolation matrices by examining
their smallest eigenvalues. The ME and Fibonacci node
sets considered here are quasi-uniform, so $h_X\asymp q_X$. Consequently, the
asymptotic rates can be expressed in terms of either quantity. For each node
set, we assemble the $2N\times 2N$ interpolation matrix in local orthonormal
tangent frames and compute its smallest eigenvalue.

The top row of Figure~\ref{fig.field1diveig} shows the results for the
divergence-free kernels $\mathbf{K}_{\dive}^{0}$,
$\mathbf{K}_{\dive}^{1}$, and $\mathbf{K}_{\dive}^{2}$ generated from the
Mat\'ern kernel $\mathrm{MA}_{7/2}$ on the ME and Fibonacci node sets. The
computed smallest eigenvalues decay approximately as
$\mathcal{O}(h_X^{9})$, $\mathcal{O}(h_X^{7})$, and
$\mathcal{O}(h_X^{5})$, respectively. These exponents are consistent with the
theoretical lower bound
$\lambda_{\min}\gtrsim q_X^{2\sigma-2}$
established in Corollary~\ref{cor:smallest-eigenvalue-rate}, where $\sigma=11/2$, $9/2$, and $7/2$ for $\mathbf{K}_{\dive}^{0}$, $\mathbf{K}_{\dive}^{1}$, and $\mathbf{K}_{\dive}^{2}$, respectively.

For Wendland kernels, the positive definiteness of the resulting
matrix-valued kernel depends on the dimension for which the underlying scalar
kernel is positive definite. As established in \cite{Sun_2026_error}, the
construction of $\mathbf{K}_{\dive}^{0}$ requires positive definiteness of the
underlying Wendland kernel in a higher dimension than the constructions of
$\mathbf{K}_{\dive}^{1}$ and $\mathbf{K}_{\dive}^{2}$. To examine this
requirement numerically, we construct $\mathbf{K}_{\dive}^{0}$ from the
Wendland kernels $\varphi_{3,3}$, $\varphi_{5,3}$, and $\varphi_{7,3}$ using the common
shape parameter $\varepsilon=4/3$. The bottom row of
Figure~\ref{fig.field1diveig} shows the smallest eigenvalues of the resulting
interpolation matrices. At $N=5041$, the smallest eigenvalues associated with
$\varphi_{3,3}$ and $\varphi_{5,3}$ are negative for both the ME and Fibonacci node
sets and are therefore omitted from the logarithmic plots. These results are
consistent with the stronger dimensional positive-definiteness requirement
for $\mathbf{K}_{\dive}^{0}$.

Finally, we examine the multiplier-preserving construction established in
Theorem~\ref{thm:special-zonal-construction}. We consider the three Wendland
kernels $\varphi_{3,3}$, $\varphi_{5,3}$, and $\varphi_{7,3}$, which correspond to
increasing values of the dimension parameter $d$. Table~\ref{tab:kappadef}
lists explicit formulas for the scalar zonal kernels $\phi(t)$ and the
corresponding auxiliary functions $\psi(t)$ used to define the divergence-free
matrix-valued kernel
\[
\bfK_{\mathrm{div}}(\bx,\by)
=\psi'(t)\mathbf{Q}(\bx,\by)+\psi(t)\mathbf{R}(\bx,\by).
\]
As shown in Figure~\ref{fig.testkappaeigfu}, the computed smallest eigenvalues
remain positive for all tested node-set sizes and for all three kernels on
both the ME and Fibonacci node sets. This behavior is consistent with the
positive-definiteness property guaranteed by
Theorem~\ref{thm:special-zonal-construction}. Moreover, the smallest
eigenvalues decay approximately as $\mathcal{O}(q_X^{7})$, consistent with
the rate given in Corollary~\ref{cor:smallest-eigenvalue-rate}.

 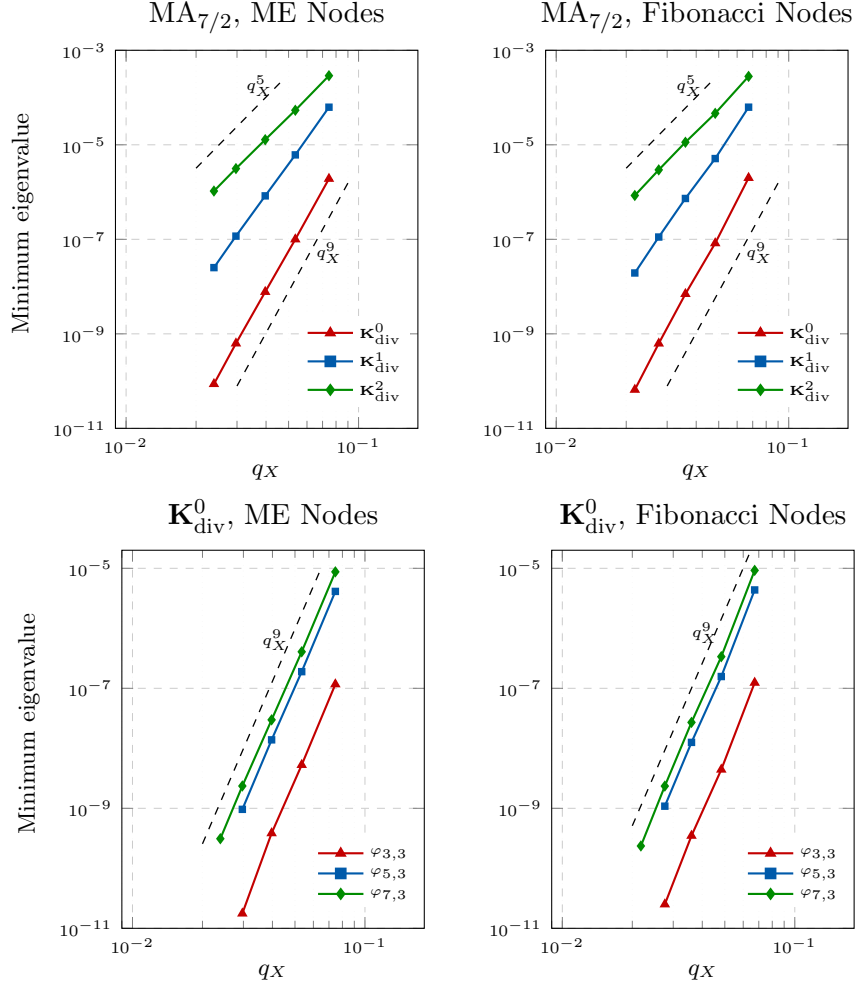
\begin{figure}
 	\centering
 	\begin{tikzpicture}
 		\begin{groupplot}[
 			group style={
 				group size=1 by 1, 
 				horizontal sep=0pt, 
 				vertical sep=0pt,
 			},
 			width=4cm, height=5cm, 
 			scale only axis,           
 			xmode=log, ymode=log,
 			xmin=9e-3, xmax=0.18,
 			ymin=1e-11, ymax=1e-3,
 			grid=both,
 			major grid style={line width=0.2pt, draw=gray!40, dashed},
 			minor grid style={line width=0.1pt, draw=gray!15, dotted},
 			minor x tick num=3, minor y tick num=3,
 			ticklabel style={font=\tiny},
 			xlabel={$q_X$},
 			xlabel style={font=\small, yshift=5pt},
 			legend style={
 				at={(0.98,0.05)}, 
 				anchor=south east,
 				font=\fontsize{4}{5}\selectfont, 
 				cells={anchor=west},
 				legend columns=1,
 				inner sep=1pt,
 				outer sep=1pt,
 				draw=none, 
 				fill opacity=0.8,
 				text opacity=1
 			},
 			legend image post style={mark size=1.8pt}, 
 			cycle list={
 				{color={myred}, mark=triangle*, line width=0.8pt,mark size=1.5pt},
 				{color={myblue}, mark=square*, line width=0.8pt,mark size=1pt},
 				{color={mygreen}, mark=diamond*, line width=0.8pt,mark size=1.5pt},
 				{color={myred}, mark=triangle*, line width=0.8pt, mark size=1.5pt, dashed, mark options={solid}},
 				{color={myblue}, mark=square*, line width=0.8pt, mark size=1pt, dashed, mark options={solid}},
 				{color={mygreen}, mark=diamond*, line width=0.8pt, mark size=1.5pt, dashed, mark options={solid}},
 			},
 			]
			
 			\nextgroupplot[
 			ylabel={Minimum eigenvalue }, 
 			ylabel style={font=\small, xshift=6pt},
 			title={$\mathrm{MA}_{7/2}$, ME Nodes},
 			title style={font=\large, yshift=-3pt},
 			ytick={1e-11,1e-9,1e-7,1e-5,1e-3}, 
 			yticklabels={$10^{-11}$,$10^{-9}$,$10^{-7}$,$10^{-5}$,$10^{-3}$},]
			
 			\addplot table[x index=0, y index=1, col sep=space,] {data_surface/methoddiveig.txt};
 			\addlegendentry{$\bfK_\dive^{0}$}
			
 			\addplot table[x index=0, y index=2, col sep=space] {data_surface/methoddiveig.txt};
 			\addlegendentry{$\bfK_\dive^{1}$}
			
 			\addplot table[x index=0, y index=3, col sep=space] {data_surface/methoddiveig.txt};
 			\addlegendentry{$\bfK_\dive^{2}$}

 			\addplot[
 			domain=0.03:0.09,
 			samples=2,
 			color=black,
 			line width=0.5pt,  %
 			dashed,
 			forget plot,
 			] {4e3* x^9};
			
 			\node[anchor=south west, font=\tiny] at (axis cs:0.06,2e-8) {$q_X^{9}$};
			
 			\addplot[
 			domain=0.02:0.048,
 			samples=2,
 			color=black,
 			line width=0.5pt,  %
 			dashed,
 			forget plot,
 			] {1e3* x^5};
			
 			\node[anchor=south west, font=\tiny] at (axis cs:0.03,7e-5) {$q_X^{5}$};
 		\end{groupplot}
 	\end{tikzpicture}
 	\hspace*{0.5cm}
 	\begin{tikzpicture}
 		\begin{groupplot}[
 			group style={
 				group size=1 by 1, 
 				horizontal sep=0pt, 
 				vertical sep=0pt,
 			},
 			width=4cm, height=5cm, 
 			scale only axis,           
 			xmode=log, ymode=log,
 			xmin=9e-3, xmax=0.18,
 			ymin=1e-11, ymax=1e-3,
 			grid=both,
 			major grid style={line width=0.2pt, draw=gray!40, dashed},
 			minor grid style={line width=0.1pt, draw=gray!15, dotted},
 			minor x tick num=3, minor y tick num=3,
 			ticklabel style={font=\tiny},
 			xlabel={$q_X$},
 			xlabel style={font=\small, yshift=5pt},
 			legend style={
 				at={(0.98,0.05)}, 
 				anchor=south east,
 				font=\fontsize{4}{5}\selectfont, 
 				cells={anchor=west},
 				legend columns=1,
 				inner sep=1pt,
 				outer sep=1pt,
 				draw=none, 
 				fill opacity=0.8,
 				text opacity=1
 			},
 			legend image post style={mark size=1.8pt}, 
 			cycle list={
 				{color={myred}, mark=triangle*, line width=0.8pt,mark size=1.5pt},
 				{color={myblue}, mark=square*, line width=0.8pt,mark size=1pt},
 				{color={mygreen}, mark=diamond*, line width=0.8pt,mark size=1.5pt},
 				{color={myred}, mark=triangle*, line width=0.8pt, mark size=1.5pt, dashed, mark options={solid}},
 				{color={myblue}, mark=square*, line width=0.8pt, mark size=1pt, dashed, mark options={solid}},
 				{color={mygreen}, mark=diamond*, line width=0.8pt, mark size=1.5pt, dashed, mark options={solid}},
 			},
 			]
			
 			\nextgroupplot[
 			title={$\mathrm{MA}_{7/2}$, Fibonacci Nodes},
 			title style={font=\large, yshift=-3pt},
 			ytick={1e-11,1e-9,1e-7,1e-5,1e-3}, 
 			yticklabels={$10^{-11}$,$10^{-9}$,$10^{-7}$,$10^{-5}$,$10^{-3}$},]
			
 			\addplot table[x index=0, y index=1, col sep=space,] {data_surface/methoddiveigham.txt};
 			\addlegendentry{$\bfK_\dive^{0}$}
			
 			\addplot table[x index=0, y index=2, col sep=space,] {data_surface/methoddiveigham.txt};
 			\addlegendentry{$\bfK_\dive^{1}$}
			
 			\addplot table[x index=0, y index=3, col sep=space] {data_surface/methoddiveigham.txt};
 			\addlegendentry{$\bfK_\dive^{2}$}
			
 			\addplot[
 			domain=0.03:0.09,
 			samples=2,
 			color=black,
 			line width=0.5pt,  %
 			dashed,
 			forget plot,
 			] {4e3* x^9};
			
 			\node[anchor=south west, font=\tiny] at (axis cs:0.06,2e-8) {$q_X^{9}$};
			
 			\addplot[
 			domain=0.02:0.048,
 			samples=2,
 			color=black,
 			line width=0.5pt,  %
 			dashed,
 			forget plot,
 			] {1e3* x^5};
			
 			\node[anchor=south west, font=\tiny] at (axis cs:0.03,7e-5) {$q_X^{5}$};
			
 		\end{groupplot}
 	\end{tikzpicture}\\
 	\begin{tikzpicture}
 		\begin{groupplot}[
 			group style={
 				group size=1 by 1, 
 				horizontal sep=0pt, 
 				vertical sep=0pt,
 			},
 			width=4cm, height=5cm, 
 			scale only axis,           
 			xmode=log, ymode=log,
 			xmin=9e-3, xmax=0.18,
 			ymin=1e-11, ymax=2e-5,
 			grid=both,
 			major grid style={line width=0.2pt, draw=gray!40, dashed},
 			minor grid style={line width=0.1pt, draw=gray!15, dotted},
 			minor x tick num=3, minor y tick num=3,
 			ticklabel style={font=\tiny},
 			xlabel={$q_X$},
 			xlabel style={font=\small, yshift=5pt},
 			legend style={
 				at={(0.98,0.05)}, 
 				anchor=south east,
 				font=\fontsize{4}{5}\selectfont, 
 				cells={anchor=west},
 				legend columns=1,
 				inner sep=1pt,
 				outer sep=1pt,
 				draw=none, 
 				fill opacity=0.8,
 				text opacity=1
 			},
 			legend image post style={mark size=1.8pt}, 
 			cycle list={
 				{color={myred}, mark=triangle*, line width=0.8pt,mark size=1.5pt},
 				{color={myblue}, mark=square*, line width=0.8pt,mark size=1pt},
 				{color={mygreen}, mark=diamond*, line width=0.8pt,mark size=1.5pt},
 			},
 			]
			
 			\nextgroupplot[
 			ylabel={Minimum eigenvalue }, 
 			ylabel style={font=\small, xshift=6pt},
 			title={$\bfK_{\dive}^{0}$, ME Nodes},
 			title style={font=\large, yshift=-3pt},
 			ytick={1e-11,1e-9,1e-7,1e-5}, 
 			yticklabels={$10^{-11}$,$10^{-9}$,$10^{-7}$,$10^{-5}$},]
			
 			\addplot table[x index=0, y index=1, col sep=space,skip coords between index={7}{8}] {data_surface/testeig.txt};
 			\addlegendentry{$\varphi_{3,3}$}
			
 			\addplot table[x index=0, y index=2, col sep=space] {data_surface/testeig.txt};
 			\addlegendentry{$\varphi_{5,3}$}
			
 			\addplot table[x index=0, y index=3, col sep=space] {data_surface/testeig.txt};
 			\addlegendentry{$\varphi_{7,3}$}

 			\addplot[
 			domain=0.02:0.065,
 			samples=2,
 			color=black,
 			line width=0.5pt,  %
 			dashed,
 			forget plot,
 			] {5e5 * x^9};
			
 			\node[anchor=south west, font=\tiny] at (axis cs:0.033,3e-7) {$q_{X}^{9}$};
 		\end{groupplot}
 	\end{tikzpicture}
 	\hspace*{0.5cm}
 	\begin{tikzpicture}
 		\begin{groupplot}[
 			group style={
 				group size=1 by 1, 
 				horizontal sep=0pt, 
 				vertical sep=0pt,
 			},
 			width=4cm, height=5cm, 
 			scale only axis,           
 			xmode=log, ymode=log,
 			xmin=9e-3, xmax=0.18,
 			ymin=1e-11, ymax=2e-5,
 			grid=both,
 			major grid style={line width=0.2pt, draw=gray!40, dashed},
 			minor grid style={line width=0.1pt, draw=gray!15, dotted},
 			minor x tick num=3, minor y tick num=3,
 			ticklabel style={font=\tiny},
 			xlabel={$q_{X}$},
 			xlabel style={font=\small, yshift=5pt},
 			legend style={
 				at={(0.98,0.05)}, 
 				anchor=south east,
 				font=\fontsize{4}{5}\selectfont, 
 				cells={anchor=west},
 				legend columns=1,
 				inner sep=1pt,
 				outer sep=1pt,
 				draw=none, 
 				fill opacity=0.8,
 				text opacity=1
 			},
 			legend image post style={mark size=1.8pt}, 
 			cycle list={
 				{color={myred}, mark=triangle*, line width=0.8pt,mark size=1.5pt},
 				{color={myblue}, mark=square*, line width=0.8pt,mark size=1pt},
 				{color={mygreen}, mark=diamond*, line width=0.8pt,mark size=1.5pt},
 				{color={myred}, mark=triangle*, line width=0.8pt, mark size=1.5pt, dashed, mark options={solid}},
 				{color={myblue}, mark=square*, line width=0.8pt, mark size=1pt, dashed, mark options={solid}},
 				{color={mygreen}, mark=diamond*, line width=0.8pt, mark size=1.5pt, dashed, mark options={solid}},
 			},
 			]
			
 			\nextgroupplot[
 			title={$\bfK_{\dive}^{0}$, Fibonacci Nodes},
 			title style={font=\large, yshift=-3pt},
 			ytick={1e-11,1e-9,1e-7,1e-5}, 
 			yticklabels={$10^{-11}$,$10^{-9}$,$10^{-7}$,$10^{-5}$},]
			
 			\addplot table[x index=0, y index=1, col sep=space] {data_surface/testeigham.txt};
 			\addlegendentry{$\varphi_{3,3}$}
			
 			\addplot table[x index=0, y index=2, col sep=space] {data_surface/testeigham.txt};
 			\addlegendentry{$\varphi_{5,3}$}
			
 			\addplot table[x index=0, y index=3, col sep=space] {data_surface/testeigham.txt};
 			\addlegendentry{$\varphi_{7,3}$}
			
 			\addplot[
 			domain=0.02:0.065,
 			samples=2,
 			color=black,
 			line width=0.5pt,  %
 			dashed,
 			forget plot,
 			] {1e6 * x^9};
			
 			\node[anchor=south west, font=\tiny] at (axis cs:0.033,4e-7) {$q_{X}^{9}$};
 		\end{groupplot}
 	\end{tikzpicture}
 	\captionsetup{font=normalsize}
     \caption{Minimum eigenvalues of the divergence-free interpolation matrices. \textbf{Top}: $\mathbf{K}_{\dive}^{0}$, $\mathbf{K}_{\dive}^{1}$, $\mathbf{K}_{\dive}^{2}$ with \textbf{MA$_{7/2}$} on ME (left) and Fibonacci (right) nodes. \textbf{Bottom}: $\mathbf{K}_{\dive}^{0}$ with Wendland kernels $\varphi_{3,3}$, $\varphi_{5,3}$, $\varphi_{7,3}$ on ME (left) and Fibonacci (right) nodes.}

 \label{fig.field1diveig}
 \end{figure}

 \begin{table}[t]
 	\centering
 	\caption{Zonal kernels $\phi$ and corresponding auxiliary functions $\psi$
 		used in the multiplier-preserving construction. Here
 		$r:=\rho^{-1}\sqrt{2-2t}$ denotes the scaled chordal distance.}
 	\label{tab:kappadef}
 	\renewcommand{\arraystretch}{1.25}
 	\setlength{\tabcolsep}{8pt}
 	\small
 	\begin{tabular}{@{}c l@{}}
 		\toprule
 		\textbf{Kernel} & \multicolumn{1}{c}{\textbf{Definitions of $\phi$ and $\psi$}} \\
 		\midrule
		
 		$\mathrm{WE}_{3,3}$
 		&
 		$\begin{aligned}
 			\phi(t)
 			&=
 			\frac{78}{7\pi\rho^2}(1-r)_+^8
 			\left(32r^3+25r^2+8r+1\right), \\[3pt]
 			\psi(t)
 			&=
 			\frac{1+t}{4\pi(1-t^2)}
 			-\frac{(1-r)_+^9}{2\pi(1-t^2)}
 			\left(
 			\frac{384}{7}r^4+\frac{453}{7}r^3
 			+\frac{237}{7}r^2+9r+1
 			\right).
 		\end{aligned}$
 		\\
 		\addlinespace[5pt]
 		\midrule
 		\addlinespace[5pt]
		
 		$\mathrm{WE}_{5,3}$
 		&
 		$\begin{aligned}
 			\phi(t)
 			&=
 			\frac{13}{15\pi\rho^2}(1-r)_+^9
 			\left(693r^3+477r^2+135r+15\right), \\[3pt]
 			\psi(t)
 			&=
 			\frac{1+t}{4\pi(1-t^2)}
 			-\frac{(1-r)_+^{10}}{2\pi(1-t^2)}
 			\left(
 			\frac{429}{5}r^4+90r^3
 			+42r^2+10r+1
 			\right).
 		\end{aligned}$
 		\\
 		\addlinespace[5pt]
 		\midrule
 		\addlinespace[5pt]
		
 		$\mathrm{WE}_{7,3}$
 		&
 		$\begin{aligned}
 			\phi(t)
 			&=
 			\frac{1}{\pi\rho^2}(1-r)_+^{10}
 			\left(960r^3+591r^2+150r+15\right), \\[3pt]
 			\psi(t)
 			&=
 			\frac{1+t}{4\pi(1-t^2)}
 			-\frac{(1-r)_+^{11}}{2\pi(1-t^2)}
 			\left(
 			128r^4+121r^3+51r^2+11r+1
 			\right).
 		\end{aligned}$
 		\\
		
 		\bottomrule
 	\end{tabular}
 \end{table}

 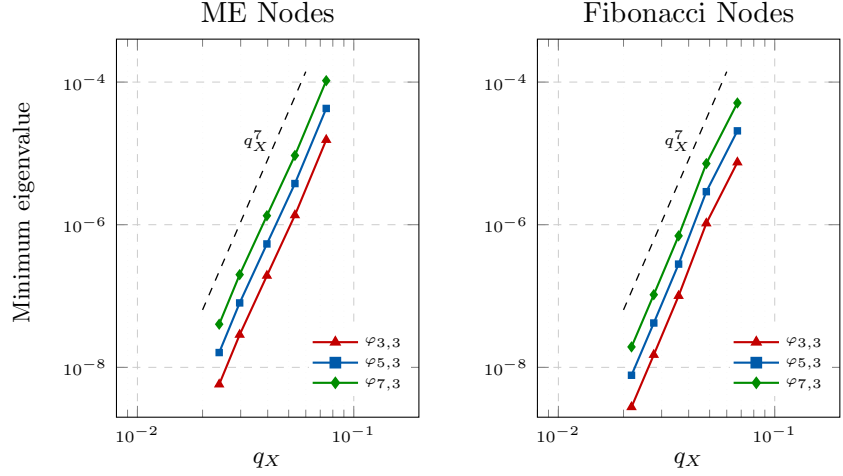
\begin{figure}
 	\centering
 	\begin{tikzpicture}
 		\begin{groupplot}[
 			group style={
 				group size=1 by 1, 
 				horizontal sep=0pt, 
 				vertical sep=0pt,
 			},
 			width=4cm, height=5cm, 
 			scale only axis,           
 			xmode=log, ymode=log,
 			xmin=0.008, xmax=0.2,
 			ymin=2e-9, ymax=4e-4,
 			grid=both,
 			major grid style={line width=0.2pt, draw=gray!40, dashed},
 			minor grid style={line width=0.1pt, draw=gray!15, dotted},
 			minor x tick num=3, minor y tick num=3,
 			ticklabel style={font=\tiny},
 			xlabel={$q_X$},
 			xlabel style={font=\small, yshift=5pt},
 			legend style={
 				at={(0.98,0.05)}, 
 				anchor=south east,
 				font=\fontsize{4}{5}\selectfont, 
 				cells={anchor=west},
 				legend columns=1,
 				inner sep=1pt,
 				outer sep=1pt,
 				draw=none, 
 				fill opacity=0.8,
 				text opacity=1
 			},
 			legend image post style={mark size=1.8pt}, 
 			cycle list={
 				{color={myred}, mark=triangle*, line width=0.8pt,mark size=1.5pt},
 				{color={myblue}, mark=square*, line width=0.8pt,mark size=1pt},
 				{color={mygreen}, mark=diamond*, line width=0.8pt,mark size=1.5pt},
 			},
 			]
			
 			\nextgroupplot[
 			ylabel={Minimum eigenvalue }, 
 			ylabel style={font=\small, xshift=6pt},
 			title={ME Nodes},
 			title style={font=\large, yshift=-3pt},
 			ytick={1e-8,1e-6,1e-4}, 
 			yticklabels={$10^{-8}$,$10^{-6}$,$10^{-4}$,}]
			
 			\addplot table[x index=0, y index=1, col sep=space] {data_surface/testkapaeig.txt};
 			\addlegendentry{$\varphi_{3,3}$}
			
 			\addplot table[x index=0, y index=2, col sep=space] {data_surface/testkapaeig.txt};
 			\addlegendentry{$\varphi_{5,3}$}
			
 			\addplot table[x index=0, y index=3, col sep=space] {data_surface/testkapaeig.txt};
 			\addlegendentry{$\varphi_{7,3}$}

 			\addplot[
 			domain=0.02:0.06,
 			samples=2,
 			color=black,
 			line width=0.5pt,  %
 			dashed,
 			forget plot,
 			] {5e4 * x^7};
			
 			\node[anchor=south west, font=\tiny] at (axis cs:0.028,8e-6) {$q_X^{7}$};
 		\end{groupplot}
 	\end{tikzpicture}
 	\hspace*{0.5cm}
 	\begin{tikzpicture}
 		\begin{groupplot}[
 			group style={
 				group size=1 by 1, 
 				horizontal sep=0pt, 
 				vertical sep=0pt,
 			},
 			width=4cm, height=5cm, 
 			scale only axis,           
 			xmode=log, ymode=log,
 			xmin=0.008, xmax=0.2,
 			ymin=2e-9, ymax=4e-4,
 			grid=both,
 			major grid style={line width=0.2pt, draw=gray!40, dashed},
 			minor grid style={line width=0.1pt, draw=gray!15, dotted},
 			minor x tick num=3, minor y tick num=3,
 			ticklabel style={font=\tiny},
 			xlabel={$q_X$},
 			xlabel style={font=\small, yshift=5pt},
 			legend style={
 				at={(0.98,0.05)}, 
 				anchor=south east,
 				font=\fontsize{4}{5}\selectfont, 
 				cells={anchor=west},
 				legend columns=1,
 				inner sep=1pt,
 				outer sep=1pt,
 				draw=none, 
 				fill opacity=0.8,
 				text opacity=1
 			},
 			legend image post style={mark size=1.8pt}, 
 			cycle list={
 				{color={myred}, mark=triangle*, line width=0.8pt,mark size=1.5pt},
 				{color={myblue}, mark=square*, line width=0.8pt,mark size=1pt},
 				{color={mygreen}, mark=diamond*, line width=0.8pt,mark size=1.5pt},
 				{color={myred}, mark=triangle*, line width=0.8pt, mark size=1.5pt, dashed, mark options={solid}},
 				{color={myblue}, mark=square*, line width=0.8pt, mark size=1pt, dashed, mark options={solid}},
 				{color={mygreen}, mark=diamond*, line width=0.8pt, mark size=1.5pt, dashed, mark options={solid}},
 			},
 			]
			
 			\nextgroupplot[
 			title={Fibonacci Nodes},
 			title style={font=\large, yshift=-3pt},
 			ytick={1e-8,1e-6,1e-4}, 
 			yticklabels={$10^{-8}$,$10^{-6}$,$10^{-4}$,}]
			
 			\addplot table[x index=0, y index=1, col sep=space] {data_surface/testkapaeigham.txt};
 			\addlegendentry{$\varphi_{3,3}$}
			
 			\addplot table[x index=0, y index=2, col sep=space,] {data_surface/testkapaeigham.txt};
 			\addlegendentry{$\varphi_{5,3}$}
			
 			\addplot table[x index=0, y index=3, col sep=space] {data_surface/testkapaeigham.txt};
 			\addlegendentry{$\varphi_{7,3}$}

 			\addplot[
 			domain=0.02:0.06,
 			samples=2,
 			color=black,
 			line width=0.5pt,  %
 			dashed,
 			forget plot,
 			] {5e4 * x^7};
			
 			\node[anchor=south west, font=\tiny] at (axis cs:0.028,8e-6) {$q_X^{7}$};
 		\end{groupplot}
 	\end{tikzpicture}
	
 	\captionsetup{font=normalsize}
 	\caption{Minimum eigenvalues of the multiplier-preserving divergence-free interpolation matrices
 		$\mathbf{K}_{\mathrm{div}}(\mathbf{x},\mathbf{y})
 		=\psi'(t)\mathbf{Q}(\mathbf{x},\mathbf{y})
 		+\psi(t)\mathbf{R}(\mathbf{x},\mathbf{y})$
 		constructed from the Wendland kernels $\varphi_{3,3}$, $\varphi_{5,3}$, and $\varphi_{7,3}$ with the common shape parameter $\varepsilon=2$. Results are shown for ME nodes (left) and Fibonacci nodes (right).}
 	\label{fig.testkappaeigfu}
 \end{figure}

\subsection{Simulations on various surfaces}

To illustrate the geometric flexibility of the proposed divergence-free
interpolation framework on surfaces other than the sphere, we consider three
embedded surfaces with different geometric and topological characteristics: a
torus, a red blood cell (RBC) surface, and a bumpy sphere. Detailed definitions
of these surfaces and their associated divergence-free target fields are
provided in Appendix~\ref{sec:def_surfaces}.


For the torus, the interpolation set comprises $N=1600$ Fibonacci-type nodes
generated using a modified golden-section construction. The error is evaluated
on an independent set of $M=5000$ nodes. We construct the matrix-valued kernels
from the Mat\'ern kernel $\mathrm{MA}_{7/2}$ with shape parameter
$\varepsilon=7/3$. Figure~\ref{fig:torus_field}(a) shows the stream function
$s_3$, while Figure~\ref{fig:torus_field}(b) shows the divergence-free
interpolant obtained using $\mathbf{K}_{\dive}^{0}$. The interpolated vectors
remain tangent to the torus and capture the oscillatory structure induced by
the stream function.

For the RBC surface, the interpolation centers are generated by radially
projecting ME nodes onto the surface. We use $N=1600$
interpolation centers and an evaluation set of $M=8100$ projected
ME nodes. Figure~\ref{fig:torus_field}(c) shows the stream function $s_4$,
while Figure~\ref{fig:torus_field}(d) shows the divergence-free interpolant
constructed using the proposed multiplier-preserving kernel
$\mathbf{K}_{\dive}^{0}$.

For the bumpy sphere, we map Fibonacci nodes onto $\Gamma_{\mathrm B}$ using the radial perturbation
defined in Appendix~\ref{sec:def_surfaces}. We use $N=3200$ interpolation
centers and an evaluation set of $M=8100$ nodes. The matrix-valued
kernels are constructed from the Mat\'ern kernel $\mathrm{MA}_{7/2}$ with shape
parameter $\varepsilon=5$. Figure~\ref{fig:torus_field}(e) shows the stream
function $s_5(x,y,z)=xy+yz+zx$, while Figure~\ref{fig:torus_field}(f) shows
the divergence-free interpolant obtained using $\mathbf{K}_{\dive}^{0}$. The
interpolated vectors remain tangent to the bumpy surface and capture the
complex variations induced by the triaxial sinusoidal perturbations. These
results illustrate the applicability of the proposed kernel construction to
surfaces with rapidly varying curvature.


Figure~\ref{fig.surfacefielddiv} shows the convergence behavior of the
divergence-free interpolants on the torus, the RBC surface, and the bumpy
sphere. On the torus and the RBC surface, the errors demonstrate superconvergent
decay. In contrast, convergence is slower on the bumpy sphere. This difference
may reflect the greater complexity of the associated native space arising from
the surface's spatially varying curvature. A detailed analysis of the
relationship between curvature, native-space structure, and convergence is
left for future work.

\begin{figure}
	\centering
	\begin{tikzpicture}
		\matrix (M) [matrix of nodes, nodes={inner sep=0pt, anchor=center}, column sep=0.5cm] {
			\includegraphics[width=0.45\textwidth, trim=2.4cm 2.7cm 2.2cm 2cm, clip]{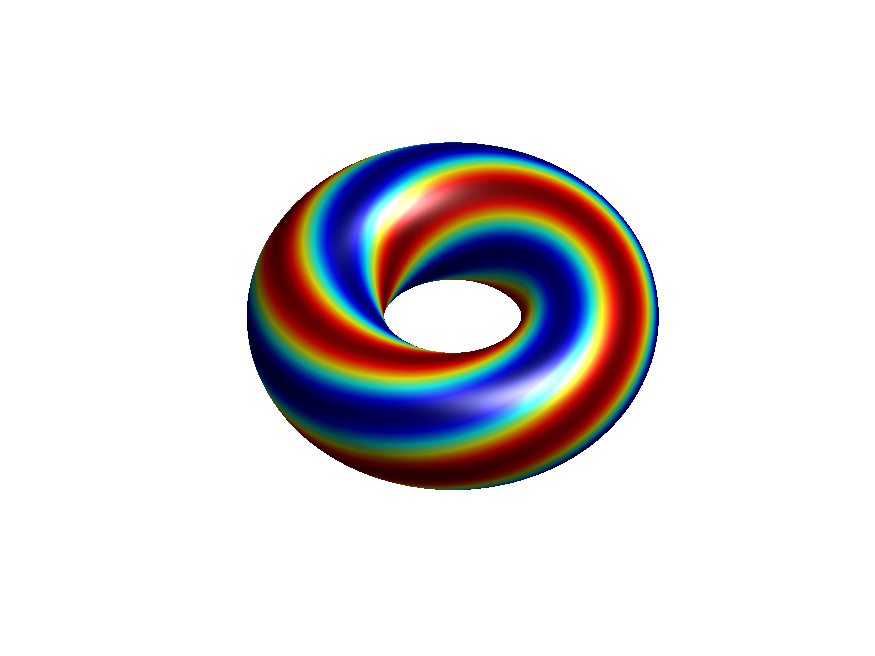} &
			\includegraphics[width=0.45\textwidth, trim=2.4cm 2.7cm 2.2cm 2cm, clip]{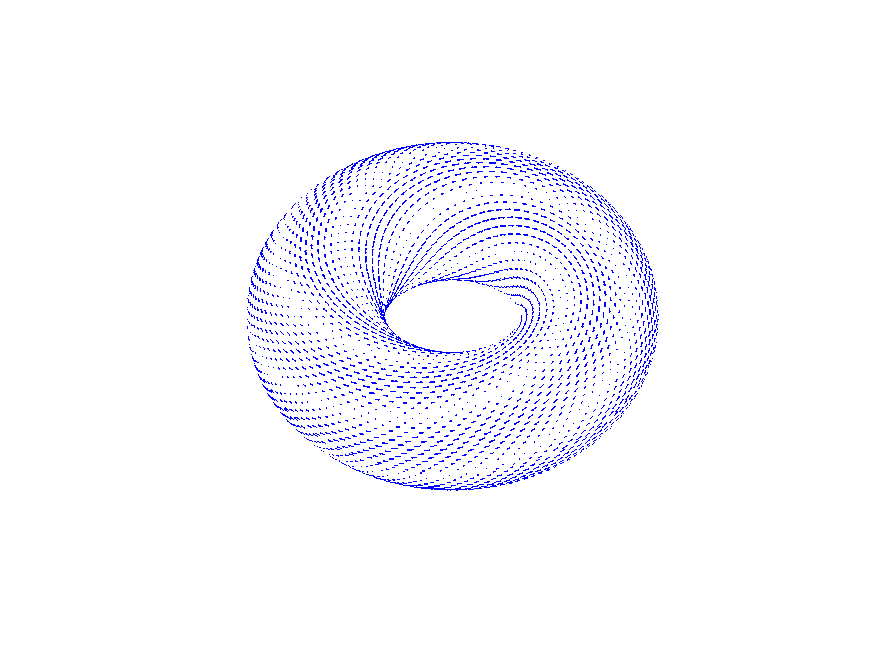} \\
		};
		\node [anchor=north, yshift=-2mm] at (M-1-1.south) {\small (a) Stream function $s_3$};
		\node [anchor=north, yshift=-2mm] at (M-1-2.south) {\small (b) Reconstructed vector field};
	\end{tikzpicture}
	\\
	\begin{tikzpicture}
		\matrix (M) [matrix of nodes, nodes={inner sep=0pt, anchor=center}, column sep=0.5cm] {
			\includegraphics[width=0.45\textwidth, trim=0cm 1.3cm 0cm 1.1cm, clip]{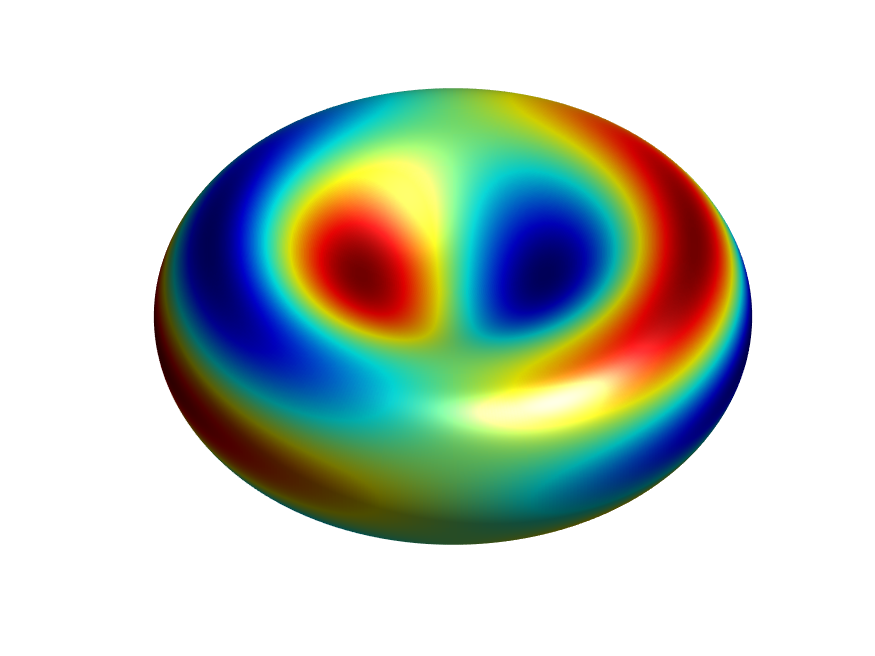} &
			\includegraphics[width=0.45\textwidth, trim=0cm 1.3cm 0cm 1.1cm, clip]{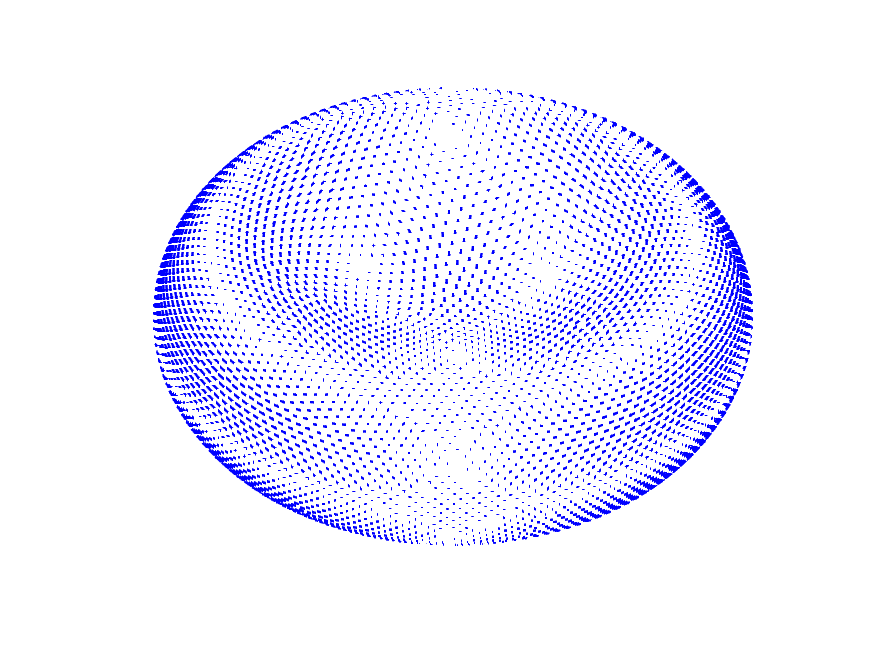} \\
		};
		\node [anchor=north, yshift=-2mm] at (M-1-1.south) {\small (c) Stream function $s_4$};
		\node [anchor=north, yshift=-2mm] at (M-1-2.south) {\small (d) Reconstructed vector field};
	\end{tikzpicture}\\
	\begin{tikzpicture}
			\matrix (M) [matrix of nodes, nodes={inner sep=0pt, anchor=center}, column sep=0.5cm] {
				\includegraphics[width=0.45\textwidth, trim=2.5cm 2.3cm 2.6cm 1.8cm, clip]{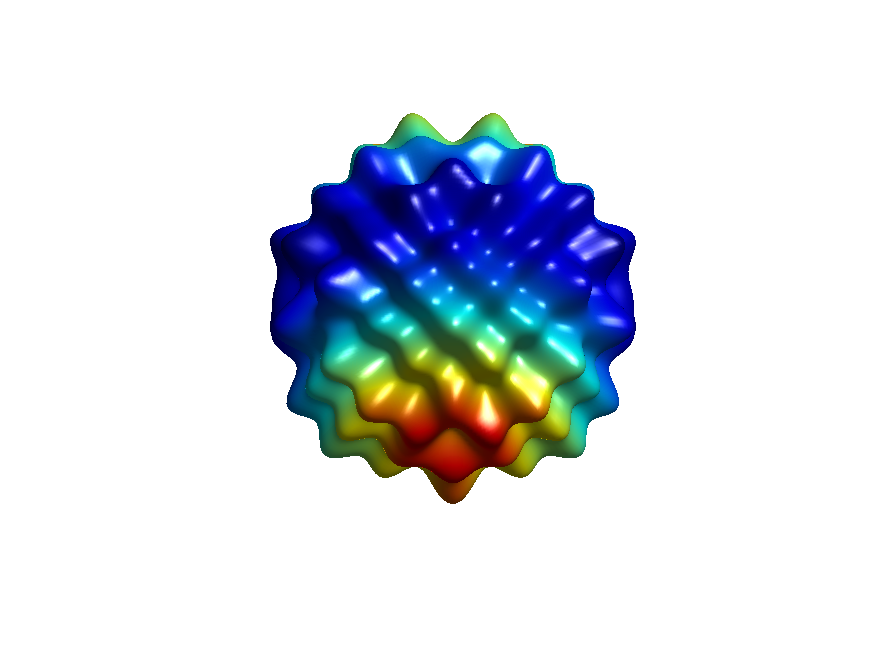} &
				\includegraphics[width=0.45\textwidth, trim=3.5cm 3cm 3.2cm 2.2cm, clip]{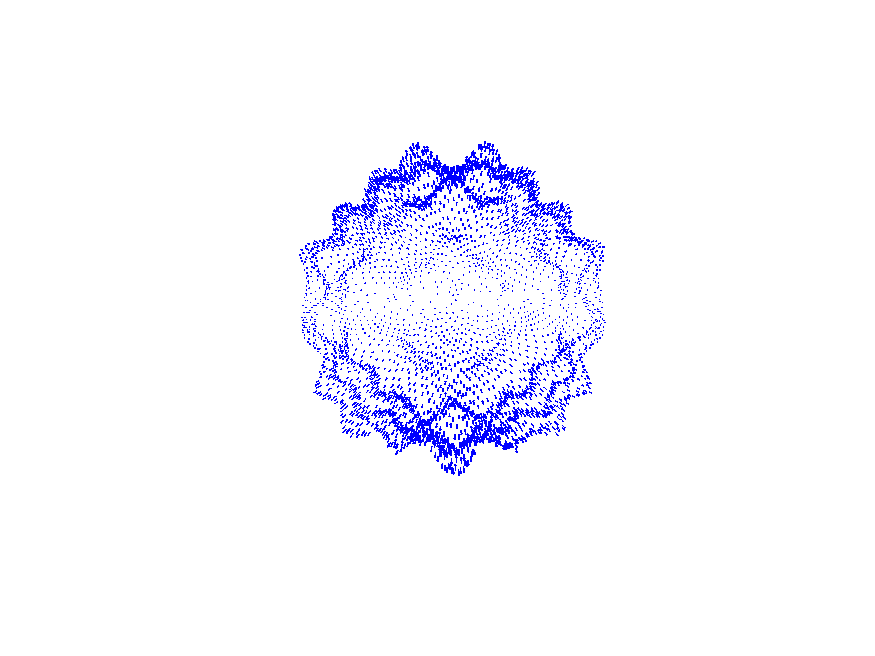} \\
			};
			\node [anchor=north, yshift=-2mm] at (M-1-1.south) {\small (e) Stream function $s_5$};
			\node [anchor=north, yshift=-2mm] at (M-1-2.south) {\small (f) Reconstructed vector field};
		\end{tikzpicture}
	\caption{Stream functions and reconstructed divergence-free vector fields on three embedded surfaces: torus, red blood cell and bumpy sphere. All reconstructions are obtained via $\mathbf{K}_{\dive}^{0}$ .}

	\label{fig:torus_field}
\end{figure}

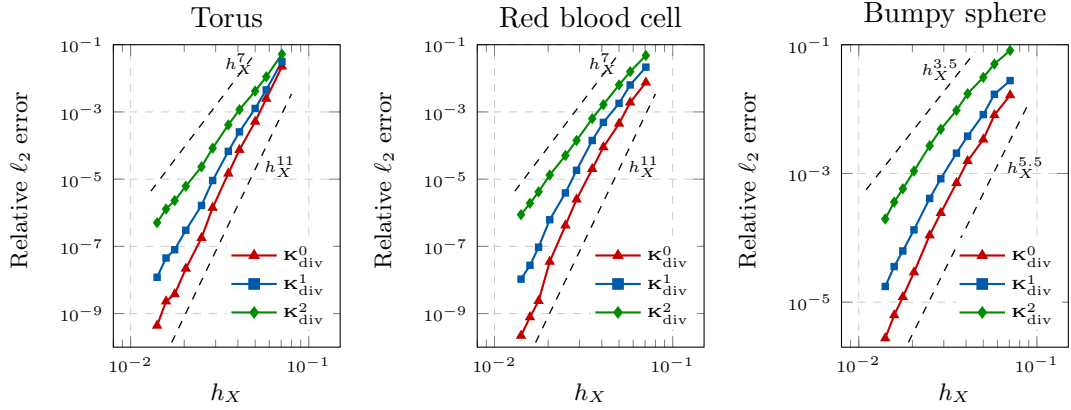
\begin{figure}
	\centering
	\begin{tikzpicture}
		\begin{groupplot}[
			group style={
				group size=1 by 1, 
				horizontal sep=0pt, 
				vertical sep=0pt,
			},
			width=3cm, height=4cm, 
			scale only axis,           
			xmode=log, ymode=log,
			xmin=8e-3, xmax=0.15,
			ymin=1e-10, ymax=1e-1,
			grid=both,
			major grid style={line width=0.2pt, draw=gray!40, dashed},
			minor grid style={line width=0.1pt, draw=gray!15, dotted},
			minor x tick num=3, minor y tick num=3,
			ticklabel style={font=\tiny},
			xlabel={$h_X$},
			xlabel style={font=\small, yshift=5pt},
			legend style={
				at={(0.98,0.05)}, 
				anchor=south east,
				font=\fontsize{4}{5}\selectfont, 
				cells={anchor=west},
				legend columns=1,
				inner sep=1pt,
				outer sep=1pt,
				draw=none, 
				fill opacity=0.8,
				text opacity=1
			},
			legend image post style={mark size=1.8pt}, 
			cycle list={
				{color={myred}, mark=triangle*, line width=0.8pt,mark size=1.5pt},
				{color={myblue}, mark=square*, line width=0.8pt,mark size=1pt},
				{color={mygreen}, mark=diamond*, line width=0.8pt,mark size=1.5pt},
				{color={myred}, mark=triangle*, line width=0.8pt, mark size=1.5pt, dashed, mark options={solid}},
				{color={myblue}, mark=square*, line width=0.8pt, mark size=1pt, dashed, mark options={solid}},
				{color={mygreen}, mark=diamond*, line width=0.8pt, mark size=1.5pt, dashed, mark options={solid}},
			},
			]
			
			\nextgroupplot[
			ylabel={Relative $\ell_2$ error}, 
			ylabel style={font=\small, xshift=6pt},
			title={Torus},
			title style={font=\large, yshift=-3pt},
			ytick={1e-9,1e-7,1e-5,1e-3,1e-1}, 
			yticklabels={$10^{-9}$,$10^{-7}$,$10^{-5}$,$10^{-3}$,$10^{-1}$},]
			
			\addplot table[x index=0, y index=1, col sep=space,] {data_surface/methodtorusdiverrork0.txt};
			\addlegendentry{$\bfK_\dive^{0}$}
			
			\addplot table[x index=0, y index=1, col sep=space] {data_surface/methodtorusdiverrork1.txt};
			\addlegendentry{$\bfK_\dive^{1}$}
			
			\addplot table[x index=0, y index=1, col sep=space] {data_surface/methodtorusdiverrork2.txt};
			\addlegendentry{$\bfK_\dive^{2}$}

			\addplot[
			domain=0.017:0.08,
			samples=2,
			color=black,
			line width=0.5pt,  %
			dashed,
			forget plot,
			] {4e9* x^11};
			
			\node[anchor=south west, font=\tiny] at (axis cs:0.05,6e-6) {$h_X^{11}$};
			
			\addplot[
			domain=0.013:0.05,
			samples=2,
			color=black,
			line width=0.5pt,  %
			dashed,
			forget plot,
			] {7e7* x^7};
			
			\node[anchor=south west, font=\tiny] at (axis cs:0.03,7e-3) {$h_X^{7}$};
		\end{groupplot}
	\end{tikzpicture}
	\hspace*{0.1cm}
	\begin{tikzpicture}
		\begin{groupplot}[
			group style={
				group size=1 by 1, 
				horizontal sep=0pt, 
				vertical sep=0pt,
			},
			width=3cm, height=4cm, 
			scale only axis,           
			xmode=log, ymode=log,
			xmin=8e-3, xmax=0.15,
			ymin=1e-10, ymax=1e-1,
			grid=both,
			major grid style={line width=0.2pt, draw=gray!40, dashed},
			minor grid style={line width=0.1pt, draw=gray!15, dotted},
			minor x tick num=3, minor y tick num=3,
			ticklabel style={font=\tiny},
			xlabel={$h_X$},
			xlabel style={font=\small, yshift=5pt},
			legend style={
				at={(0.98,0.05)}, 
				anchor=south east,
				font=\fontsize{4}{5}\selectfont, 
				cells={anchor=west},
				legend columns=1,
				inner sep=1pt,
				outer sep=1pt,
				draw=none, 
				fill opacity=0.8,
				text opacity=1
			},
			legend image post style={mark size=1.8pt}, 
			cycle list={
				{color={myred}, mark=triangle*, line width=0.8pt,mark size=1.5pt},
				{color={myblue}, mark=square*, line width=0.8pt,mark size=1pt},
				{color={mygreen}, mark=diamond*, line width=0.8pt,mark size=1.5pt},
				{color={myred}, mark=triangle*, line width=0.8pt, mark size=1.5pt, dashed, mark options={solid}},
				{color={myblue}, mark=square*, line width=0.8pt, mark size=1pt, dashed, mark options={solid}},
				{color={mygreen}, mark=diamond*, line width=0.8pt, mark size=1.5pt, dashed, mark options={solid}},
			},
			]
			
			\nextgroupplot[
			ylabel={Relative $\ell_2$ error}, 
			ylabel style={font=\small, xshift=6pt},
			title={Red blood cell},
			title style={font=\large, yshift=-3pt},
			ytick={1e-9,1e-7,1e-5,1e-3,1e-1}, 
			yticklabels={$10^{-9}$,$10^{-7}$,$10^{-5}$,$10^{-3}$,$10^{-1}$},]
			
			\addplot table[x index=0, y index=1, col sep=space,] {data_surface/methodrbcdiverrork0.txt};
			\addlegendentry{$\bfK_\dive^{0}$}
			
			\addplot table[x index=0, y index=1, col sep=space] {data_surface/methodrbcdiverrork1.txt};
			\addlegendentry{$\bfK_\dive^{1}$}
			
			\addplot table[x index=0, y index=1, col sep=space] {data_surface/methodrbcdiverrork2.txt};
			\addlegendentry{$\bfK_\dive^{2}$}

			\addplot[
			domain=0.017:0.08,
			samples=2,
			color=black,
			line width=0.5pt,  %
			dashed,
			forget plot,
			] {4e9* x^11};
			
			\node[anchor=south west, font=\tiny] at (axis cs:0.05,6e-6) {$h_X^{11}$};
			
			\addplot[
			domain=0.013:0.05,
			samples=2,
			color=black,
			line width=0.5pt,  %
			dashed,
			forget plot,
			] {7e7* x^7};
			
			\node[anchor=south west, font=\tiny] at (axis cs:0.03,7e-3) {$h_X^{7}$};
		\end{groupplot}
	\end{tikzpicture}
	\hspace*{0.1cm}
	\begin{tikzpicture}
		\begin{groupplot}[
			group style={
				group size=1 by 1, 
				horizontal sep=0pt, 
				vertical sep=0pt,
			},
			width=3cm, height=4cm, 
			scale only axis,           
			xmode=log, ymode=log,
			xmin=8e-3, xmax=0.15,
			ymin=2e-6, ymax=1e-1,
			grid=both,
			major grid style={line width=0.2pt, draw=gray!40, dashed},
			minor grid style={line width=0.1pt, draw=gray!15, dotted},
			minor x tick num=3, minor y tick num=3,
			ticklabel style={font=\tiny},
			xlabel={$h_X$},
			xlabel style={font=\small, yshift=5pt},
			legend style={
				at={(0.98,0.05)}, 
				anchor=south east,
				font=\fontsize{4}{5}\selectfont, 
				cells={anchor=west},
				legend columns=1,
				inner sep=1pt,
				outer sep=1pt,
				draw=none, 
				fill opacity=0.8,
				text opacity=1
			},
			legend image post style={mark size=1.8pt}, 
			cycle list={
				{color={myred}, mark=triangle*, line width=0.8pt,mark size=1.5pt},
				{color={myblue}, mark=square*, line width=0.8pt,mark size=1pt},
				{color={mygreen}, mark=diamond*, line width=0.8pt,mark size=1.5pt},
				{color={myred}, mark=triangle*, line width=0.8pt, mark size=1.5pt, dashed, mark options={solid}},
				{color={myblue}, mark=square*, line width=0.8pt, mark size=1pt, dashed, mark options={solid}},
				{color={mygreen}, mark=diamond*, line width=0.8pt, mark size=1.5pt, dashed, mark options={solid}},
			},
			]
			
			\nextgroupplot[
			ylabel={Relative $\ell_2$ error}, 
			ylabel style={font=\small, xshift=6pt},
			title={Bumpy sphere},
			title style={font=\large, yshift=-3pt},
			ytick={1e-5,1e-3,1e-1}, 
			yticklabels={$10^{-5}$,$10^{-3}$,$10^{-1}$},]
			
			\addplot table[x index=0, y index=1, col sep=space,] {data_surface/methodbumpydiverr.txt};
			\addlegendentry{$\bfK_\dive^{0}$}
			
			\addplot table[x index=0, y index=2, col sep=space] {data_surface/methodbumpydiverr.txt};
			\addlegendentry{$\bfK_\dive^{1}$}
			
			\addplot table[x index=0, y index=3, col sep=space] {data_surface/methodbumpydiverr.txt};
			\addlegendentry{$\bfK_\dive^{2}$}

			\addplot[
			domain=0.011:0.045,
			samples=2,
			color=black,
			line width=0.5pt,  %
			dashed,
			forget plot,
			] {4e3* x^3.5};
			
			\node[anchor=south west, font=\tiny] at (axis cs:0.02,2e-2) {$h_X^{3.5}$};
			
			\addplot[
			domain=0.019:0.09,
			samples=2,
			color=black,
			line width=0.5pt,  %
			dashed,
			forget plot,
			] {7e3* x^5.5};
			
			\node[anchor=south west, font=\tiny] at (axis cs:0.06,6e-4) {$h_X^{5.5}$};
		\end{groupplot}
	\end{tikzpicture}		
	\captionsetup{font=normalsize}
	\caption{Relative $\ell_2$ errors of the divergence-free approximation on three different surfaces using the restricted \textbf{MA$_{7/2}$} kernel. }
	\label{fig.surfacefielddiv}
\end{figure}

\section{Conclusion}
We have developed a general framework for constructing divergence-free kernels on embedded surfaces. The proposed construction incorporates the surface geometry through normal-dependent cross-product operators while retaining flexibility in the choice of the underlying scalar kernel. In particular, it avoids the repeated application of surface differential operators required by classical potential-based constructions.
On the unit sphere, the vector spherical harmonic representation provides an explicit characterization of the kernel multipliers, positive definiteness, and associated native spaces. We have also established a multiplier-preserving construction for which positive definiteness and native space equivalence follow directly from the Fourier coefficients of the underlying zonal kernel. 

Under an appropriate multiplier decay condition, we derived a stability estimate for the interpolation matrix together with a corresponding condition number bound. We further established pointwise and fractional-order Sobolev error estimates for target fields with both native space and lower regularity. For targets smoother than the native space, a Hilbert-scale argument yields superconvergence estimates that quantify the additional accuracy achieved by the proposed constructions.
The stability and convergence analysis developed in this work relies on the spherical harmonic structure of $\mathbb{S}^2$, whereas the experiments on more general surfaces provide numerical evidence of the broader applicability of the framework. Natural directions for future research include extending the multiplier analysis and error theory to general compact surfaces, establishing complete positive-definiteness criteria for lower-order constructions, and developing scalable preconditioners and localized implementations.



\bibliographystyle{plain}
\bibliography{reference}

\bigskip
\appendix
\section{Proof of Lemma~\ref{lem:filtered-localization}}

To prove Lemma~\ref{lem:filtered-localization}, we first recall a result about uniform filtered Jacobi localization.

\begin{lemma}
\label{lem:uniform-jacobi-localization}
Let $k\in\mathbb N_0$, $0<c_3<c_4<\infty$, and $M>0$.
For the standard Jacobi polynomials $P_\ell^{(k,k)}$, define
\[
\mathfrak h_\ell^{(k)}
:=
\int_{-1}^1
|P_\ell^{(k,k)}(t)|^2(1-t)^k(1+t)^k\,\mathrm{d}t.
\]
Let $J=3\lceil M\rceil-1$. Suppose that
$\{\chi_L\}_{L\ge1}\subset C_c^\infty([0,\infty))$ satisfies
\[
\operatorname{supp}\chi_L\subset[c_3,c_4],
\quad
\sup_{L\ge1}\max_{0\le i\le J}
\|\chi_L^{(i)}\|_{L^\infty}<\infty.
\]
Then there exists a constant $C>0$ such that
\[
\left|
\sum_{\ell=0}^\infty
\chi_L\!\left(\frac{\ell}{L}\right)
\frac{P_\ell^{(k,k)}(1)}{\mathfrak h_\ell^{(k)}}
P_\ell^{(k,k)}(\cos\theta)
\right|
\le
C L^{2k+2}(1+L\theta)^{-M}
\]
for $0\le\theta\le\pi$ and $L\ge1$.
\end{lemma}

\begin{proof}
Let $\varrho=\lceil M\rceil$ and $b=\max\{1,c_4\}$. We set
\[
N_L:=\lceil bL\rceil,
\quad
b_L:=\frac{N_L}{L},
\quad
\varphi_L(u):=\chi_L(b_Lu).
\]
Then, we obtain
\[
\varphi_L\!\left(\frac{\ell}{N_L}\right)
=
\chi_L\!\left(\frac{\ell}{L}\right),
\quad
\operatorname{supp}\varphi_L\subset[0,1].
\]
Because $\varphi_L$ vanishes in a neighborhood of zero, all its
derivatives vanish at zero. Moreover, $b\le b_L\le b+1$, and hence
\[
\sup_L\|\varphi_L^{(i)}\|_{L^\infty}
\le
(b+1)^i\sup_L\|\chi_L^{(i)}\|_{L^\infty}.
\]

The Jacobi normalization gives
\[
\frac{P_\ell^{(k,k)}(1)}{\mathfrak h_\ell^{(k)}}
=
\frac{1}{2^{2k+1}\Gamma(k+1)}
(2\ell+2k+1)
\frac{\Gamma(\ell+2k+1)}{\Gamma(\ell+k+1)}.
\]
Thus, up to a multiplicative constant, the sum
in the statement is the kernel $G_{N_L}^{(k,k)}$ defined in
\cite[(2.6.7)]{Dai_2013book_approximation}, with cutoff $\varphi_L$.
Applying \cite[Theorem~2.6.7]{Dai_2013book_approximation} with
$\alpha=\beta=k$, $n=N_L$, Jacobi derivative order zero, and
decay parameter $\varrho$, yields
\[
|G_{N_L}^{(k,k)}(\cos\theta)|
\le
C_{k,\varrho}\|\varphi_L^{(3\varrho-1)}\|_{L^\infty}
N_L^{2k+2}(1+N_L\theta)^{-\varrho}.
\]
Since $N_L\asymp L$ and $\varrho\ge M$, the right-hand side is bounded by
\[
CL^{2k+2}(1+L\theta)^{-M}.
\]
This completes the proof.
\end{proof}

\begin{proof}[\textbf{Proof of Lemma~\ref{lem:filtered-localization}}]
The vector addition formula gives
\begin{equation}\label{eq:K-Lw-from-H}
\mathbf K_{L,w}(\bx,\by)
=
\calB_{L,\mathsf{w}}''(t)\bfQ(\bx,\by)
+
\calB_{L,\mathsf{w}}'(t)\bfR(\bx,\by),
\quad
t=\bx\cdot\by.
\end{equation}
We then establish uniform bounds for the first two
derivatives of $\calB_{L,\mathsf{w}}$.

For $j=1,2$, the Jacobi differentiation identity is
\[
\frac{d^j}{dt^j}P_\ell(t)
=
2^{-j}(\ell+1)_jP_{\ell-j}^{(j,j)}(t),
\quad (z)_j=\frac{\Gamma(z+j)}{\Gamma(z)},\quad
\ell\ge j,
\]
Put $n=\ell-j$. The standard Jacobi normalization gives
\[
P_n^{(j,j)}(1)
=
\frac{\Gamma(n+j+1)}
{\Gamma(j+1)\Gamma(n+1)}
\]
and
\[
\mathfrak h_n^{(j)}
=
\frac{2^{2j+1}}{2n+2j+1}
\frac{\Gamma(n+j+1)^2}
{\Gamma(n+1)\Gamma(n+2j+1)}.
\]
Consequently,
\[
\frac{P_n^{(j,j)}(1)}{\mathfrak h_n^{(j)}}
=
\frac{2n+2j+1}{2^{2j+1}\Gamma(j+1)}
\frac{\Gamma(n+2j+1)}{\Gamma(n+j+1)}.
\]
Substituting it into the differentiated series yields the following identity
\begin{equation}\label{eq:H-derivative-Jacobi}
\calB_{L,\mathsf{w}}^{(j)}(t)
=
c_jL^{-2}
\sum_{n=0}^\infty
\chi_{L,j}\!\left(\frac nL\right)
\frac{P_n^{(j,j)}(1)}{\mathfrak h_n^{(j)}}
P_n^{(j,j)}(t),
\end{equation}
where
\[
c_j
=
\frac{2^{j-1}\Gamma(j+1)}{\pi},
\quad
\chi_{L,j}(u)
:=
\frac{\mathsf{w}(u+j/L)}
{(u+j/L)(u+(j+1)/L)}.
\]

We next verify that the constants supplied by
Lemma~\ref{lem:uniform-jacobi-localization} are uniform in $L$.
Fix $L_0\ge8(j+1)$. If $L\ge L_0$, then
$\operatorname{supp}\chi_{L,j}\subset[1/8,1]$.
On this support, both denominator factors are bounded below by
$1/4$. Differentiating $\chi_{L,j}$ therefore shows that, for any integer $i\ge0$,
\[
\sup_{L\ge L_0}
\|\chi_{L,j}^{(i)}\|_{L^\infty}
\le
C_{i,j}
\max_{0\le s\le i}
\|\mathsf{w}^{(s)}\|_{L^\infty}.
\]
Thus, the support condition and the required finite collection of smoothness seminorms in
Lemma~\ref{lem:uniform-jacobi-localization} are uniform in $L$.

Applying Lemma~\ref{lem:uniform-jacobi-localization} to
\eqref{eq:H-derivative-Jacobi} with
$k=j, ~M=\nu+2, ~c_3=\frac18, ~c_4=1$ yields
\begin{equation}\label{eq:H-derivative-localization}
|\calB_{L,\mathsf{w}}^{(j)}(\cos\theta)|
\le
C_{\nu,\mathsf{w},j}L^{2j}(1+L\theta)^{-\nu-2},
\quad j=1,2.
\end{equation}

Finally, because
$\|\bfQ(\bx,\by)\|_2=\sin^2\theta
\le\theta^2$, $\|\bfR(\bx,\by)\|_2\le2$,
combining these bounds with \eqref{eq:K-Lw-from-H} and
\eqref{eq:H-derivative-localization} gives
\[
\begin{aligned}
\|\mathbf K_{L,\mathsf{w}}(\bx,\by)\|_2
&\le
C_{\nu,\mathsf{w}}
\left(L^4\theta^2+L^2\right)
(1+L\theta)^{-\nu-2}
\\
&\le
C_{\nu,\mathsf{w}}L^2(1+L\theta)^{-\nu},
\end{aligned}
\]
which completes the proof.
\end{proof}

\bigskip

\section{Proof of Lemma~\ref{lem:packing-estimate}}
\begin{proof}[\textbf{Proof of Lemma~\ref{lem:packing-estimate}}]
By definition of $q_X$, the nodes satisfy
$\dist(\bx_i,\bx_j)\ge 2q_X$ for $i\ne j$; hence, the spherical caps $B(\bx_j,q_X)$ are pairwise disjoint. Fix $i$ and $k\ge 1$. If $kq_X\le \dist(\bx_i,\bx_j)<(k+1)q_X$,
then the cap $B(\bx_j,q_X)$ is contained in the annulus
\[
\left\{
\bz\in \mathbb S^2:\;
(k-1)q_X
\le
\dist(\bx_i,\bz)
<
(k+2)q_X
\right\}.
\]
Comparing the total area of these disjoint caps with that of the annulus yields
\[
\#\big\{
j:\; kq_X\le \dist(\bx_i,\bx_j)<(k+1)q_X
\big\}
\le C(k+1).
\]
This proves the first assertion.

If $Lq_X\ge 1$, the packing estimate gives
\[
\begin{aligned}
\sum_{j\ne i}
\left(1+L\dist(\bx_i,\bx_j)\right)^{-\nu}
&\le
C\sum_{k=1}^{\infty}
(k+1)(1+kLq_X)^{-\nu}  \\
&\le
C(Lq_X)^{-\nu}
\sum_{k=1}^{\infty}
(k+1)k^{-\nu}.
\end{aligned}
\]
The final series converges because $\nu>2$. The estimate is therefore uniform in $i$, which proves the second assertion.
\end{proof}

\bigskip
\section{Surfaces in numerical experiments}
\label{sec:def_surfaces}
We define the torus, red blood cell surface, and bumpy sphere, together with their associated divergence-free target vector fields.

\smallskip
\textbf{1. Torus.}
Let $\Gamma_{\mathrm T}\subset\mathbb{R}^3$ be the torus parameterized by
\begin{equation*}
	\bx_{\mathrm T}(\theta,\lambda)
	=
	\bigl(
	(R+r\cos\lambda)\cos\theta,\,
	(R+r\cos\lambda)\sin\theta,\,
	r\sin\lambda
	\bigr)^\top,
	\quad
	\theta,\lambda\in[0,2\pi),
	\label{eq:torus_param}
\end{equation*}
where $R$ and $r$ denote the major and minor radii, respectively. 
We generate a tangential divergence-free vector field from the stream function
$s_3(\theta,\lambda)=\sin(m\theta-n\lambda)$, where
$m,n\in\mathbb{N}$. 
Here, we set $R=1$, $r=\frac{1}{2}$, $m=2$ and $n=3$.
\smallskip

\textbf{2. Red blood cell.}
The RBC surface $\Gamma_{\mathrm R}\subset\mathbb{R}^3$ is parameterized by
\begin{equation*}
	\begin{aligned}
		\Gamma_{\mathrm R} = \bigl\{ (x, y, z) \in \mathbb{R}^3 \mid
		& x = r_0 \cos\lambda \cos\theta, \;
		y = r_0 \sin\lambda \cos\theta, \\
		& z = \tfrac{1}{2} \sin\theta
		\bigl(c_0 + c_2 \cos^2\theta + c_4 \cos^4\theta\bigr)
		\bigr\},
	\end{aligned}
	\label{eq:rbc_param}
\end{equation*}
where $-\pi/2\leq\theta\leq\pi/2$, $-\pi\leq\lambda<\pi$,
$r_0=3.91/3.39$, $c_0=0.81/3.39$, $c_2=7.83/3.39$, and
$c_4=-4.39/3.39$. This biconcave surface differs substantially from both the sphere and the torus and therefore provides a geometrically more heterogeneous test case.

We define the scalar stream function
\begin{equation*}
	s_4(\theta,\lambda)
	=
	\cos(5\theta)\cos(\lambda)
	\label{eq:rbc_stream}
\end{equation*}
and the target field
$\bff_{\mathrm R}=\bL_{\mathrm R}s_4$,
where $\bL_{\mathrm R}$ denotes the surface curl operator on $\Gamma_{\mathrm R}$. 
\smallskip

\textbf{3. Bumpy sphere.}
The bumpy sphere $\Gamma_{\mathrm B}\subset\mathbb{R}^3$ is defined as a radial perturbation of the unit sphere:
\begin{equation*}
	\mathbf{x}_{\mathrm B}(\theta,\lambda)
	=
	\rho(\theta,\lambda)
	\begin{pmatrix}
		\sin\lambda\cos\theta\\
		\sin\lambda\sin\theta\\
		\cos\lambda
	\end{pmatrix},
	\quad
	\theta\in[0,2\pi),
	\quad
	\lambda\in[0,\pi],
	\label{eq:bumpy_param}
\end{equation*}
where the radial function is given by
\begin{equation*}
	\rho(\theta,\lambda)
	=
	R_{\mathrm{base}}
	+
	a\left[
	\sin(\omega x)
	+
	\sin(\omega y)
	+
	\sin(\omega z)
	\right],
	\label{eq:bumpy_rho}
\end{equation*}
with
$(x,y,z)
=
(\sin\lambda\cos\theta,\,
\sin\lambda\sin\theta,\,
\cos\lambda)$. We take
$R_{\mathrm{base}}=1,a=0.05$ and $\omega=20$.
The stream function is chosen as
\begin{equation*}
	s_5(x,y,z)
	=
	xy+yz+zx,
	\label{eq:bumpy_stream}
\end{equation*}
and the target divergence-free vector field is given by
	$\bff_{\mathrm B}
	=
	\bL_{\mathrm B}
	s_5$.

\end{document}